\documentclass[reqno, 11pt]{amsart}

\usepackage{amscd}
\usepackage{amssymb}
\usepackage{mathrsfs}
\usepackage[all,cmtip]{xy}
\usepackage{geometry}
\usepackage{longtable}
\usepackage{dsfont}
\usepackage{tikz}
\usepackage{hyperref}
\usepackage{xcolor}

\numberwithin{equation}{subsection}
\newtheorem{thm}[equation]{Theorem}
\newtheorem*{thm*}{Theorem}
\newtheorem{cor}[equation]{Corollary}
\newtheorem*{cor*}{Corollary}

\newtheorem{lemma}[equation]{Lemma}

\newtheorem{prop}[equation]{Proposition}

\newtheorem{thmA}{Theorem}

\theoremstyle{definition}
\newtheorem{definition}[equation]{Definition}
\theoremstyle{remark}
\newtheorem{remark}[equation]{Remark}
\newtheorem{example}[equation]{Example}

\theoremstyle{definition}

\usepackage[utf8]{inputenc}
\usepackage{tikz-cd}
\usetikzlibrary{calc} 
\usetikzlibrary{arrows} 
\usetikzlibrary {arrows.meta}
\usetikzlibrary{decorations.pathreplacing} 
\usepackage{float}
\usepackage{amsmath}
\usepackage{diagrams-1} 

\usepackage{lipsum}
\usetikzlibrary{calc}
\usepackage{mathtools,leftindex,tensor}
\usepackage[version=4]{mhchem}
\usetikzlibrary{decorations.pathreplacing,calligraphy}
\usetikzlibrary{decorations.markings}
\usepackage{ulem}

\def\End{\operatorname{End}}

\def\ker{\operatorname{ker}}

\def\im{\operatorname{im}}

\def\dim{\operatorname{dim}}

\def\id{\operatorname{id}}

\def\Ind{\operatorname{Ind}}

\def\Hom{\operatorname{Hom}}
\def\sgn{\operatorname{sgn}}

\def\C{\mathbb{C}}

\def\R{\mathbb{R}}
\def\N{\mathbb{N}}
\def\Z{\mathbb{Z}}

\def\F{\mathbb{F}}

\def\OO{\mathcal{O}}
\def\KK{\mathcal{K}}

\def\HH{\mathcal{H}}

\def\FF{\mathcal{F}}
\def\GG{\mathcal{G}}

\def\NN{\mathcal{N}}
\def\PP{\mathcal{P}}
\def\SS{\mathcal{S}}
\def\CC{\mathcal{C}}

\def\As{\mathscr{A}}
\def\Bs{\mathscr{B}}
\def\Cs{\mathscr{C}}
\def\Ds{\mathscr{D}}

\def\g{\mathfrak{g}}
\def\h{\mathfrak{h}}

\def\l{\mathfrak{l}}
\def\s{\mathfrak{s}}

\def\ol{\overline}
\def\mod{\operatorname{mod}}
\def\tr{\text{tr}}

\def\Spec{\text{Spec}}

\def\Mod{\text{mod}}
\def\Ver{\operatorname{Ver}}
\def\sub{\subseteq}

\def\t{\tilde}

\def\Rep{\operatorname{Rep}}

\def\verp{\operatorname{Ver}_p}
\def\rep{\operatorname{Rep}_\Bbbk}
\def\vec{\operatorname{vec}_\Bbbk}
\def\svec{\operatorname{svec}_\Bbbk}
\newcommand{\op}[1]{\operatorname{#1}}
\newcommand{\tab}{\operatorname{Tab}}
\newcommand{\Tak}{\mathcal{T}ak}
\newcommand{\AbCat}{\mathcal{A}b\mathcal{C}at}
\newcommand{\comod}{\mathrm{comod}}

\usepackage{comment}

\newcommand{\trun}{\operatorname{tr}}
\newcommand{\st}{\operatorname{st}}
\newcommand{\Tilt}{\operatorname{Tilt}}
\newcommand{\Young}{\operatorname{Young}}
\def\De{\Delta}
\def\Na{\nabla}
\newcommand{\Filt}{\operatorname{Filt}\nolimits}
\newcommand{\inj}{\operatorname{inj}\nolimits}
\newcommand{\proj}{\operatorname{proj}\nolimits}
\def\D{\mathbf D}
\def\K{\mathbf K}
\def\L{\mathbf L}
\def\R{\mathbf R}

\def\xto{\xrightarrow}

\usepackage{relsize}
\usepackage{fancyhdr}
\usepackage[all,cmtip]{xy}

\usepackage{ulem}

\begin{document}
\normalem

\title{Semisimplifying categorical Heisenberg actions and periodic equivalences}

\author{Chris Hone}
\address{C.~Hone: School of Mathematics and Statistics, University of Sydney, Australia}
\email{christopher.t.hone@gmail.com}
\author{Finn Klein}
\address{F.~Klein: School of Mathematics and Statistics, University of Sydney, Australia}
\email{finn.klein@sydney.edu.au}
\author{Bregje Pauwels}
\address{B.~Pauwels: School of Mathematical and Physical Sciences, Macquarie University, Australia}
\email{bregje.pauwels@mq.edu.au}
\author{Alexander Sherman}
\address{A.~Sherman: School of Mathematics and Statistics, University of New South Wales, Australia}
\email{alex.sherman@unsw.edu.au}
\author{Oded Yacobi}
\address{O.~Yacobi: School of Mathematics and Statistics, University of Sydney, Australia}
\email{oded.yacobi@sydney.edu.au}
\author{Victor L. Zhang}
\address{V.L.~Zhang: School of Mathematics and Statistics, University of New South Wales, Australia}
\email{victor.l.zhang@unsw.edu.au}

\begin{abstract}
We systematically apply semisimplification functors in modular representation theory. Motivated by the Duflo--Serganova functor in Lie superalgebras, we construct various functors of interest. In the setting of finite groups, we refine the cyclic group Brauer construction and categorify the Glauberman correspondence.  In the setting of degenerate categorical Heisenberg actions, we obtain a rich collection of functors which commute with the categorical action.  Applied to well-known categorifications of the basic representation and Fock space, our functors give explicit realizations of periodic equivalences for  polynomial functors and symmetric groups first studied by Henke--Koenig. In fact, we globalize the equivalences of Henke--Koenig by symmetric monoidal functors.  We apply these results to deduce branching properties of certain modular representations of~$S_n$.  
\end{abstract}

\maketitle
\pagestyle{plain}


\section{Introduction}\label{section:introduction}

Representation theory in characteristic $p>0$ is incredibly rich and complicated.  For a finite group $G$, representations need not be semisimple, and many approaches have been developed to study the homological properties of $\rep G$.  Local representation theory and support theory have been particularly successful and continue to attract widespread interest (\cite{Alperin}, \cite{Evens}).  In a separate vein, computing the decomposition numbers (or even the dimensions of irreducible representations) of the symmetric group remains one of the largest open problems in modular representation theory (\cite{James2}, \cite{Williamson}), and has important connections to the modular representation theory of the algebraic group $GL_n$.

In this paper we introduce \emph{One Tree Island (OTI) functors} for studying $\rep G$ and categorical Heisenberg actions in characteristic $p$.  Such functors were already introduced and studied in \cite{BFRSW} in the setting of algebraic groups, where a (conjectural) connection was found with the Finkelberg--Mirkovic (FM) conjecture.  The themes of \cite{BFRSW} and this paper are that OTI functors:
\begin{enumerate}
    \item have explicit, computable definitions,
    \item are compatible with categorical structures (e.g.~they are symmetric monoidal and/or define morphisms of categorical actions),
    \item are non-exact but define maps on Grothendieck groups mod $p$,
    \item naturally detect the cohomological support of a module,
    \item behave well on generating objects such as permutation modules and tilting modules,
    \item categorify a central element of the enveloping algebra of $\widehat{\mathfrak{sl}}_p$ modulo $p$,
    \item and provide extensions of known functors of interest (hypercohomology composed with FM equivalence in the setting of \cite{BFRSW}, and periodic equivalences of symmetric group representations in the setting of this paper). 
\end{enumerate}

Using OTI functors we obtain explicit consequences for representations of symmetric groups, which we believe further motivate the study of OTI functors in modular representation theory and categorical Heisenberg actions.

The remainder of the introduction will provide a brief description of the OTI functor and the main results of this paper.  We will also discuss how this work is motivated by the Duflo--Serganova functor from the theory of Lie superalgebras.

\subsection{What is an OTI functor? The case of finite groups}\label{section intro fin gps} Let $G$ be a finite group such that $p \operatorname{\big|} |G|$.  Let $\Bbbk$ be an algebraically closed field of characteristic $p$, and let $\rep G$ be the category of finite dimensional representations of $G$ over $\Bbbk$.

By Cauchy's theorem, there exists a cyclic subgroup $H\sub G$ with $H\cong C_p$.  One may then consider the functor $\Phi_H$ defined by the following commutative diagram:
\[
\xymatrix{ 
\rep G \ar[rr]^{\operatorname{Res}_H}\ar[drr]_{\Phi_H} && \rep C_p \ar[d]^{ss} \\
&& \operatorname{Ver}_p}
\]
Here, $\operatorname{Ver}_p$ is a semisimple, symmetric tensor category defined as the \textit{semisimplification} of the tensor category $\operatorname{Rep}C_p$. As an abelian category, this category is just $\prod_{i=1}^{p-1}\vec$.
The semisimplication functor $ss$ is symmetric monoidal, but not exact, and in this case admits an explicit, computable description given in Section \ref{section oti explicit}.

Thus the functor $\Phi_H$ is defined by restricting, then semisimplifying.  The centralizer subgroup $C_G(H)$ naturally acts on $\Phi_H(M)$ for any module $M$, and $H$ acts trivially.  Hence we obtain a functor
\[
\Phi_H:\rep G\to \Rep_{\Ver_p} C_G(H)/H,
\]
where $\Rep_{\Ver_p} C_G(H)/H$ denotes the representations of $C_G(H)/H$ in $\Ver_p$.  The functor $\Phi_H$ is our first example of an OTI functor, and it satisfies many desirable properties collected in Lemma \ref{lemma properties of OTI finite groups}.

As an application in the finite groups setting, we consider the situation where we have a group $H$ of order $p$ acting on $G$. The Glauberman correspondence is a canonical bijection between the irreducible representations of $G$ fixed by $H$ and the irreducible representations of the invariants $G^H$ (Theorem \ref{Glauberman corresp}). By Clifford theory, an irreducible representation of $G$ that is fixed by $H$ extends uniquely to an irreducible representation of $G \rtimes H$. 
Observing that $C_{G\rtimes H}(H)/H \cong G^H$, we obtain an OTI functor:
\begin{align*}
\Phi_H:\rep (G\rtimes H) \to \Rep_{\Ver_p}(G^H).
\end{align*}
\begin{thm*}
    The OTI functor $\Phi_H$ above categorifies the Glauberman correspondence. 
\end{thm*}

We refer to Theorem \ref{OTIGlauberman} for a precise statement and proof. Note that one novel aspect of this categorification is that it explains the sign that appears in the Glauberman correspondence, see~\ref{sec:Glauberman} for more details.

\subsection{Periodic equivalences for symmetric groups}  It was first observed in~\cite{HK} (see also~\cite{MW}, and~\cite{Har} for a more recent treatment) that there is an equivalence between certain full abelian subcategories of $\rep S_{n}$ and $\rep S_{n-p^r}$, corresponding to truncating the first row of a partition by $p^r$.  

For the purposes of this introduction, we say a partition $\lambda$ of $n$ is $p^r$-stable if $\lambda_1\gg p^r$ and $\sum\limits_{i>1}\lambda_i<p^r$.  Let $\mathcal{S}_{n,p^r}$ be the smallest abelian subcategory of $\rep S_n$ containing all permutation modules $M^{\lambda}$, where $\lambda$ is $p^r$-stable.  The category $\mathcal{S}_{n,p^r}$ contains all Specht modules for $p^r$-stable partitions and all simple modules for $p^r$-stable, $p$-regular partitions.  Then~\cite{Har, HK} produce equivalences $\mathcal{S}_{n,p^r}\xto{\sim}\mathcal{S}_{n-p^r,p^r}$, which on simple, Specht and permutation modules correspond to taking the partition $\lambda=(\lambda_1,\lambda_2,\dots)$ to $(\lambda_1-p^r,\lambda_2,\dots)$. We call these ``periodic equivalences''.

The equivalences of~\cite{Har, HK} have useful applications in computing decomposition numbers and constructing Deligne interpolation categories, but they are inexplicit and difficult to compute directly.  Both functors are only defined on subcategories of $\operatorname{Rep}S_n$, and require passage through Schur--Weyl duality amongst other operations to define them.  

Our first theorem improves upon~\cite{Har, HK} by providing explicit, global functors $\rep S_n\to\rep S_{n-p^r}$ which restrict to the periodic equivalences.
Furthermore, our functors commute with the well-known categorical action of $\widehat{\mathfrak{sl}_p}$ on $\bigoplus_n\rep S_n$.

To state our theorem, let $S_{p^r}\sub S_n$ denote the natural subgroup determined by the permutations of $\{n-p^r+1,\dots,n\}$.  Then $S_{p^r}$ contains a transitive, elementary abelian subgroup $A$ of order $p^r$. Write $\sigma_1,\dots,\sigma_r$ for generators of $A$. Let $a_1,\dots,a_r\in\Bbbk$ be linearly independent over $\mathbb{F}_p$, and  set 
\[
z_{\mathbf{a}}=a_1(\sigma_1-1)+\dots+a_r(\sigma_r-1)\in\Bbbk A.
\]
Note that $z_{\mathbf{a}}$ defines an operator on any $V \in \rep S_n$, 
which commutes with $S_{n-p^r}$. Hence $\ker(z_{\mathbf{a}}), \im(z_{\mathbf{a}})$, etc. can be viewed as functors from $\rep S_{n}\to\rep S_{n-p^r}$.
The following is a direct corollary of~Theorem~\ref{thm periodicity for scs}.

\begin{thmA}\label{thm A intro}
    The  functors $\rep S_{n}\to\rep S_{n-p^r}$ given by
    \[
    \frac{\ker(z_{\mathbf{a}})}{\ker(z_{\mathbf{a}})\cap\im(z_{\mathbf{a}})} \ \text{ and } \ \frac{\ker(z_{\mathbf{a}})}{\im(z_{\mathbf{a}}^{p-1})}
    \]
    have the following properties:
    \begin{enumerate}
        \item they restrict to the periodic equivalences $\mathcal{S}_{n,p^r}\xto{\sim}\mathcal{S}_{n-p^r,p^r}$, and
        \item they commute with the $\widehat{\mathfrak{sl}_p}$ categorical action.
    \end{enumerate}
\end{thmA} 

Theorem \ref{thm A intro} provides many distinct functors which globalize our equivalence of interest: we can choose any generic tuple $\mathbf{a}$ and  either functor given in the  statement.  
In fact,~Theorem~\ref{thm periodicity for scs} exhibits many more global functors that satisfy (1) and (2), some of which are symmetric monoidal. 
In the next section we will generalise the construction of OTI functors, and the functors above will be specific examples of this general construction.

\subsection{An OTI functor from categorical actions} 
Let us distill the definition of the OTI functor from the finite group case as follows.  We may replace $\rep G$ and $\rep H$ by (nice) abelian categories $\Cs,\Ds$, and replace $\operatorname{Res}_H:\rep G\to\rep H$ by a functor $R:\Cs\to\Ds$ with an action of $C_p$. Then we obtain a functor
\[
\xymatrix{ 
\Cs \ar[rr]^-{R}\ar[drr]_-{\Phi} && \Ds\boxtimes \rep C_p \ar[d]^{ss} \\
&& \Ds\boxtimes \operatorname{Ver}_p.}
\]
Here, $\Ds\boxtimes \rep C_p$ is the category of objects in $\Ds$ with a $C_p$-action, and $\Ds\boxtimes\operatorname{Ver}_p$ is equivalent to $p-1$ copies of $\Ds$.  Note that the Frobenius functor is of the above form, where $R(M)=M^{\otimes p}$.

Our primary case of interest comes from the categorical action of the Heisenberg algebra.
Recall that a degenerate categorical Heisenberg action on $\Cs$ is the data of a pair of biadjoint functors $E,F:\Cs\to\Cs$ and natural transformations $x:E \to E$ and $T: E^2 \to E^2$ satisfying certain compatibilities.  One such requirement is that $T$ generates an action of $S_n$ on $E^n$, which in particular provides an action of $C_p$ on $E^p$. This provides the functor $R=E^p$ above.

We also generalise the semisimplification functor $ss:\operatorname{Rep}C_p\to\Ver_p$ to any $\Bbbk$-linear functor $\varphi:\rep S_{p^r}\to\As$ such that $\varphi\circ\Ind_{H}^{S_{p^r}}=0$ for all non-transitive subgroups $H\sub S_{p^r}$, see~Section~\ref{section cf functors}.
We call such a functor $\varphi$ a \textit{CF functor}, a modest generalisation of the notion of $\mathbf{V}$-functor introduced in~\cite{CF}.  CF functors satisfy the minimal conditions necessary for our machinery to behave well.

Now given a category $\Cs$ with a degenerate Heisenberg action and a CF functor $\varphi$, we define the OTI functor $\Phi_{\varphi}:\Cs\to\Cs\boxtimes\As$ by the following (now familiar) diagram:
\[
\xymatrix{
\Cs \ar[rr]^{E^{p^r}} \ar[drr]_{\Phi_{\varphi}} && \Cs\boxtimes\rep S_{p^r} \ar[d]^{1\boxtimes\varphi}\\
&& \Cs\boxtimes\As
}
\]
There are many interesting, well-studied examples of degenerate categorical Heisenberg actions,
the most famous being on the category $\mathcal{S}ym:=\bigoplus\limits_{n\in\N}\rep S_n$.  In this case, $\Phi_{\varphi}$ will be symmetric monoidal whenever $\varphi$ is.  For example, if $p^r=p$ and $\varphi:\rep S_p\to \Ver_p$ is given by restriction to $C_p$ and semisimplification, we obtain an OTI functor $\rep S_n\to\rep S_{n-p}\boxtimes \Ver_p$ (see Example \ref{example fin gps S_n}).  On the other hand, the functors introduced in the statement of Theorem \ref{thm A intro} are also examples of OTI functors.  

To state our next theorem, note that $\Cs\boxtimes\As$ inherits a degenerate categorical Heisenberg action from $\Cs$.  Furthermore, recall that a degenerate categorical Heisenberg action on $\Cs$ induces an action of an integral form of $\widehat{\mathfrak{sl}_p}$ on $K_0(\Cs)$, where $e:=[E]$ is the sum of positive Chevalley generators.  

\begin{thmA}\label{thm intro B}
Let $\varphi$ be a CF functor.
\begin{enumerate}
    \item The OTI functor $\Phi_{\varphi}:\Cs\to\Cs\boxtimes\As$ defines a morphism of degenerate categorical Heisenberg actions.
    \item If $\varphi:\rep S_p\to\Ver_p$ is given by restriction to $C_p$ followed by semisimplification, then $\Phi_{\varphi}$ induces an endomorphism of the mod $p$ Grothendieck group $K_0(\Cs)/(p)$. This endomorphism is given by the central element $e^p$ in the mod $p$ enveloping algebra of $\widehat{\mathfrak{sl}_p}$.
\end{enumerate}
\end{thmA}

The complete statement of this theorem appears in~Sections~\ref{sec:OTIcommute} and~\ref{ss:mod_p_groth_ring}. The proof of~(1) uses a diagrammatic framework and appears in~Section~\ref{section:appendix_proof_of_categorical_commutation}. 

\subsection{Strengthening Theorem \ref{thm A intro}}  
Let $\PP ol=\bigoplus\limits_{n\geq 0}\PP ol_n$ be Friedlander and Suslin's category of strict polynomial functors (see~Section~\ref{section:polynomial_functors}).  There is a degenerate categorical Heisenberg action on $\PP ol$, and the Schur--Weyl duality functor $\mathcal{F}:\PP ol\to\mathcal{S}ym$ is a morphism of categorical actions.  Because the OTI functor is defined purely in terms of the categorical action, we obtain a commutative diagram for any CF functor $\varphi$:
\[
\xymatrix{
\PP ol_n\ar[rr]^{\mathcal{F}}\ar[d]_{\Phi_{\varphi}} && \rep S_n\ar[d]^{\Phi_{\varphi}} \\
\PP ol_{n-p^r}\boxtimes\As \ar[rr]^{\mathcal{F}\boxtimes 1} && \rep S_n\boxtimes\As.
}
\]
This allows us to work with the category of polynomial functors, which is a highest weight category, to deduce results about OTI functors on $\mathcal{S}ym$.  The following is a special case of Theorem~\ref{thm proof from intro ss}:

\begin{thmA}\label{thm intro C}
    Consider a subgroup $H$ with $A\sub H\sub S_{p^r}$ and let $\varphi:\rep H\to(\rep H)^{ss}$ be the semisimplification functor.
    The OTI functor $\Phi_{\varphi}:\rep S_n\to\rep S_{n-p^r}\boxtimes(\rep H)^{ss}$ is a symmetric monoidal functor which restricts to the  periodic equivalence $\mathcal{S}_{n,p^r}\xto{\sim}\mathcal{S}_{n-p^r,p^r}$.  
\end{thmA}

Modular branching results are extremely difficult, with one of the most famous results being Kleshchev's description of branching laws for socles of simple modules \cite{Kl}.  Theorem~\ref{thm intro C} can be used to prove a branching result of a different flavour, that is more in the spirit of Cor.~1.1 of \cite{BFRSW}. Indeed,~Corollary~\ref{cor proof from intro} says that up to negligible summands, any object in $\SS_{n,p^r}$ restricts to $H$ as a trivial module:

\begin{cor*}
    Suppose that $V\in\mathcal{S}_{n,p^r}$.  Then we have a decomposition 
    \[
    V|_{H}=\mathbf{1}^{\oplus \ell} \oplus \bigoplus\limits_i N_{i},
    \]
    where each $N_i$ is indecomposable and $p \mid \dim(N_i)$ for all $i$. 
\end{cor*}

\subsection{Motivation from Lie superalgebras}

Representations of Lie superalgebras over the complex numbers are generally not semisimple, even for the general linear Lie superalgebra.  One of the most powerful tools for their study is the \emph{Duflo--Serganova (DS)} functor (originally introduced in \cite{DS}, although see \cite{GHSS} for a more recent survey).  The DS functor has been used to describe central characters, understand certain tensor products, construct abelian envelopes of Deligne interpolating categories, and compute superdimensions of irreducible representations (which is the super-analogue of taking dimension mod $p$). The Theorem \ref{thm A intro} can be viewed as a modular analogue of the equivalence induced by the $DS$ functor in \cite{EntovaHinichSerganova}.

To make the connection to the OTI functor clear, consider a (quasireductive) Lie superalgebra $\mathfrak{g}$. Then $\mathfrak{g}$ has non-semisimple representation theory if and only if it contains a one-dimensional, odd abelian subalgebra $\h=\mathbb{C}\langle x\rangle$ (up to semisimple central extension).   One obtains the $DS$ functor as follows:
\[
\xymatrix{
\Rep\g\ar[rr]^{\operatorname{Res}_{\h}} \ar[drr]_{DS_x} && \Rep\h\ar[d]^{ss}\\
&& \operatorname{svec}.
}
\]
In this case, the semisimplification functor $ss$ takes the particularly simple form $\frac{\operatorname{ker}(x)}{\operatorname{\ker}(x)\cap\im(x)}$ (compare with the formula for OTI functor in Section \ref{section oti explicit}).  
We remark that the OTI functor exhibits some of the remarkable behaviour previously observed for $DS$ functors, along with new phenomena.

\subsection{Roadmap} In Section \ref{section:preliminaries} we set up preliminary notation and machinery for Deligne tensor products and semisimplification which will be used throughout much of the paper. In Section \ref{sec verp} we recall the Verlinde category $\operatorname{Ver}_p$ and an explicit formula for the semisimplification functor of $\rep C_p$.  Section \ref{section:oti_for_finite_groups} introduces the OTI functor for finite groups, proves its basic properties, and shows that it categorifies the Glauberman correspondence.  

In Section \ref{section:oti_functors_on_heisenberg_categorifications} we shift to the categorical setting, defining degenerate categorical Heisenberg actions, with important examples.  We introduce the OTI functor in this setting, and precisely state~Theorem~\ref{thm intro B}(1). Sections \ref{section:symmetric_groups} and \ref{section:polynomial_functors} contain necessary background on representations of symmetric groups and polynomial functors, and study how the OTI functor behaves on special classes of modules.  Schur--Weyl duality, which connects polynomial functors to representations of symmetric groups, is proven to be compatible with the OTI functor.  Our work culminates with Section \ref{section:periodicity_result} where we prove Theorems \ref{thm A intro} and \ref{thm intro C} using all the machinery developed prior.  Finally, in Section \ref{section:appendix_proof_of_categorical_commutation} we introduce diagrammatics for the OTI functor which provide a streamlined proof of Theorem \ref{thm intro B}.

\subsection{Acknowledgements}  We thank Joe Baine, Tasman Fell, Anna Romanov, and Geordie Williamson, whose support and consultation made this project  possible.  We thank Nate Harman for explaining his work \cite{Har} and suggesting an idea which led to Theorem \ref{thm intro C}.  We further thank Kevin Coulembier, Eoghan McDowell, and Vera Serganova for many inspiring conversations.  We acknowledge Magma for its assistance in computations toward the start of this project. We would also like to thank Ruby, Heinrich and Lukas from the University of
Sydney’s One Tree Island Research Station for a stimulating research environment and a suitable name for our functor.  This work was supported by the ARC grant DP230100654. A.S. was also partially supported by ARC grant DP210100251 and an AMS-Simons travel grant.  B.P. was supported by a 2024 Macquarie University FSE Strategic Start-up Grant. V.L.Z was supported by a UNSW FRTG2024 grant.

\section{Preliminaries}\label{section:preliminaries}

Let $p$ be a prime.  We work throughout over an algebraically closed field $\Bbbk$ of characteristic $p>0$.  In particular, we assume that all categories are $\Bbbk$-linear.  We will work exclusively with finite-dimensional representations, writing $\vec$ for the category of finite dimensional vector spaces, and $\rep G$ for the category of finite-dimensional representations of $G$.

\subsection{Locally finite categories}\label{sec ab cat}  
Let $\Cs$ be an essentially small abelian category.  We say that $\Cs$ is \emph{locally finite} if every object of $\Cs$ has finite length and hom spaces are finite-dimensional.  By a theorem of Takeuchi~(\cite[Theorem 1.9.15]{EGNO}), $\Cs$ is locally finite if and only if there exists a $\Bbbk$-coalgebra $C$ such that $\Cs$ is equivalent to the category of finite-dimensional right comodules over $C$, written $C\operatorname{--}\comod$.

We denote by $\AbCat$ the 2-category of essentially small $\Bbbk$-linear abelian categories and we write $\Tak$ for the 2-subcategory of locally finite abelian categories. The subcategory $\Tak^{\text{ex}}\subset \Tak$ has the same objects, but the 1-morphisms are exact functors. 

\subsection{Deligne tensor product}\label{deligne tensor} Let $\Bs$ and $\Cs$ be essentially small $\Bbbk$-linear abelian categories. In \cite[\S 5.1]{Del90}, Deligne defines their tensor product $\Bs\boxtimes\Cs$, when it exists, to be a $\Bbbk$-linear abelian category admitting a bilinear bifunctor 
\[
\Bs\times\Cs\to\Bs\boxtimes\Cs, \ \ (X,Y)\mapsto X\boxtimes Y,
\]
which is universal in the following sense: for any $\Bbbk$-linear abelian category $\Ds$, restriction yields an equivalence between the category of right exact $\Bbbk$-linear functors $\Bs\boxtimes\Cs\to \Ds$ and the category of bilinear bifunctors $\Bs\times\Cs\to\Ds$ which are right exact in $\Bs$ and $\Cs$.  

Now suppose $\Cs\simeq C\operatorname{--}\comod$ for some $\Bbbk$-coalgebra $C$. Following~\cite{CF}, we define the category $C\operatorname{--}\comod_{\Bs}$ as the category of objects $X$ in $\Bs$ together with an algebra morphism $C^*_X\to\End(X)$
for some finite-dimensional subcoalgebra $C_X\subset C$.
There is a bilinear bifunctor
$\Cs\times\Bs\to C\operatorname{--}\comod_{\Bs}$
and~\cite[Lemma 3.2.3]{CF} shows we get an equivalence
$$\Cs\boxtimes\Bs\xrightarrow{\simeq} C-\operatorname{comod}_{\Bs}.$$
Using this equivalence, \cite[Proposition 3.2.6]{CF} further constructs a pseudo-functor
$$-\boxtimes -: Tak^{\text{ex}}\times\AbCat\to \AbCat$$
such that $(F\boxtimes \varphi)(X\boxtimes Y)\simeq F(X)\boxtimes \varphi(Y)$ for any 1-morphism $F$ in $\Tak^{\text{ex}}$ and 1-morphism $\varphi$ in $\AbCat$.  The following are immediate consequences.

\begin{lemma}\label{lemma additive functors extend}  
Let $\Cs$ be a locally finite abelian category.
    \begin{enumerate}
        \item The pseudo-functor $\Cs\boxtimes-:\AbCat\to \AbCat$ preserves adjunctions and exactness properties of functors (1-morphisms) in $\AbCat$.
        \item There is a 2-isomorphism in $\AbCat$
        \[(F\boxtimes 1)\circ(1\boxtimes\varphi)\simeq(1\boxtimes\varphi)\circ(F\boxtimes1)\]
        natural for any 1-morphism $F$ in $\Tak^{\text{ex}}$ and 1-morphism $\varphi$ in $\AbCat$.
    \end{enumerate}
\end{lemma}

\begin{remark}\label{k0}
Let us write $K_0(\mathscr{C})$ for the Grothendieck group of an abelian category $\Cs$.  Recall that for locally finite abelian categories $\Bs,\Cs$, we have a natural isomorphism
\begin{equation*}
K_0(\Bs\boxtimes\Cs)\cong K_0(\Bs)\otimes_{\Z}K_0(\Cs).
\end{equation*}
Indeed, all simple objects of $\Bs\boxtimes\Cs$ are of the form $L\boxtimes L'$ for simple objects $L$ of $\Bs$, and $L'$ of $\Cs$.  Further, $L_1\boxtimes L_1'\cong L_2\boxtimes L_2'$ if and only if $L_1\cong L_2$ and $L_1'\cong L_2'$.
\end{remark}

\begin{remark}
There is another class of abelian categories which play an important role in categorical actions, namely those which are \emph{Schurian}. Schurian abelian categories are in some sense dual to locally finite abelian categories, with the main technical difference being that their objects are not all finitely presented (equivalently, compact).  To work with them in this paper we would be forced to treat them as a separate case, which seems both undesirable and possibly unmotivated at this stage of our work. Thus we do not consider Schurian categories in our setting.  However our understanding is that much of our work carries over to that setting. 
\end{remark}

 If $\mathscr{E}_1\sub \Cs$ and $\mathscr{E}_2\sub \Ds$ are additive subcategories of abelian categories $\Cs,\,\Ds$, we will at times abuse notation and write $\mathscr{E}_1\boxtimes\mathscr{E}_2$ for the smallest full additive subcategory of $\Cs\boxtimes \Ds$ containing the objects $X_1\boxtimes X_2$ with $X_i\in\mathscr{E}_i$.
\subsection{Symmetric tensor categories} An essentially small, $\Bbbk$-linear symmetric monoidal category $(\mathscr{C},\otimes,\mathbf{1})$ will be called a \emph{symmetric tensor category} (over $\Bbbk$) if
    \begin{enumerate}
        \item $\mathscr{C}$ is abelian,
        \item $\Bbbk\xto{\sim}\End(\mathbf{1})$,
        \item $(\mathscr{C},\otimes,\mathbf{1})$ is rigid,
        \item each object of $\mathscr{C}$ is of finite length (in particular $\Cs$ is artinian).
    \end{enumerate}
    In particular, it follows that $\mathscr{C}$ is locally finite.  Given an object $X$ in a tensor category $\mathscr{A}$, we write $X^*$ for its dual.
    It is well known (see~\cite[\S 5]{Del90}) that if $\Bs$ and $\Cs$ are symmetric tensor categories over an algebraically closed field,  then $\Bs\boxtimes\Cs$ exists and will once again be a symmetric tensor category.

\begin{example}
    The two most significant symmetric tensor categories in this paper are:  
    \begin{enumerate}
        \item $\operatorname{Rep}_{\Bbbk}(G)$, where $G$ is a finite group.  This is the category of $G$-modules, equivalently modules over the group algebra $\Bbbk G$, or equivalently comodules over the coordinate algebra $\Bbbk[G]:=(\Bbbk G)^*$.  

        \item The second category is the Verlinde category $\operatorname{Ver}_p$, which will be explained in~Section~\ref{sec verp}.
    \end{enumerate}
\end{example}

\subsection{Semisimplification}\label{eqn trace in tens cat}  We now recall the notion of semisimplication, which was originally introduced in \cite{BW}.  Our main reference will be \cite{EO2}.  

Let $\Cs$ be a symmetric tensor category.  Recall that the \textit{trace} $\tr(f)\in\Bbbk=\End(\mathbf{1})$ of a morphism $f:X\to X$, is given by  
\begin{equation*}
\mathbf{1}\to X\otimes X^*\xto{f\otimes 1}X\otimes X^*\xto{\sim}X^*\otimes X\to \mathbf{1},   
\end{equation*}
where the first map is coevaluation and the last map is evaluation.  

The \textit{ideal of negligible morphisms} $\NN$ of $\Cs$ consists of all morphisms $f:X\to Y$ such that for any morphism $g:Y\to X$, we have $\tr(fg)=0$.
It is well known that for an indecomposable object $X$, the identity morphism $\operatorname{id}_{X}$ is negligible if and only if~$\operatorname{tr}(\operatorname{id}_{X})=~0$.

By \cite[Lem.~2.3]{EO2}, $\NN$ forms a tensor ideal, so the quotient category $\Cs^{ss}:=\CC/\NN$ is a symmetric monoidal category.  We call $\Cs^{ss}$ the \emph{semisimplification} of $\Cs$, and write 
$$\varphi_{\Cs}:\Cs\to\Cs^{ss}$$
for the canonical quotient functor.  By \cite[Prop.~2.4]{EO2}, $\Cs^{ss}$ is a semisimple tensor category, whose simple objects are given by $\varphi_{\Cs}(M)$, where $M$ runs over all indecomposable objects of $\Cs$ with $\tr(\id_M)\neq0$. Note that $\varphi_{\Cs}$ is additive, but not exact in general. 

\begin{example}
    If $\mathscr{C}=\Rep_\Bbbk G$ for a finite group $G$ and a field $\Bbbk$ of characteristic $p$, we will write $\varphi_{G}$ for $\varphi_{\Cs}$.  
    For an indecomposable representation $V$, $\id_V$ is negligible if and only if $\dim V$ is a multiple of $p$. 
    Therefore $(\Rep_\Bbbk G)^{ss}$ is a semisimple category with simple objects given by $\varphi_G(V)$, where $V$ runs over all indecomposable $G$-modules of dimension not a multiple of $p$.

    By the main theorem of \cite{CEO}, there exists a faithful, exact tensor functor $(\rep G)^{ss}\to\Ver_p$.  This implies that objects in $(\rep G)^{ss}$ may be viewed as objects in $\Ver_p$ with additional structure (a linear action by an algebraic group). 
\end{example}

\section{The Verlinde Category}\label{sec verp} 

The Verlinde category $\Ver_p$ plays an important role in this paper, and in this section we recall some of the basics about this category and the linear algebraic construction of semisimplification in this setting (see~\cite{EO} and references therein).

\subsection{The Verlinde category}
We define the \emph{Verlinde category} as the semisimplification $\Ver_p := (\Rep_\Bbbk C_p)^{ss}$ (cf. Section \ref{eqn trace in tens cat}). This is a semisimple tensor category with simple objects given by $L_i := \varphi_{C_p}(M_i)$ for $1 \leq i \leq p-1$, where $\varphi_{C_p}$ denotes the semisimplification functor. 
Hence as an abelian category, this category is just $\prod_{i=1}^{p-1}\vec$. The decomposition of the tensor product of simple objects $L_i$ and $L_j$ is given by the ``Verlinde formula":
\[
L_i\otimes L_j=\bigoplus\limits_{k=1}^{\min(i,j,p-i,p-j)}L_{|i-j|+2k-1}.
\]
Note that $L_1$ and $L_{p-1}$ generate a copy of $\svec$ in $\Ver_p$.

\begin{example}
If $p=2$, $\Ver_p$ has one simple object, $L_1$, and it is equivalent to $\Vec$. If $p=3$, $\Ver_3$ is equivalent to the category of super vector spaces.  The first interesting example is $\Ver_5$, which has 4 simple objects.  We have 
    \[L_2\otimes L_4\cong L_3, \ \ \ L_3\otimes L_4\cong L_2, \ \ \ L_3^{\otimes 2}=L_3\oplus L_1.
    \] 
    In the Grothendieck ring of $\Ver_5$, this last equality becomes $[L_3]^2-[L_3]-[L_1]=0$ so there cannot exist an additive monoidal functor to vector spaces.
\end{example}

\begin{remark}
The Grothendieck ring of $\verp$ is isomorphic to $\mathbb{Z}[\zeta_p+\zeta_p^{-1}][x]/x^2$, with $[L_{p-1}]\mapsto x$ and $[L_2]\mapsto \zeta_p+\zeta_p^{-1}$.
This description reflects a decomposition $\verp\simeq\verp^+\boxtimes\svec$.
\end{remark}

\subsection{A linear algebraic functor}\label{section oti explicit}
Consider the group algebra $\Bbbk C_p=:R\cong \Bbbk[x]/x^p$. We view an object of $R\operatorname{--}\mod$ as a pair $(V,x)$ of a finite-dimensional $\Bbbk$-vector space $V$ and an endomorphism $x$ satisfying $x^p=0$. The indecomposable objects in $R\operatorname{--}\mod$ are all of the form 
\[
M_j:=R/x^j=\Bbbk[x]/x^j.
\]
The module $M_1$ is a one dimensional vector space with trivial $x$ action, and $M_p$ is a projective $R$ module. All modules $M_i$ are iterated extensions of $M_1$.

We define functors $\varphi_i:R\operatorname{--}\mod \to \Vec$ by the formula:
\begin{equation}\label{formulaOTI}
    \varphi_i(V):=\frac{\ker(x)\cap \im(x^{i-1})}{\ker(x)\cap \im(x^i)}.
\end{equation}
One may interpret $\varphi_i$ as  extracting the $M_i$ isotypic subspace inside $V$:

\begin{lemma}\label{prop OTI on indecomps}
The additive functor $\varphi_i$ applied to $M_j$ is zero unless $i=j$, and $\varphi_i(M_i)=\Bbbk$.
On morphisms, 
$\varphi_i:\End(M_i)\to \Bbbk$ is the (unique) $\Bbbk$-algebra morphism to the residue field~$\Bbbk$.
\end{lemma}

\begin{proof}
The kernel of $x$ in $M_j$ is always one dimensional, and lies in the image of $x^{i-1}$ if and only if $i\leq j$. If we have $i<j$, then this kernel lies in the image of $x^i$, so $\varphi_i$ is nonzero precisely on $M_i$. The description of $\varphi_i$ on morphisms is uniquely determined by the fact that $\varphi_i$ is $\Bbbk$-linear.
\end{proof}

\subsection{The semisimplification functor}\label{ss components}

We can describe the semsimplification functor $\varphi_{C_p}:\Rep_\Bbbk(C_p) \to \Ver_p$ using the functors $\varphi_i$. 
To do so, endow $R\operatorname{--}\mod$ and $\prod_{i=1}^{p-1}\vec$ with symmetric tensor structures using the equivalences
\begin{align*}
    R\operatorname{--}\Mod \cong \Rep_\Bbbk(C_p)\quad \text{ and }\quad \prod_{i=1}^{p-1}\vec \cong \Ver_p.
\end{align*}
Then, define $\varphi:R\operatorname{--}\mod \to \prod_{i=1}^{p-1}\vec$  by 
\begin{equation}\label{eqn varphi}
V\mapsto \varphi(V):=(\varphi_1(V),\varphi_2(V),...,\varphi_{p-1}(V))   
\end{equation}

The following is proven in \cite[Prop.~4.8]{EO2}.
\begin{prop}
The functor $\varphi$ is equivalent to the semisimplification functor $\varphi_{C_p}$. In particular, $\varphi$ is monoidal.
\end{prop}

While $\varphi$ is not left or right exact, we have the following:
\begin{lemma}\label{lemma exactness property}
    Suppose that $0\to X\to Y\to Z\to 0$ is a short exact sequence of $R$-modules.  The following are equivalent:
    \begin{enumerate}
        \item $\varphi(X)\to\varphi(Y)$ is injective.
        \item $\varphi(Y)\to\varphi(Z)$ is surjective.
        \item $0\to\varphi(X)\to\varphi(Y)\to\varphi(Z)\to 0$ is exact.
        \item $0\to X\to Y\to Z\to 0$ is split as a short exact sequence of $R$-modules.
    \end{enumerate}
\end{lemma}

\begin{proof}
It is clear that (4) implies (3), which implies (1) and (2). Assuming (1), choose decompositions of $X$ and $Y$, say $X\cong \oplus_i M_i^{e_i}$ and $Y\cong \oplus_i M_i^{f_i}$. Then by injectivity of $\varphi(X)\to \varphi(Y)$, the induced map \[M_i^{e_i}\to X\to Y\to M_i^{f_i}\] is a split monomorphism. So we may define splittings on each summand of this decomposition of $Y$, giving a map $\pi:Y\to X$, whose composition $X\to Y\xto{\pi} X$ is equal to the identity modulo the radical of $\End(X)$. This composition is thus an isomorphism, so the inclusion of $X$ into $Y$ is split. An analogous argument shows that (2) implies (4).
\end{proof}

\begin{remark}
For an $R$-module $M$, the Tate cohomology of $M$ is defined by $$\ker(x)/\im(x^{p-1}).$$ This carries a filtration by the images of $x^i$, and the associated graded pieces of this filtration are our components $\varphi_i$.
This description allows for the general expression of exactness for the components $\varphi_i$: they are associated graded parts of a natural filtration on zero-th Tate cohomology for the action of $x$. In general, one may understand the exactness properties of $\varphi$ as arising from a filtration by images of $x$ on the cohomology of the two periodic complex \[M\xto{x}M\xto{x^{p-1}}M\to\dots \ .\]
\end{remark}

\section{OTI for finite groups}\label{section:oti_for_finite_groups}
In this section, we introduce the One Tree Island (OTI) functor, in the concrete case of representations of a finite group. 

\subsection{Representations in $\Ver_p$}

Let $G$ be a finite group. We write $\operatorname{Rep}_{\Ver_p}G$ for the category of $G$-modules in $\Ver_p$. The objects of this category are $\mathbf{V} \in \Ver_p$ equipped with an algebra map $\Bbbk G \to \End(\mathbf{V})$. Note that $\operatorname{Rep}_{\Ver_p}G$ is equivalent to the Deligne tensor product 
$$\Rep_{\Ver_p}G\cong\Rep_{\Bbbk}G\boxtimes\Ver_p,$$ 
hence every object $\mathbf{V}$ of $\Rep_{\Ver_p}G$ decomposes as 
\[
\mathbf{V}=(V_1\boxtimes L_1)\oplus (V_2\boxtimes L_2)\oplus\cdots\oplus (V_{p-1}\boxtimes L_{p-1}),
\]
where $V_1,\dots,V_{p-1}$ are representations of $G$.
The indecomposable objects are then of the form $V\boxtimes L_i$, where $V$ is an indecomposable $G$-module.  
Morphisms between objects are determined by
\[
\Hom(V\boxtimes L_i,V'\boxtimes L_j)\cong\begin{cases}
    \Hom(V,V') & \text{ if }i=j,\\
    0 & \text{ if }i\neq j.
\end{cases}
\]
Similarly, the tensor product on $\Rep_{\Ver_p}G$ is given by:
\[
(V\boxtimes L_i)\otimes(V'\boxtimes L_j)=(V\otimes V')\boxtimes(L_i\otimes L_j).
\]
Consider the natural inclusion of tensor categories $\Ver_p\sub\Rep_{\Ver_p}G$, given by endowing an object in $\Ver_p$ with the trivial action of $G$. 
The tensor product then induces an exact functor
\begin{equation}\label{eqn action verp on repG}
\Rep_{\Ver_p}G\boxtimes\Ver_p\to\Rep_{\Ver_p}G.
\end{equation} 

\subsection{OTI functors for finite groups}

In characteristic $p$, the representation theory of a finite group is not semisimple if and only if the order of the group $|G|$ is divisible by $p$. By Cauchy's theorem, this occurs exactly when $G$ contains a cyclic group $H$ of order $p$. 

Let $H$ be such a cyclic subgroup of $G$.
Consider the composition of the restriction functor with semisimplification:
\[
\Phi_H:\operatorname{Rep}_{\Bbbk}G\xto{\operatorname{Res}_{H}}\operatorname{Rep}_{\Bbbk}H\xto{\varphi_H}\operatorname{Ver}_p.
\]
Note that the centraliser $C_G(H)$ naturally acts on each component $\varphi_i(V)$~(see~\eqref{formulaOTI}) of $\varphi_H(V)$ with trivial $H$ action, and hence $\Phi_H(V)$ naturally carries an action of $C_G(H)/H$.    
The \textit{OTI functor for finite groups} is given by:
\begin{equation*}
    \Phi_H:
    \Rep_{\Bbbk}G\to \Rep_{\verp}C_G(H)/H.
\end{equation*}

Note we may upgrade the source to get a functor 
\begin{equation*}
\Rep_{\verp}G\cong \Rep_{\Bbbk}(G)\boxtimes \verp\xrightarrow{\Phi_H\boxtimes 1} \Rep_{\verp}(C_G(H)/H)\boxtimes \verp\xrightarrow{(\ref{eqn action verp on repG})} \Rep_{\verp}C_G(H)/H,   
\end{equation*}
which we also denote by $\Phi_H$.
Explicitly, $\Phi_H(M\boxtimes L_i)=\Phi_H(M)\otimes L_i$, where $\Phi_H(M)$ is a $C_G(H)/H$-module in $\verp$.

\begin{remark}
The functor $\Phi_H$ semisimplifies the action of the cyclic subgroup $H$, hence $\rep C_G(H)/H$ is ``closer'' to being semisimple.  
Now set $G_1:=G$, $H_1:=H$, and~$G_2~:=~C_{G}(H)/H$.  If $p \op{\big|} |G_2|$, choose a cyclic subgroup $H_2\sub G_2$ of order $p$, and let $G_3:=C_{G_2}(H_2)/H_2$.  Continuing like this until $(|G_n|,p)=1$, we obtain symmetric monoidal functors
\[
\Rep_{\Ver_p}G_1\xto{\Phi_{H_1}}\Rep_{\Ver_p}G_2\xto{\Phi_{H_2}}\Rep_{\Ver_p}G_3\xto{}\cdots\xto{\Phi_{H_{n-1}}}\Rep_{\Ver_p}G_n,
\]
where $\Rep_{\Ver_p}G_n$ is \emph{semisimple}.  Note that this composition of functors depends on the choices of cyclic subgroups at each step.  
The above iteration procedure has been much exploited in the super world, and was used to the compute the superdimensions of irreducible representations (see for instance \cite{GH} and \cite{HW}).
\end{remark}

\begin{example}\label{example fin gps S_n}
Let $G=S_n$ for $n\geq p$, and let $H\sub S_n$ denote the subgroup generated by the $p$-cycle $(n-p+1,\dots,n)$.  Then we obtain an OTI functor $\Phi_H:\rep S_n\to\Rep_{\Ver_p}S_{n-p}$.
Applying this iteratively as above, we get a functor $\Rep_{\Ver_p} S_n\to\Rep_{\Ver_p}S_{r}$, where $n=qp+r$ for $r<p$. Note that $\Rep_{\Ver_p}S_{r}$ is semisimple.
\end{example}

We collect some properties of $\Phi_H$.  First, for $\mathbf{V}=\bigoplus_i(V_i\boxtimes L_i)\in\Rep_{\Ver_p}G$, define 
\begin{equation}\label{eqn dimp}
\dim_p\mathbf{V}:=\dim (V_1)+2\dim (V_2)+\dots+(p-1)\dim (V_{p-1}) \quad \text{ (mod }p)\in\F_p.
\end{equation}
Note that $\dim_p\mathbf{V}=\operatorname{tr}(\id_{\mathbf{V}})$ (see~Section~\ref{eqn trace in tens cat}), the categorical dimension of $\mathbf{V}$.

\begin{lemma}\label{lemma properties of OTI finite groups}
    The OTI functor $\Phi_H:\rep G\to \Rep_{\Ver_p}(C_G(H)/H)$ satisfies the following properties:
    \begin{enumerate}
        \item[$\operatorname{(i)}$] $\Phi_H$ is $\Bbbk$-linear and symmetric monoidal;
        \item[$\operatorname{(ii)}$] $\dim_p \Phi_H(M)\equiv\dim M \quad (\mod p)$;   
        \item[$\operatorname{(iii)}$] $\Phi_H(M)=0$ if and only if $\operatorname{Res}_HM$ is a projective (equivalently, free) $H$-module;
        \item[$\operatorname{(iv)}$] $\Phi_H(P)=0$ for any projective $G$-module.  
        \item[$\operatorname{(v)}$] If $X$ is a $G$-set and we write $\Bbbk[X]$ for the corresponding permutation module, then 
        \[
        \Phi_H(\Bbbk[X])=\Bbbk[X^H]\boxtimes L_1.
        \]
    \end{enumerate}
\end{lemma}
\begin{proof}
    Property (i) follows because $\Phi_H$ is a composition of restriction and semisimplification, which are both symmetric monoidal functors.  Property (ii) holds because symmetric monoidal functors preserve the trace of an endomorphism and $\tr(\id_M)\equiv \dim M$ mod $p$. Property (iii) follows from~Lemma~\ref{prop OTI on indecomps}, and (iv) follows immediately from (iii).

    For (v), we may decompose $X=\bigcup\limits_{i}\OO_i$ into a union of $H$-orbits, and obtain a corresponding decomposition 
    $\Bbbk[X]=\bigoplus\limits_i\Bbbk[\OO_i]$ of $H$-modules, with each $\Bbbk[\OO_i]$ an indecomposable $H$-module of dimension $|\OO_i|$.  Since $H$ is a $p$-group, all nontrivial orbits are of size a multiple of $p$, hence are killed by $\Phi_H$, leaving us with only trivial orbits, i.e.~$\Bbbk[X^H]$.
\end{proof}

\begin{remark}
    On summands of permutation modules, $\Phi_H$ lands in the $L_1$-component of $\Rep_{\verp}(G)$, and~Lemma~\ref{lemma properties of OTI finite groups}(v) shows that $\Phi_H$ agrees with the Brauer construction (Tate cohomology) on these modules. 
\end{remark}

\begin{remark}
There is another approach to obtaining the action of $C_G(H)$ on $\Phi_H$ which is more in the spirit of Hopf algebras.  Recall that a $G$-module is simply a module over the group algebra $\Bbbk G$, which is itself a cocommutative Hopf algebra.  Since $\Phi_H:\rep G\to\Ver_p$ is symmetric monoidal, $\Phi_H(\Bbbk G)$ is once again a cocommutative Hopf algebra, where $H$ acts on $\Bbbk G$ by conjugation.  Further, for any $G$-module $M$, $\Phi_H(M)$ is naturally a module over $\Phi_H(\Bbbk G)$.  Using (v) of~Lemma~\ref{lemma properties of OTI finite groups}, it is not difficult to show that $\Phi_H(\Bbbk G)$ is canonically identified, as a Hopf algebra, with $\Bbbk C_G(H)$.
\end{remark}

\subsection{OTI functors from shifted cyclic subgroups}\label{section OTI shifted cyclic}  We continue to let $G$ be a finite group.  We may generalize our construction above by considering subalgebras $R$ of $\Bbbk G$ isomorphic to $\Bbbk[x]/(x^p)$, and applying the functor 
\[
\Phi_R:\Rep_{\Bbbk} G\xto{\operatorname{Res}_{R}} R\operatorname{--}\mod\xrightarrow{\varphi} \Ver_p,
\]
with $\varphi$ defined as in (\ref{eqn varphi}).  In this case, $\Phi_R(M)$ admits a natural action of the centralizer of $x$ in $\Bbbk G$, and $\Phi_R(M)=0$ if and only if $M$ is free over $R$. 

One case of particular interest are the subalgebras arising from shifted cyclic subgroups, which we now briefly recall (see \cite[Sec.~5.8]{Benson}).

Let $A\cong C_p^r$ be an elementary abelian $p$-group, and write \[
\Bbbk A=\Bbbk[x_1,\dots,x_r]/(x_1^p,\dots,x_r^p)
\] where $1-x_1,\dots,1-x_r$  are generators for $A\cong C_p^r$.  A \emph{shifted cyclic subgroup} $R_{\mathbf{a}}$ of $A$ is a subalgebra of $\Bbbk A$ generated by an element of the form $z_{\mathbf{a}}:=a_1x_1+\dots+a_rx_r$, where $\mathbf{a}=(a_1,\dots,a_r)\in\Bbbk^r$ is non-zero. Thus $R_{\mathbf{a}}\cong\Bbbk[z_{\mathbf{a}}]/(z_{\mathbf{a}}^p)$.  Note that $R_{\mathbf{a}}$ will not be the group algebra of a subgroup of $A$ in general.

Shifted cyclic subgroups correspond to the closed points of $\Spec H^\bullet(A,\Bbbk)$ (which happens to be the Balmer spectrum \cite{Balmer} of $\D(\rep A)$).  If $M$ is a finite-dimensional $A$-module, then $\Phi_R(M)=0$ for all shifted cyclic subgroups $R$ of $A$ if and only if $M$ is a projective (equivalently, free) $A$-module.  Combined with Chouinard's Lemma \cite{Benson}, which states that a module over a general finite group $G$ is projective if and only if it is so over every elementary abelian subgroup, the OTI functors $\Phi_R$ can detect projectivity of any finite-dimensional $G$-module. 

\subsection{Morphism on mod $p$ Grothendieck rings}\label{lem:K0 of ss}  We again let $H\sub G$ be a cyclic subgroup of order~$p$.
The functor $\Phi_H$ is not exact, so does not induce a morphism on Grothendieck groups.  Nevertheless, a weaker statement, shown below, still holds.
Throughout, we will write $\overline{K_0}(\Cs):=K_0(\Cs)\otimes_{\Z}\F_p$ for the $\mod p$ Grothendieck group of an abelian category $\Cs$.
Note that $\mathbb{F}_p$ is a $K_0(\Ver_p)$-algebra via the morphism $\dim_p$ (see~(\ref{eqn dimp})), and we have an isomorphism
\begin{align*}
    K_0(\operatorname{Rep}_{\operatorname{Ver}_p}G)\otimes_{K_0(\Ver_p)}\mathbb{F}_p &\cong
    (K_0(\operatorname{Rep}_{\Bbbk}G)\otimes_{\mathbb{Z}} K_0(\operatorname{Ver}_p))\otimes_{K_0(\Ver_p)}\mathbb{F}_p
    \\
    & \cong \overline{K_0}(\operatorname{Rep}_{\Bbbk}G).
\end{align*}
Since the semisimplification functor $\varphi:\Rep_{\Bbbk} C_p\to\Ver_p$ is symmetric monoidal, it preserves categorical dimension, meaning that $\dim_p\varphi(M)\equiv\dim M$ (mod $p$).  It follows that the assignment
\[
\phi:\overline{K_0}(\Rep_{\Bbbk} C_p)\to  \mathbb{F}_p, \quad [X]\mapsto \dim_p \varphi(X)
\]
is well-defined and agrees with the algebra morphism 
$$\overline{K_0}(\Rep_{\Bbbk} C_p)\to  \overline{K_0}(\operatorname{vec}) = \mathbb{F}_p$$
induced by the exact functor $\Rep_\Bbbk C_p\to \operatorname{vec}$, forgetting the $C_p$-action. The following is an easy corollary, and is proved more generally in~Proposition~\ref{prop:mod_p_groth_ring}.

\begin{prop}\label{lemma red groth ring finite gps}
The OTI functor $\Phi_H:\rep G\to \Rep_{\Ver_p}(C_G(H)/H)$
induces an algebra morphism on $\mod p$ Grothendieck groups, 
\[
\phi_H:\overline{K_0}(\rep G)\to \overline{K_0}(\rep C_G(H)/H), \quad [M]\mapsto[\Phi_H(M)]
\]
Under the natural identification $\ol{K_0}(C_G(H))\cong\ol{K_0}(C_G(H)/H)$, we have 
\[
\phi_H([M])=[\operatorname{Res}_{C_G(H)}(M)].
\]
\end{prop}

\begin{remark}\label{remark super red gr gp}
    Note that Proposition~\ref{lemma red groth ring finite gps} has a well-known analogue in the super setting, first observed in \cite{HR} (see \cite{GHSS} for a general description in terms of Grothendieck rings).
\end{remark}
\subsection{Application: Glauberman Correspondence}\label{sec:Glauberman}

In this section, we use the OTI functor to categorify the cyclic case of the Glauberman correspondence \cite{GL}, which is a  bijection relating irreducible characters of finite groups. Our categorification 
provides a natural interpretation for the sign that appears in this bijection.  First we recall the statement.  

\begin{thm}\label{Glauberman corresp}
Let $H$ be a group of order $p$ acting on a finite group $G$ of order coprime to $p$. Then there exists a canonical bijection between the (characteristic zero) irreducible representations of $G$ fixed by $H$ and the irreducible representations of the invariants $G^H$: \[V\mapsto \widehat{V}\]
Up to signs, this correspondence is induced by the restriction map on Grothendieck rings mod $p$:
\begin{equation}\label{eq:glaub_congruence}
    [\mathrm{Res}^G_{G^H}(V)]\equiv\pm [\widehat{V}] \quad \mathrm{mod}\ p.
\end{equation}
\end{thm}

As the orders of $G$ and $G^H$ are coprime to $p$, the representation theory of these groups in characteristic $p$ is semisimple, with irreducible representations mirroring those in characteristic $0$.
Therefore, we may just as well interpret the Glauberman correspondence as a bijection relating simple modules in characteristic $p$.

Moreover, by Clifford theory (see Theorem 4 of \cite{Alperin}), for any simple module $V$ of $G$ fixed by $H$, there is a unique simple module $\tilde{V}$ of $G\rtimes H$ restricting to $V$. Let $\Phi_{G\rtimes H,H}$ denote the OTI functor for $H$ in $G\rtimes H$ and note that $C_{G\rtimes H}(H)/H\cong G^H$.

\begin{thm}\label{OTIGlauberman}
In characteristic $p$, the OTI functor 
\[\Phi_{G\rtimes H,H}:\Rep_{\Bbbk}(G\rtimes H)\to \Rep_{\verp} (G^H)\]
categorifies the Glauberman correspondence: it takes the extension of an $H$-fixed simple module $V$ in $\Rep_{\Bbbk}(G\rtimes H)$ to its Glauberman correspondent $\widehat{V}$, in the component determined by the sign mod $p$:
\[\Phi_{G\rtimes H,H}(\tilde{V})=\widehat{V}\boxtimes L_{\pm 1}\]
\end{thm}

\begin{remark}
Note that on the level of vector spaces, the Glauberman correspondence may also be induced by the Brauer construction for $H$. By involving the Verlinde category we provide a new categorical interpretation of the sign that appears in the correspondence. Our approach also illustrates an extra monoidal compatibility of this map.
\end{remark}

The proof is a corollary of a compatibility between induction and the OTI functor:

\begin{lemma}
Let $H$ be a group of order $p$ acting on a group $G$ of order coprime to $p$, with centraliser $G^H$. Then as functors from $\Rep_{\Bbbk}(G^H\times H)$ to $\Rep_{\verp} (G^H)$ we have:
\[\Phi_{G\rtimes H,H}\circ \Ind^{G\rtimes H}_{G^H\times H}\cong \Phi_{G^H\times H,H}.\]
\end{lemma}

\begin{proof}
This induction splits as a direct sum of subspaces for the cosets of $G^H$ in $G$, and the action of $H$ permutes these. As $H$ and $G$ are coprime, the action of $H$ on $G/G^H$ fixes only the trivial coset of $G^H$ (see Lemma 3 of \cite{GL}). Thus, the action of $H$ on this decomposition fixes only the trivial component, inducing the desired identification of functors.
\end{proof}

\begin{proof}[Proof of Theorem \ref{OTIGlauberman}]
First note that the index of $G^H\times H$ in $G\rtimes H$ is coprime to $p$, so for any $G$ module $U$, the canonical map \[\Ind^{G\rtimes H}_{G^H\times H}\mathrm{Res}^{G\rtimes H}_{G^H\times H}(U)\to U\] is split surjective. From this, we see that every simple module of $G\rtimes H$ is a direct summand of the induction of an indecomposable module of $G^H\times H$. By uniserial theory (see \cite{Alperin}), these indecomposable modules all have the form $W\boxtimes M_j$ for $W$ an irreducible module of~$G^H$. So for $U$ a simple $G\rtimes H$ module, there exists some $W$ and $M_j$ such that
\[U\subset_{\oplus} \Ind^{G\rtimes H}_{G^H\times H}(W\boxtimes M_j).\]
Now with $\Phi \coloneqq \Phi_{G\rtimes H,H}$, we have the split inclusion:
\[\Phi(U)\subset_{\oplus} \Phi\Ind^{G\rtimes  H}_{G^H\times  H}(W\boxtimes M_j)\cong \Phi(W\boxtimes M_j)\cong W\boxtimes L_j.\]
This shows for simple $U$, the object $\Phi(U)$ is either simple or zero in $\Rep_{\verp}(G^H)$. As the Glauberman correspondence (with corresponding sign) is determined by restriction in mod $p$ Grothendieck groups, Proposition~\ref{lemma red groth ring finite gps} completes the proof.
\end{proof}


\section{OTI functors on  categorical Heisenberg actions}\label{section:oti_functors_on_heisenberg_categorifications}

We discuss the OTI functor in the setting of categorical Heisenberg actions. 

\subsection{CF functors}\label{section cf functors}
Let $n\in\N$ and fix a subgroup $H\sub S_n$. 
Recall that a subgroup $H'\sub S_n$ is \emph{transitive} if its action on $\{1,\dots,n\}$ is transitive. 
 
\begin{definition}\label{def:cf-functor}
    Let $\As$ be an abelian category.
    A $\Bbbk$-linear functor $\varphi:\rep H\to\As$ is a  \textit{CF functor} for $H\sub S_n$ if $\varphi(\Ind_{H'}^{H}M)=0$ for all non-transitive subgroups $H'\sub H$ and all $H'$-modules $M$. 
\end{definition}

We note that Definition \ref{def:cf-functor} allows for $\varphi$ to be the zero functor, but we will usually impose an extra nontriviality assumption on $\varphi$ (see also Lemma \ref{lemma CF nontriv}).

\begin{example}\label{example CF functors}
    If $n=p$, $C_p$ is the unique minimal transitive subgroup of $S_p$ up to conjugacy.  A $\Bbbk$-linear functor $\varphi:\rep C_p\to\As$ is a CF functor for $C_p\subset S_p$ if and only if $\varphi$ vanishes on projective modules.  There are several examples of such functors:
    \begin{enumerate}
        \item We may take $\As=\Ver_p$ and $\varphi:\rep C_p\to\Ver_p$ the semisimplification functor.  Then $\varphi$ is given, componentwise, by the formulas in (\ref{formulaOTI}).  In this case $\varphi$ is symmetric monoidal.
        
        \item Take $\As=\operatorname{vec}_{\Bbbk}$ and let $\varphi=\varphi_1:\rep C_p\to\As$ 
        (see (\ref{formulaOTI})). This choice of CF functor is a direct summand of the one in example (1).  It is no longer monoidal, although it is lax monoidal.
        
        \item Take $\As=\operatorname{vec}_{\Bbbk}$ and let $\varphi(V)=\ker(x)/\im(x^p)$, where we follow the notation of~Section~\ref{section oti explicit}.  This CF functor is the well-known Brauer construction.
    \end{enumerate}
\end{example}

\begin{example}\label{cf varphi 1}
    Let $\As$ be a locally finite abelian category.
    Suppose $\varphi:\rep H\to \As$ is a CF functor for $H\sub S_{n}$.  Write $\mathbf{1}$ for the trivial representation in $\rep H$ as usual.
    Set $\varphi_{\mathbf{1}}:\rep H\to\vec$ to be 
    \[
    \varphi_{\mathbf{1}}(M)=\frac{M^H}{M^H\cap\left(\sum\limits_{h\in H}\im(1-h)\right)}.
    \]
    It is straightforward to check that $\varphi_{\mathbf{1}}$ is a CF functor, and that we have a noncanonical, split inclusion $\varphi_{\mathbf{1}}(M)\boxtimes\varphi(\mathbf{1})\to \varphi(M)$ for any $H$-module $M$.
\end{example}

\begin{example}\label{example CF to vec general}
    Let $\varphi:\rep H\to\As$ be a CF functor for $H\sub S_n$.  If $\As$ is additionally locally finite (which is our main case of interest), then it admits an exact conservative functor $\omega:\As\to\vec$, and $\omega\circ\varphi:\rep H\to\vec$ will be another CF functor.
\end{example}

\begin{remark}
    In \cite{CF}, Coulembier and Flake introduce a very similar notion, which they call a $\mathbf{V}$-functor.  The difference to our setup is that in \emph{loc.~cit.}~it is assumed that $\As$ is a tensor category and that $\varphi$ is monoidal.  We drop this monoidality assumption as it is less natural in our setting.
    In particular, if $\As$ and $\varphi$ are monoidal, then the condition to be a CF functor may be weakened to $\varphi(\Ind_{H'}^{H}\mathbf{1})=0$, as is done in \cite{CF}. 
\end{remark}

The following is a slight generalization of \cite[Lem.~5.1.3]{CF}.

\begin{lemma}\label{lemma CF nontriv}
    The following are equivalent for a subgroup $H\sub S_{n}$:
    \begin{enumerate}
        \item There exists a nonzero CF functor for $H\sub S_{n}$;
        \item The semisimplification functor $\rep H\to (\rep H)^{ss}$ is a CF functor;
        \item There is a $p$-subgroup $P\sub H$ which is transitive in $S_n$;
        \item $n$ is a power of $p$, and $H$ is a transitive subgroup of $S_n$.
    \end{enumerate}
\end{lemma}
\begin{proof}
    We may use the same proof as given in \cite{CF}, only the direction $(1)\Rightarrow(3)$ needs a small adjustment. For this, let $P\sub H$ be a Sylow subgroup.  Then for any $H$-module $M$, $M$ splits off $\Ind_P^HM$.  Since $\varphi$ is nonzero, this forces $P$ to be a transitive subgroup.
\end{proof}

\begin{remark}\label{remark CF general subgroup}
    If $\varphi:\rep H\to\As$ is a CF functor for $H\sub S_{p^r}$, then $\tilde{\varphi}:=\varphi\circ\operatorname{Res}_H:\rep S_{p^r}\to\As$ is also a CF functor.  Thus without loss of generality we may assume all CF functors have $H=S_{p^r}$. However we will work with more general subgroups, mostly for conceptual clarity.
\end{remark}


\subsection{Categorical Heisenberg actions}\label{sec:Heis-act} In the following, we mainly follow the conventions and notation developed in \cite{BSW}. Let $\Cs$ be a locally finite abelian category.  A \emph{degenerate categorical Heisenberg action} on $\Cs$ is the data of:
\begin{enumerate}
    \item An integer $k\in\Z$ (the \emph{central charge}).
    \item Biadjoint endofunctors $E,F:\Cs\to\Cs$ (in particular, $E$ and $F$ are exact).  This gives rise to units and counits of adjunctions, which we denote as follows:
    \[
        \eta_1:\mathbf{1} \to FE, \hspace{1em} \eta_2: \mathbf{1}\to EF,
    \]
    \[
        \epsilon_1: EF \to \mathbf{1}, \hspace{1em} \epsilon_2: FE \to \mathbf{1}.
    \]
    \item A natural transformation $x:E\to E$.
    \item A natural transformation $T:E^2\to E^2$.
\end{enumerate}
This data is subject to a number of relations, which are explained in~\cite[Sec.~3]{BSW}, as well as in~Section~\ref{section diagrammatics}.
One consequence of these relations is that the natural transformation $T$ induces an action of $S_n$  on $E^n$ via $s_i \mapsto E^{n-i-1}TE^{i-1}$, where $s_i$ is the simple transposition $(i,i+1)$.  Another consequence is that in the Grothendieck group we have the Heisenberg relation $[E][F]-[F][E]=k$.

Now suppose that $\As$ is another locally finite abelian category.  Then by~Lemma~\ref{lemma additive functors extend},  the Deligne tensor product $\Cs\boxtimes\As$ inherits a natural degenerate categorical Heisenberg action.  For example, our biadjoint functors will be $E\boxtimes1$ and $F\boxtimes1$. 

\begin{example}\label{example rep S_n 1}
Representations of symmetric groups provide the central example of a categorical Heisenberg action in this paper.  Let $\mathcal{S}ym=\bigoplus\limits_{n\in\N}\rep S_n$.  For $M\in\rep S_n$, set
$$E=\bigoplus\limits_{n\geq 1}\operatorname{Res}_{S_{n-1}}^{S_n}\quad \text{and}\quad F=\bigoplus\limits_{n\in\N}\Ind_{S_{n}}^{S_{n+1}}.$$
Then $E$ and $F$ are biadjoint functors on $\mathcal{S}ym$.  We have an endomorphism $x:E\to E$ given on $\operatorname{Res}_{S_{n-1}}^{S_n}$ by the action of the Jucys--Murphy element $x=(1,n)+(2,n)+\dots+(n-1,n)\in \Bbbk E$.  
There is a natural transformation $T$ on $E^2 $ given by right multiplication by the simple transposition $s_{n-1}$ on the component $\operatorname{Res}^{S_n}_{S_{n-2}}$.
This data defines a degenerate categorical Heisenberg action on $\mathcal{S}ym$ of central charge $-1$  (see \cite{LLT} for more details), which categorifies the Fock space representation of the Heisenberg algebra.
\end{example}

\subsection{Morphisms  of categorical actions} 
Suppose that $(\Cs,E,F,T,x,\eta_1,\eta_2,\epsilon_1,\epsilon_2),$ and $(\Ds,E',F',T',x',\eta_1',\eta_2',\epsilon_1',\epsilon_2')$ are locally finite abelian categories with degenerate categorical Heisenberg actions of the same central charge.

\begin{definition}\label{defn categorical commutation}
    A morphism of degenerate categorical Heisenberg actions is the data of a functor $\Psi:\Cs\to\Ds$ and natural isomorphisms
    \[
    \zeta_E:\Psi E\xto{\sim}E'\Psi, \ \ \ \zeta_F:\Psi F\xto{\sim}F'\Psi,
    \]
    such that the following diagrams commute:
    \begin{enumerate}
        \item \phantom{.}\vspace*{-1em} 
            \[
            \xymatrix{
            \Psi E\ar[rr]^{\zeta_E} \ar[d]^{\Psi x} & & E'\Psi\ar[d]^{x'\Psi}\\
            \Psi E \ar[rr]^{\zeta_E} &  &E'\Psi
            }
            \vspace{0.5em}
            \]
        \item \phantom{.}\vspace*{-1em} 
            \[
            \xymatrix{
            \Psi E^2\ar[rr]^{\zeta_EE}\ar[d]_{\Psi T} && E'\Psi E\ar[rr]^{E'\zeta_E} && (E')^2\Psi\ar[d]^{T'\Psi} \\ 
            \Psi E^2\ar[rr]^{\zeta_EE} && E'\Psi E\ar[rr]^{E'\zeta_E} && (E')^2\Psi
            }
            \vspace{0.5em}
            \]
        \item \phantom{.}\vspace*{-1em} 
            \[
            \xymatrix{
            && \Psi\ar[dll]_{\Psi\eta_1} \ar[drr]^{\eta_1'\Psi} && \\
            \Psi FE \ar[rr]^{\zeta_FE} && F'\Psi E\ar[rr]^{F'\zeta_E} && F'E'\Psi
            }
            \vspace{0.5em}
            \]
        \item \phantom{.}\vspace*{-1em} 
             \[
            \xymatrix{
            && \Psi\ar[dll]_{\Psi\eta_2} \ar[drr]^{\eta_2'\Psi} && \\
            \Psi EF \ar[rr]^{\zeta_EF} && E'\Psi F\ar[rr]^{E'\zeta_F} && E'F'\Psi
            }
            \vspace{0.5em}
            \]
    \end{enumerate}
\end{definition}

\subsection{Categorical Kac--Moody action}
By the biadjunction between $E$ and $F$, the natural transformation $x:E\to E$ induces a natural transformation $y:F\to F$.  By the assumptions made on $\Cs$, the actions of $x$ and $y$ on $E(M)$ and $F(M)$ will be locally finite for any object $M$ of $\Cs$.  Thus $E(M)$ and $F(M)$ admit generalized eigenspace decompositions which are natural in $M$.  This implies we obtain a decomposition of functors:
\[
E=\bigoplus_{i\in\Bbbk} E_{i}, \ \ \ F=\bigoplus\limits_{i\in\Bbbk}F_i.
\]
Let $I$ be the set of eigenvalues that appear in these decompositions. There is an associated Kac--Moody algebra $\mathfrak{g}_I$ constructed from $I$ (see \cite{BSW} for the precise definition). The main result of \cite{BSW} is that this gives rise to a categorical Kac--Moody representation of $\mathfrak{g}_I$ on $\Cs$.  Decategorifying induces an action of $\mathfrak{g}_I$ on $K_0(\Cs)$, where $E_i$ and $F_i$ decategorify to the Chevalley generators.
In cases of interest to us below, $I=\F_p$ and the Kac--Moody algebra is $\widehat{\s\l}_p' = \s\l_p(\C[t,t^{-1}])\oplus \C c$.  

\begin{example}\label{example rep S_n2}
 The categorical Heisenberg action of Example \ref{example rep S_n 1} induces an action of $\widehat{\s\l}_p'$ on $\mathcal{S}ym$. The Grothendieck group of $\mathcal{S}ym$ then provides an integral form for the basic representation of $\widehat{\s\l}_p'$:
$$
K_0(\mathcal{S}ym)\otimes_{\Z}\C \cong L(\Lambda_0).
$$
\end{example}

\subsection{OTI functor on categorical actions}

Let $H$ be a subgroup of $S_{p^r}$, so that by restriction we obtain a homomorphism $\Bbbk H\to \End(E^{p^r})$.  It follows that we may view $E^{p^r}$ as defining a functor:
\[
E^{p^r}:\Cs\to \Rep_{\Cs}(H)\simeq \Cs\boxtimes\rep H.
\]

\begin{definition}
Let $\varphi:\rep H\to\As$ be a CF functor for $H\subset S_{p^r}$. Given a degenerate categorical Heisenberg action on $\Cs$, define the \textit{OTI functor} $\Phi_{\varphi}:\Cs\to\Cs\boxtimes\As$ by the composition:
    \[
    \Cs\xto{E^{p^r}}\Cs\boxtimes\rep H\xto{1\boxtimes\varphi}\Cs\boxtimes\As.
    \]
\end{definition}

\begin{remark}
    Note that the roles of $E$ and $F$ may be interchanged in the definition of a categorical Heisenberg action, so we could also define an OTI functor in terms of $F$.
\end{remark}

We now provide examples of OTI functors for different, well-known categorical actions.

\begin{example}\label{example rep S_n 3}
We continue to study $\mathcal{S}ym$ as in Examples \ref{example rep S_n 1} and \ref{example rep S_n2}.
Let $H\sub S_{p^r}$, and let $\varphi:\rep H\to \As$ be a CF functor.  Then we obtain an OTI functor $\Phi_{\varphi}:\mathcal{S}ym\to\mathcal{S}ym\boxtimes \As$ whose components are given by the following composition:
\[
\rep S_n\xto{E^{p^r}}\rep S_{n-p^r}\boxtimes\rep H\xto{1\boxtimes\varphi}\rep S_{n-p^r}\boxtimes\As. 
\]
Here, the first arrow is simply the restriction functor to the subgroup $S_{n-p^r}\times H\sub S_n$.  In particular, if $r=1$ and $\varphi:\rep S_p\to\Ver_p$ is the semisimplification functor, then $\Phi_{\varphi}$ agrees with the OTI functor in~Example~\ref{example fin gps S_n}.  
This example is the main topic of~Section~\ref{section:symmetric_groups}.
\end{example}

\begin{example}
   Let $\PP ol$ be the category of strict polynomial functors. Objects in $\PP ol$ are functors $M:\vec\to\vec$ such that the induced map on morphisms is polynomial (see~Section~\ref{section poly functors defn} for the precise definition).  The category $\PP ol$ admits a degenerate categorical Heisenberg action, originally defined in \cite{HY}, see also Section \ref{section poly functors defn}.  It was proven in \cite{HY} that the Schur--Weyl duality functor $\mathcal{F}:\PP ol\to\mathcal{S}ym$ is a morphism of categorical actions, and we will use this in Section \ref{section SW duality} to show that $\mathcal{F}$ intertwines the OTI functors on each category.
\end{example}
We provide a final example in which the OTI functor exhibits a very different behavior.
\begin{example}\label{example glv}
    Let $V$ be a vector space, and consider the category of algebraic representations $\rep GL(V)$.  We have a degenerate categorical Heisenberg action of central charge $0$ given by:
    \[
    E(M)=V\otimes M, \ \ \ F(M)=V^*\otimes M.
    \]
    We refer to the introduction of \cite{BSW} for more details and references.
    Let $H\sub S_{p^r}$ and choose a CF functor $\varphi:\rep H\to\As$.  Then we have
    \[
    \Phi_{\varphi}(M)=V^{(r)}\otimes M\boxtimes \varphi(\mathbf{1}),
    \]
    where $V^{(r)}$ denotes the $r$th Frobenius twist of $V$.  (If we define the OTI functor in terms of $F$, then we obtain 
    $\Phi_{\varphi}(M)=(V^*)^{(r)}\otimes M\boxtimes \varphi(\mathbf{1})$.)

    Two phenomena of note appear in this example.  One is that $\Phi_{\varphi}$ only depends on $\varphi(\mathbf{1})$ and the power of $p$. 
    The second is that $\Phi_{\varphi}$ is exact.  The exactness of $\Phi_{\varphi}$ is notable, as it holds for all known (to the authors) examples of categorical Heisenberg actions of central charge 0.  
\end{example}

\begin{remark}
    Consider the setting of Example \ref{example glv}. Let $A_0$ denote the (unshifted) fundamental alcove, $\lambda_r$ the $r$th Steinberg weight, and $\mathfrak{X}_r=\lambda_r+p^rA_0$.  Then for $\lambda\in\mathfrak{X}_r$, the Donkin tensor product theorem (\cite{Don93}) implies that 
    \[
    T_{\lambda}\otimes V^{(r)}\cong T_{\lambda+p^r\varpi_1},
    \]
    where $T_{\mu}$ denotes in the indecomposable tilting module of highest weight $\mu$.
    This isomorphism is suggestive of a periodic equivalence in this setting, in the spirit of those we study for $\SS ym$ and $\PP ol$ further on in the paper.
\end{remark}

\subsection{OTI functors commute with the categorical action}\label{sec:OTIcommute}
Categorical actions are both rich and highly rigid.  For instance, we saw in Example \ref{example rep S_n2} that $\SS ym$ categorifies an integral form of the basic representation of $\widehat{\mathfrak{sl}_p}'$, which is irreducible over the complex numbers.  Irreducibility is an immediate obstruction to the existence of a nontrivial exact functor  which commutes with the categorical action.  

However, it is still possible for a \emph{non-exact} functor to commute with the categorical action.  Furthermore, the mod $p$ enveloping algebra of $\widehat{\mathfrak{sl}_p}'$ admits a very large centre, and the mod $p$ reduction of the basic representation is reducible (see \cite{DKK}).  In light of Proposition \ref{lemma red groth ring finite gps}, one might hope that there in an OTI functor which categorifies a central element of the mod $p$ center, and realizes an endomorphism of the mod $p$ reduction of the basic representation.  

We will now explain that there is an OTI functor which indeed does all of the above.  To this end, suppose that $\Cs$ admits a degenerate categorical Heisenberg action, and let $\As$ be a locally finite abelian category.  As explained in Section \ref{sec:Heis-act}, $\Cs\boxtimes\As$ will inherit the natural structure of a degenerate categorical Heisenberg action of the same central charge. In Section~\ref{section:appendix_proof_of_categorical_commutation}, we prove the following remarkable theorem.

\begin{thm}\label{thm OTI categorical commutation}
    The OTI functor $\Phi_{\varphi}:\Cs\to\Cs\boxtimes\As$ defines a morphism of categorical actions, where the natural isomorphisms
    \[
    \zeta_E:\Phi_{\varphi}\circ E\xto{\sim}(E\boxtimes1)\circ\Phi_{\varphi}, \ \ \ \zeta_F:\Phi_{\varphi}\circ F\xto{\sim}(F\boxtimes1)\circ\Phi_{\varphi},
    \]
    are explicitly described in~Section~\ref{proof thm cat commutation}.
\end{thm}

Theorem \ref{thm OTI categorical commutation} immediately implies the following:
\begin{cor}
    For $i\in\Bbbk$, we have natural isomorphisms
    \[
    \Phi_{\varphi}\circ E_i\cong(E_i\boxtimes 1)\circ\Phi_{\varphi}, \ \ \ \Phi_{\varphi}\circ F_i\cong(F_i\boxtimes 1)\circ\Phi_{\varphi}.
    \]
\end{cor}

\begin{remark}
    For the reader interested in functors $\Cs\to\Cs$ which are morphisms of categorical actions, this amounts to the case $\As=\vec$, and many examples of such CF functors were provided in Examples \ref{cf varphi 1} and \ref{example CF to vec general}.  
\end{remark}

\subsection{OTI functors and mod p Grothendieck rings}\label{ss:mod_p_groth_ring}

\begin{prop}
\label{prop:mod_p_groth_ring}
Consider the semisimplification functor $\varphi: \Rep_{\Bbbk} C_p\to\Ver_p$.
Let $\Cs $ and $\Ds$ be locally finite abelian categories and consider an exact functor $S:\Cs\to \Ds\boxtimes \Rep_\Bbbk(C_p)$. Then the composition
$$\Cs\xto{S}\Ds\boxtimes \Rep_\Bbbk(C_p)\xto{1\boxtimes \varphi} \Ds\boxtimes \Ver_p$$
induces a morphism on $\mod p$ Grothendieck groups, 
\begin{gather*}
    \overline{K_0}(\Cs) \to \overline{K_0}(\Ds)\cong K_0(\Ds\boxtimes \Ver_p)\otimes_{K_0(\Ver_p)} \mathbb{F}_p: \quad [X] \mapsto [(1\boxtimes \varphi)S(X)],
\end{gather*}
that agrees with the morphism induced by the exact functor
\[\Cs\xto{S} \Ds\boxtimes \Rep_\Bbbk(C_p)\to \Ds\boxtimes \vec\xto{\sim} \Ds.\]
\end{prop}

\begin{proof}
In view of the explicit formula for $\varphi$ (see~\eqref{eqn varphi}), it suffices to prove the following claim: for any object $X\in \Ds$ with endomorphism $x$ satisfying $x^p=0$, we have that 
\[[X]\equiv \sum_{i=1}^{p-1} i\cdot \bigg[\frac{\ker(x)\cap \im(x^{i-1})}{\ker(x)\cap \im(x^i)}\bigg] \hspace{1em} {(\op{mod} p)}. \]
in the mod $p$ Grothendieck group of $\Ds$.

The short exact sequence \[0\to \ker(x)\cap \im(x^{i-1})\to \im(x^{i-1})\to \im(x^i)\to 0\] valid for $1\leq i\leq p-1$ gives
\begin{align*}
    \sum_{i=1}^{p-1} i\cdot \bigg[\frac{\ker(x)\cap \im(x^{i-1})}{\ker(x)\cap \im(x^i)}\bigg]&=\sum _{i=1}^{p-1}i\cdot [\ker(x)\cap \im(x^{i-1})]-i\cdot [\ker(x)\cap \im(x^i)]\\
    &=\sum _{i=1}^{p-1}i\cdot \bigg( [\im(x^{i-1})]-2[\im(x^i)] +  [\im(x^{i+1})]\bigg).
\end{align*}
The coefficient of $[\im(x^i)]$ in this sum is zero if $1\leq i\leq p-2$, so this sum telescopes to the desired equality mod $p$: \[[\im(x^0)]+(p-2)[\im(x^{p-1})]-2(p-1)[\im(x^{p-1})]\equiv [X] \hspace{1em} {(\op{mod} p)}. \]
\end{proof}

Suppose now $\Cs$ admits a degenerate categorical Heisenberg action. We write $e:\overline{K_0}(\Cs)\to \overline{K_0}(\Cs)$ for the morphism on mod $p$ Grothendieck groups induced by $E$. The following corollary follows immediately from the proposition.

\begin{cor}\label{cor mod p e^p}
Let $\varphi: \Rep_{\Bbbk} C_p\to\Ver_p$ be given by semisimplification.
The OTI functor $\Phi_{\varphi}:\Cs\to \Cs\boxtimes \Ver_p$ 
induces a morphism 
$$\overline{K_0}(\Cs)\to \overline{K_0}(\Cs)\cong K_0(\Cs\boxtimes \Ver_p)\otimes_{K_0(\Ver_p)} \mathbb{F}_p:\quad [X]\mapsto [ \Phi_{\varphi}(X)],$$
that agrees with $e^p$.
\end{cor}

\begin{example}
    Returning to $\mathcal{S}ym$ from Examples \ref{example rep S_n 1}, \ref{example rep S_n2}, and \ref{example rep S_n 3}, recall that $K_0(\mathcal{S}ym)$ is an integral form of the basic representation for $\widehat{\mathfrak{sl}_p}$.  Thus $K_0(\mathcal{S}ym)\otimes\mathbb{F}_p$ is the mod $p$ basic representation.  Let $\varphi:\rep C_p\to\operatorname{Ver}_p$ be the CF functor given by semisimplification.  Then it is clear that $\Phi_{\varphi}(\mathbf{1}_{S_n})=\mathbf{1}_{S_{n-p}}$.  By Corollary \ref{cor mod p e^p}, this implies $e^p$ defines an endomorphism of the mod $p$ basic representation which is locally nilpotent but for which no power is zero.  It follows that this representation is reducible and has an infinite composition series, which in particular reproves the reducibility result from \cite{DKK}.
\end{example}

\section{OTI functors on Symmetric groups}\label{section:symmetric_groups}

In this section, we consider symmetric groups and show that OTI functors correspond to taking fixed points on the underlying set of the permutation modules $M^\lambda$.  We introduce a stable range of partitions for which OTI functors are an equivalence on the corresponding subcategories of permutations modules.  This result was essentially proven in \cite{HK} and \cite{Har}, but we recast the proof in terms of OTI functors.  We will extend this equivalence to certain abelian subcategories of $\rep S_n$ in Section \ref{section:periodicity_result}.

Throughout this section, we will fix a prime power $p^r>1$ and consider natural numbers $n$ such that $n \geq p^r$. We will view $S_{p^r}$ embedded in $S_n$ as the subgroup fixing $\{1,\dots,n-p^r\}$ pointwise, and consider the subgroup $S_{n-p^r} \subseteq S_n$ which fixes $\lbrace n-p^r+1,\dots,n-1,n \rbrace$ pointwise. Let $H\sub S_{p^r}$ be a transitive subgroup and fix a CF functor~$\varphi:\rep H\to\As$.  
We will investigate the OTI functor $\Phi_{\varphi}$ (see~Example~\ref{example rep S_n 3}),
    \[\Phi_{\varphi}:
     \rep S_n\xto{\operatorname{Res}}\rep S_{n-p^r}\boxtimes\rep H\xto{\mathbf{1}\boxtimes\varphi}\rep S_{n-p^r}\boxtimes \As.
    \]

\subsection{Partitions and permutation modules in the stable range}  
We follow the notations and conventions used in \cite{James} for partitions, the dominance order, Young diagrams, and tableaux.  In particular we write $\lambda=(\lambda_1,\lambda_2,\dots,\lambda_{\ell(\lambda}))$ for a partition $\lambda$, where $\lambda_1\geq\lambda_2\geq\dots \geq\lambda_{\ell(\lambda)}>0$.  If $\lambda$ is a partition of $n$ such that $\lambda_i\geq k$, then we write $\lambda-ke_i$ for the unique partition of $n-k$ with parts $\lambda_1,\dots,\lambda_i-k, \lambda_{i+1},\ldots,\lambda_{\ell(\lambda)}$.  If $\lambda,\mu$ are partitions, then the dominance order on partitions says that $\lambda\geq\mu$ if for all $i$ we have
\[
\sum\limits_{j\leq i}\lambda_j\geq \sum\limits_{j\leq i}\mu_j.
\]
For a partition $\lambda\vdash n$, we will write $\tab^{\lambda}:=S_n/S_{\lambda}$, where $S_{\lambda}:=S_{\lambda_1}\times\cdots\times S_{\lambda_{\ell(\lambda)}}$.  The permutation module $\Bbbk[\tab^{\lambda}]$ associated to $\tab^{\lambda}$ will be written as $M^{\lambda}$, and the Young module as $Y^{\lambda}$. We further write $S^{\lambda}$ for the Specht module corresponding to $\lambda$, and $D^{\lambda}$ for the irreducible $S_n$-module corresponding to a $p$-regular partition $\lambda$.

\begin{definition}\label{def:stability_conditions}
Let $\lambda\vdash n$. We say that
    \begin{enumerate}
        \item $\lambda$ is \textit{$p^r$-quasistable} if $\lambda_1\geq p^r$ and $\lambda_2<p^r$;
        \item $\lambda$ is \textit{$p^r$-stable} if both
            \begin{enumerate}
                \item $\sum\limits_{i = 2}^{\ell(\lambda)} \lambda_j < p^r$, and
                \item if $n < 3p^r$, then $\lambda_1 \geq \frac{n+p^r}{2}$;
            \end{enumerate}
        \item $\lambda$ is \textit{$p^r$-truncated} if $\lambda = \mu - p^r e_1$ for some $p^r$-stable partition $\mu \vdash (n + p^r)$.
    \end{enumerate}
\end{definition}

\begin{remark}\label{rem:stab}
    For $n \geq 3p^r$, any $p^r$-stable partition $\lambda \vdash n$ satisfies $\lambda_1 >n-p^r\geq \frac{n+p^r}{2}$.  In this case $p^r$-stable is equivalent to $p^r$-truncated.
\end{remark}

The following lemma is straightforward, but will be important later.

\begin{lemma}\label{lemma:p^r_stability_open}
    The subsets of $p^r$-stable and $p^r$-truncated partitions of $n$ are upwards closed with respect to the dominance order of partitions.
\end{lemma}

\begin{prop}\label{prop:oti_on_permutation_modules}
    For $\lambda \vdash n$, the following holds:
    
    \begin{enumerate}
        \item The inclusion $(\tab^\lambda)^H \hookrightarrow \tab^\lambda$ induces an isomorphism
        \[
            \Phi_\varphi(\Bbbk[\tab^\lambda]) \cong \Bbbk[(\tab^\lambda)^H]\boxtimes\varphi(\mathbf{1}).
        \]

        \item The $H$-fixed points on $\tab^\lambda$ decompose, as an $S_{n-p^r}$-set, as
        \[
            (\tab^\lambda)^{H} \cong \coprod\limits_{\lambda_i \geq p^r} \tab^{\lambda - p^re_i}.
        \]
In particular, we have an isomorphism
        \[
            \Phi_{\varphi}(M^{\lambda})\cong\bigoplus\limits_{\lambda_i\geq p^r}M^{\lambda-p^re_i}\boxtimes\varphi(\mathbf{1}).
        \]
        \item $(\tab^\lambda)^H$ has exactly one $S_{n-p^r}$-orbit if and only if $\lambda$ is $p^r$-quasistable.
    \end{enumerate}
\end{prop}
\begin{proof}
   For (1), $\tab^\lambda$ decomposes into finitely many $H$-orbits $\mathcal O_i$. We thus have an isomorphism of $H$-modules
        \[
            \Bbbk[\tab^\lambda] \cong \bigoplus_i \Bbbk[\mathcal O_i].
        \]
        By the definition of a CF functor, we are done if we can show that $\varphi$ annihilates $\Bbbk[\mathcal O_i]$ if and only if $\mathcal O_i$ is not a singleton.  Since these are permutation modules, it is equivalent to show that the stabilizer of any $t\in\mathcal{O}_i$ is either all of $H$ or is not transitive.  However the stabilizer of $t$ in $S_n$ is a conjugate $S^{\sigma}_{\lambda}:=\sigma S_{\lambda}\sigma^{-1}$ of $S_{\lambda}$ in $S_n$. We have $H\cap S_{\lambda}^{\sigma}\sub S_{p^r}\cap S^{\sigma}_{\lambda}$, and the latter is transitive if and only if $S_{p^r}\cap S_{\lambda}^{\sigma}=S_{p^r}$, or equivalently $H\cap S_{\lambda}^{\sigma}=H$.   

     To prove (2), we use the explicit description of $\tab^\lambda$ as fillings of the Young diagram of $\lambda$ with the set $\{1,2..,n\}$ such that the entries in each row are increasing. From this description, we see that if $t \in \tab^\lambda$ is fixed by $H$, then by transitivity of $H$, the labels $n, n-1, \ldots, n-p^r+1$ must be contained in the same row of $t$, and hence this row must have at least $p^r$ boxes. Removing these boxes from $\lambda$ gives a tableau of shape $\lambda - p^re_i$ for some $i$. It is easy to check that this procedure yields a bijection.

    Finally, part (3) is clear from (2).
\end{proof}

\begin{remark}
By Proposition \ref{prop:oti_on_permutation_modules}, $\Phi_{\varphi}(M^{\lambda})$ is a direct sum of permutation modules $M^{\mu}$ where each $\mu$ is obtained by removing $p^r$ boxes from some row of the Young diagram of $\lambda$.  In light of this, the condition that $\lambda$ is $p^r$-quasistable is equivalent to the Young diagram of $\lambda$ having exactly one row where $p^r$ boxes can be removed.
\end{remark}

\subsection{Periodic equivalence for permutation modules}\label{section periodic equiv perm mod}

Write $\operatorname{Perm}_n$ (resp. $\operatorname{Perm}^{\st}_{p^r, n}$, $\operatorname{Perm}^{\trun}_{p^r, n}$) for the full, additive subcategory of $\rep S_n$ containing permutation modules $M^{\lambda}$ for all  (resp. $p^r$-stable, $p^r$-truncated) partitions $\lambda\vdash n$.  Similarly write $\operatorname{Young}_n$ (resp. $\operatorname{Young}_{p^r,n}^{\st}$, $\operatorname{Young}_{p^r,n}^{\trun}$) for the full, additive subcategories generated by the corresponding Young modules.  Note that the subcategories $\operatorname{Young}_{p^r,n}^{(-)}$ are equivalent to the idempotent completions of the categories $\operatorname{Perm}_{p^r,n}^{(-)}$.  

In the remainder of this section, we suppose that $\As$ is a locally finite abelian category.
Given a simple object $L$ in $\As$, write $\langle L\rangle$ for the smallest full, abelian, replete subcategory containing $L$.  It is easy to see that $\langle L\rangle\simeq \operatorname{vec}_{\Bbbk}$, and we will abuse this identification going forward, simply writing vector spaces for objects in this subcategory.  Similarly, if $\mathscr{C}$ is another abelian category such that $\mathscr{C}\boxtimes\mathscr{A}$ exists, we will identify the subcategory $\mathscr{C}\boxtimes\langle L\rangle$ with~$\mathscr{C}$.

The following is the main result of this section.

\begin{thm}\label{thm:oti_eq_on_perm}
    Suppose $\varphi(\mathbf{1})$ is simple. Then $\Phi_{\varphi}$ induces an equivalence
    \[
        \operatorname{Perm}^{\st}_{n, p^r} \xto{\sim} \operatorname{Perm}^{\trun}_{n-p^r, p^r}.
    \]
\end{thm}
Passing to idempotent completions, we obtain:

\begin{cor}\label{cor:oti_eq_on_young}
    Suppose $\varphi(\mathbf{1})$ is simple.
    Then $\Phi_\varphi$ induces an equivalence
    \[
        \operatorname{Young}^{\st}_{n, p^r} \xto{\sim} \operatorname{Young}^{\trun}_{n-p^r, p^r}.
    \]
\end{cor}

\begin{remark}
    A statement equivalent to Theorem \ref{thm:oti_eq_on_perm} can be found in \cite{HK}, Theorem 3.3, albeit in a different language (see also \cite{MW}). A (slightly weaker) formulation which is more in the spirit of our work appears in \cite{Har}.
\end{remark}

We will denote the category of signed Young modules by $\operatorname{sgnYoung}_n:=\Young_n\otimes \sgn$. 
The proofs of this section can be readily adapted to prove the following Corollary.   

\begin{cor}\label{rem sgnyoung}
    Suppose $\varphi(\sgn_H)$ is simple. Then $\Phi_\varphi$ induces an equivalence
    \[
     \operatorname{sgnYoung}^{\st}_{n, p^r} \xto{\sim} \operatorname{sgnYoung}^{\trun}_{n-p^r, p^r}.
    \]
\end{cor}   
For the proof of Theorem \ref{thm:oti_eq_on_perm}, we begin with a lemma.

\begin{lemma}\label{prop:components_of_hom_are_quasistable}
    Let $\lambda, \mu \vdash n$ be $p^r$-stable. Then we have an isomorphism of $S_n$-sets
    \[ 
        \tab^\lambda \times \tab^\mu \cong \coprod_i \tab^{\nu_i},
    \]
    where each $\nu_i$ that appears is $p^r$-quasistable.
\end{lemma}

\begin{proof}
By the Mackey formula, the $\nu_i$ that appear correspond to Young subgroups $S_\lambda\cap S_\mu^g$ for some $g\in S_n$. 
Hence it suffices to show that $|S_{\lambda_i}\cap S^g_{\lambda_j}| \geq p^r$ if and only if $i = j = 1$.
Suppose $(i, j) \neq (1, 1)$. 
Then 
$|S_{\lambda_i}\cap S^g_{\lambda_j}|\leq \min\{|S_{\lambda_i}|, |S^g_{\lambda_j}| \}<p^r$ follows immediately by the stability assumption.

It remains to show $|S_{\lambda_1}\cap S^g_{\mu_1}| \geq p^r$. This follows from applying the inclusion-exclusion principle to the union of $S_{\lambda_1}$ and $S^g_{\mu_1}$. 
By the stability assumptions,
    \begin{align*}
        n\geq|S_{\lambda_1} \cup S^g_{\mu_1}| &= |S_{\lambda_1}| + |S^g_{\mu_1}| - |S_{\lambda_1} \cap S^g_{\mu_1}| \\
        &\geq 2 \cdot \left(\frac{n+p^r}{2}\right) - | S_{\lambda_1} \cap S^g_{\mu_1}| \\
        &= n+p^r - |S_{\lambda_1} \cap S^g_{\mu_1}|,
    \end{align*}
hence $|S_{\lambda_1} \cap S^g_{\mu_1}| \geq p^r$. 
\end{proof}

\begin{proof}[Proof of Theorem 6.2.1]
By~Proposition~\ref{prop:oti_on_permutation_modules}(2), the essential image of the functor
\[\operatorname{Perm}^{\st}_{n, p^r} \xto{\Phi_{\varphi}} \rep S_{n-p^r}.\]
is exactly $\operatorname{Perm}^{\trun}_{n-p^r, p^r}$. Hence it suffices to show that if $\lambda, \mu$ are $p^r$-stable, then $\Phi_{\varphi}$ induces an isomorphism
    \[\Hom_{S_n}(M^\lambda, M^\mu) \xto{\sim} \Hom_{S_{n-p^r}}(\Phi_{\varphi}(M^\lambda), \Phi_{\varphi}(M^\mu)).\]
Proposition \ref{prop:oti_on_permutation_modules}(1) shows this is equivalent to
    \[
        \Hom_{S_n}(\Bbbk[\tab^\lambda], \Bbbk[\tab^\mu]) \xto{\sim} \Hom_{S_{n-p^r}}(\Bbbk[(\tab^\lambda)^H], \Bbbk[(\tab^\mu)^H]).
    \]
Permutation modules are self-dual, so we need to check that
    \[
        \Hom_{S_n}(\mathbf{1}, \Bbbk[\tab^\lambda \times \tab^\mu]) \xto{\sim} \Hom_{S_{n-p^r}}(\mathbf{1}, \Bbbk[(\tab^\lambda \times \tab^\mu)^H]).
    \]
This follows because the inclusion $(\tab^\lambda \times \tab^\mu)^H\sub (\tab^\lambda \times \tab^\mu)$ induces a bijection of orbits 
    \[S_{n-p^r} \backslash (\tab^\lambda \times \tab^\mu)^H
    \xto{\sim} S_{n}\backslash (\tab^\lambda \times \tab^\mu).
    \]
Indeed, Lemma~\ref{prop:components_of_hom_are_quasistable} shows that $\tab^\lambda \times \tab^\mu$ decomposes as a disjoint union $\coprod_i \tab^{\nu_i}$ for $p^r$-quasistable partitions $\nu_i$. 
By Proposition~\ref{prop:oti_on_permutation_modules}(3), each inclusion $(\tab^{\nu_i})^H\sub \tab^{\nu_i}$ induces a bijection of orbits $S_{n-p^r} \backslash (\tab^{\nu_i})^H
\xto{\sim} S_n \backslash \tab^{\nu_i}$. 
\end{proof}

\section{OTI on Polynomial functors}\label{section:polynomial_functors}

One of the basic examples of a category admitting an action of the degenerate Heisenberg algebra is the category of strict polynomial functors. In this section we recall this example (Section \ref{section poly functors defn}), and study the OTI functor in this setting (Section \ref{section OTI on pol}). This example is related to representations of symmetric groups via Schur--Weyl duality, and in Section \ref{section SW duality} we prove the compatibility of Schur--Weyl duality with the OTI functor. This will be used in Section \ref{section:periodicity_result} to prove periodicity results for symmetric groups.

    \subsection{Polynomial functors}\label{section poly functors defn}
    
     The following  is due to Friedlander and Suslin \cite{FS97}. 
    \begin{definition}
    The category $\PP ol$  of strict polynomial functors consists of functors $M:\operatorname{vec}_{\Bbbk}\to\operatorname{vec}_{\Bbbk}$,
    such that for any $V,W \in \operatorname{vec}_{\Bbbk}$:
    \begin{enumerate}
        \item The morphism $\Hom_{\Bbbk}(V,W)\to\Hom_{\Bbbk}(M(V),M(W))$ is polynomial for all $V,W$.
        \item The degree of the map $\End(V)\to\End(M(V))$ is uniformly bounded over all $V \in \operatorname{vec}_{\Bbbk}$.
    \end{enumerate}  
    Morphisms in $\PP ol$ are given by natural transformations. 
    We say that $M\in\PP ol$ has \emph{degree} $d\in\N$ if for any vector spaces $V,W$, the morphism $$\Hom_{\Bbbk}(V,W)\to\Hom_{\Bbbk}(M(V),M(W)),$$ is homogeneous of degree $d$. We write $\PP ol_d$ for the full subcategory of $\PP ol$ of degree $d$ polynomial functors.
    \end{definition}
    \begin{example}
    The following are objects in $\PP ol_d$:
    \begin{enumerate}
        \item The $d$-th tensor power functor $V\mapsto \bigotimes^dV$.
        \item The $d$-th symmetric power functor $V\mapsto S^d(V)$.
        \item The $d$-th divided power functor $V\mapsto \Gamma^d(V)\coloneqq (\bigotimes^d V)^{S_d}$.
        \item The $d$-th exterior power functor $V\mapsto \Lambda^d(V)$.
        \item Any Schur functor corresponding to a partition of $d$.
    \end{enumerate}
    \end{example}
    The category $\PP ol$ is $\Bbbk$-linear and abelian, and has a natural decomposition:
    \[
    \PP ol=\bigoplus\limits_{d\in\N}\PP ol_d.
    \]
    There is a  monoidal product on $\PP ol$ given by
    $
    (M\otimes N)(V):=M(V)\otimes N(V).
    $
    Observe that $\otimes$ is an exact bifunctor on $\PP ol$, and $\otimes:\PP ol_{d_1}\times \PP ol_{d_2}\to \PP ol_{d_1+d_2}$.

     Recall that a rational representation $M \in \rep GL(V)$ is polynomial of degree $d$ if its matrix coefficients are homogeneous polynomials of degree $d$.    
    \begin{lemma}\label{lemma poly fun poly rep}
    Write $\operatorname{Pol}_dGL(V)$ for the full subcategory of $\rep GL(V)$ consisting of polynomial representations of degree $d$.
    \begin{enumerate}
        \item  \cite{FS97} Suppose that $\dim V\geq d$.  Then the functor:
        \[
        M\mapsto M(V)
        \]
        gives an equivalence of categories $\PP ol_{d} \to \operatorname{Pol}_dGL(V)$.
        \item \cite{HY1} If $n>d$, then the functor $R_0:\operatorname{Pol}_dGL(\Bbbk^n)\to\operatorname{Pol}_dGL(\Bbbk^{n-1})$, $$M \mapsto M^{GL_1},$$ given by restriction to $GL_{n-1}\times GL_1$ and taking $GL_1$ invariants, is an equivalence such that the following diagram commutes:
        \[
        \xymatrix{
        \PP ol_d\ar[r]^{\sim} \ar[dr]_{\sim} & \operatorname{Pol}_dGL_n\ar[d]^{R_0} \\
        & \operatorname{Pol}_{d}GL_{n-1}
        }
        \]
    \end{enumerate}
    \end{lemma}
We recall that $\PP ol_d$ is a highest weight category~\cite{Don93}, where the standard objects are Weyl functors $\Delta_\lambda$ and costandard objects are Schur functors $\nabla_\lambda$.  
The tilting objects for this highest weight structure are summands of tensor products $\Lambda^{k_1}\otimes \cdots \otimes \Lambda^{k_j}$.

For a partition $\lambda$ of $d$, we write $L_\lambda \in \PP ol_d$ for the simple object 
and $T_\lambda \in \PP ol_d$ for the indecomposable tilting object of highest weight $\lambda$.
We write $\Tilt_d$ for the full additive subcategory of $\PP ol_d$ consisting of tilting objects.

We also recall that $\PP ol$ admits a categorical Heisenberg action in the sense of~Section~\ref{sec:Heis-act}, see~\cite{HY}. Here the functor $F$ is given by  
    \[
    F(M)(V):=M(V)\otimes V,
    \]
    while $E$ is given by
    \begin{equation*}\label{eqn heis act poly}
    E(M)(V):=M(V\oplus\Bbbk)^{1}. 
    \end{equation*}
    In the above, we consider the embedding $GL(\Bbbk)\sub GL(V\oplus\Bbbk)$ induced by the natural map $\Bbbk \to V\oplus\Bbbk$, and we take the eigenspace of weight $1$ with respect to this copy of $GL(\Bbbk)\cong\mathbb{G}_m$.  Observe that $F$ increases the degree of a polynomial functor by one, while $E$ decreases the degree by one.  
    
    The natural transformation $T:E^2\to E^2$ is induced by the natural map $V\oplus\Bbbk\oplus\Bbbk\to V\oplus\Bbbk\oplus\Bbbk$ which permutes the last two entries.
    
    \subsection{OTI functors} \label{section OTI on pol}
    Choose a CF functor $\varphi:\rep H\to\As$ for $H\sub S_{p^r}$.  In this section we will study the corresponding OTI functor 
    $$
    \Phi_\varphi:\PP ol_d \to \PP ol_{d-p^r}\boxtimes \As.
    $$
    By definition, we have
    \begin{align*}
        \Phi_\varphi(M)(V)&= \varphi(E^{p^r}(M)(V)) 
        =\varphi(M(V\oplus \Bbbk^{p^r})^{\overbrace{\scriptsize \scriptstyle 1,1,\dots,1}^{p^r}}),
    \end{align*}
    where 
    we are taking the $(1,\hdots,1)$ weight space with respect to the action of $\underbrace{\mathbb{G}_m\times\cdots\times\mathbb{G}_m}_{p^r \text{copies}}$. 
    
    \begin{remark}
        The OTI functor induced by $F$ instead of $E$ is given by 
        \[
        M\mapsto (V\mapsto M(V)\otimes V^{(r)})\boxtimes \varphi(\mathbf{1}),
        \]
        as in~Example~\ref{example glv}.   
        Notice that this functor only depends on $\varphi(\mathbf{1})$, which is not the case for the functor $\Phi_\varphi$ above.
    \end{remark}

   \begin{example}
   We evaluate the OTI functor on some basic examples of polynomial functors.   In the following we assume that $d\geq p^r$.
\begin{enumerate}
       \item  First, we show that $\Phi_{\varphi}(\otimes^d)=0$.  Set $M=\otimes^d$. Then we have that 
        \begin{align}\label{eq:tensex1}
       E^{p^r}(M)(V) \cong \Ind^{S_d}_{S_{d-p^r}}(V^{\otimes (d-p^r)}).
        \end{align}
        Note that the $S_d$-module structure in \ref{eq:tensex1} comes from the natural action of $S_d$ on~$M$.  On the other hand, there is also an $S_{p^r}$-action on $E^{p^r}$, 
        and via \eqref{eq:tensex1} this corresponds to right multiplication by $S_{p^r}$ (embedded in $S_d$ as the permutations of $d-p^r+1,\hdots,d$). More precisely, for $g \in S_d$, $\xi \in V^{\otimes (d-p^r)}$ and $h \in S_{p^r}$,
        $$
        (g\otimes \xi)\cdot h = gh \otimes \xi.
        $$
        Note this action is free (indeed if $K \times H \leq G$ then right multiplication by $H$ on
        $\Ind_K^G(U)$ is free for any $K$-module $U$).  
        Therefore the CF functor $\varphi$ must kill  $E^{p^r}(M)(V)$.
       \item For symmetric powers $S^d$, we show that $\Phi_{\varphi}(S^d)=S^{d-p^r}\boxtimes \varphi(\mathbf{1})$. 
       We know
       \begin{align*}\label{eq:tensex}
       E(S^d)(V) = S^d(V \oplus \Bbbk)^1 \cong \bigoplus_{a+b=d}S^a(V)\otimes S^b(\Bbbk)^1
        \end{align*}
        Observe that $S^b(\Bbbk)^1=0$ unless $b=1$. Hence we have $E(S^d)=S^{d-1}$, and iterating we get that $E^{p^r}(S^d)=S^{d-p^r}$. Since $S_{p^r}$ acts trivially, the result follows.
        
       \item Similarly, we have that $\Phi_{\varphi}(\Gamma^d)=\Gamma^{d-p^r}\boxtimes \varphi(\mathbf{1})$. Indeed, as in (2) we have  $\Gamma^d(V\oplus \Bbbk)\cong \bigoplus_{a+b=d}\Gamma^a(V)\otimes \Gamma^b(\Bbbk)$ and $\Gamma^b(\Bbbk)^1=0$ unless $b=1$. Therefore $E^{p^r}(\Gamma^d)=\Gamma^{d-p^r}$, and $S_{p^r}$ acts trivially. 
       
       \item Finally, $\Phi_{\varphi}(\Lambda^d)=\Lambda^{d-p^r}\boxtimes\varphi(\operatorname{sgn}_H)$.  Indeed, $\Lambda$ satisfies the same compatibility with $\oplus,\otimes$ as in (2) and (3), and $\Lambda^b(\Bbbk)=0$ for $b>1$.  It follows that $E^{p^r}(\Lambda^{d})=\Lambda^{d-p^r}$, and $S_{p^r}$ acts by the sign representation.
   \end{enumerate}
   \end{example}

For the remainder of~Section~\ref{section:polynomial_functors}, we assume that $\As$ is a locally finite abelian category. 
Recall that given a simple object $L$ in $\As$, we identify the subcategory $\mathscr{C}\boxtimes\langle L\rangle$ of $\Cs\boxtimes\As$ with~$\mathscr{C}$ (see Section \ref{section periodic equiv perm mod}).

\begin{lemma}\label{lemma tilts to tilts}
The OTI functor $\Phi_{\varphi}:\PP ol_d \to \PP ol_{d-p^r}\boxtimes\As$ satisfies $$\Phi_{\varphi}(\Tilt_d) \subseteq \Tilt_{d-p^r}\boxtimes \langle\varphi(\sgn_H)\rangle.$$
Thus if $\varphi(\sgn_H)$ is simple, $\Phi_{\varphi}$ induces a functor $\Tilt_d \to \Tilt_{d-p^r}$, which is independent of the choice of  $\varphi$. 
\end{lemma}
\begin{proof}
Since $\Phi_\varphi$ is additive and any tilting object in $\PP ol_d$ is a summand of a tensor power of the form $\Lambda^{k_1}\otimes \cdots \otimes \Lambda^{k_j}$, the result follows if we prove 
\begin{equation}\label{eq:tilt_lemma}
\Phi_\varphi(\Lambda^{k_1}\otimes\cdots \otimes \Lambda^{k_j}) \cong \left(\bigoplus_{k_i\geq p^r} \Lambda^{k_1}\otimes\cdots \otimes\Lambda^{k_i-p^r} \cdots \otimes \Lambda^{k_j}\right)\boxtimes\varphi(\sgn_H).
\end{equation}
First, we show that $E$ is a derivation, i.e. $E(M\otimes N) \cong (E(M)\otimes N) \oplus (M\otimes E(N))$. Indeed, observe that 
    \begin{align*}
        E(M\otimes N)(V) \cong \left(M(V\oplus \Bbbk)^1\otimes N(V\oplus \Bbbk)^0 \right) \oplus
        \left(M(V\oplus \Bbbk)^0\otimes N(V\oplus \Bbbk)^1 \right), 
    \end{align*}
    and by part (2) of Lemma \ref{lemma poly fun poly rep}, we have $M(V\oplus \Bbbk)^0 \cong M(V)$. 

    Next, consider $M=M_1\otimes \cdots \otimes M_j$. By the derivation property we have 
    \begin{align*}
       E^n(M) \cong \bigoplus_{n_1+\cdots+n_j=n}(E^{n_1}(M_1)\otimes \cdots \otimes E^{n_j}(M_j))^{\binom{n}{n_1,...,n_j}}.
    \end{align*}
    Specialising to the case $M=\Lambda^{k_1}\otimes \cdots \otimes \Lambda^{k_j}$ and $n=p^r$, summands in $E^{p^r}(M)$ look like
    \[(\Lambda^{k_1-n_1}\otimes \cdots \otimes \Lambda^{k_j-n_j})^{\binom{p^r}{n_1,...,n_j}}.
    \]
    Letting $S_{n_i}$ act by sign on $\Lambda^{k_i-n_i}$, we have that this is isomorphic to 
    \begin{align}\label{eqn lambdas}
    \Ind_{S_{n_1}\times\cdots\times S_{n_j}}^{S_{p^r}}(\Lambda^{k_1-n_1}\otimes \cdots \otimes \Lambda^{k_j-n_j}).\end{align}
    The terms where $n_i \neq p^r$ for all $i$ will be killed by $\varphi$, because $S_{n_1}\times\cdots\times S_{n_j}$ is a non-transitive subgroup.  On the remaining terms, $S_{p^r}$ acts by sign, so we obtain~\eqref{eq:tilt_lemma}.
\end{proof}

\subsection{Schur--Weyl duality}\label{section SW duality}
We consider the Schur--Weyl duality functor
\[
\FF:\PP ol_d\to\operatorname{Rep}S_d,\hspace{1em} M\mapsto \Hom_{\PP ol_d}(\otimes^d,M).
\]
We recommend \cite{Hemm} for basic properties of Schur--Weyl duality, including the following.
The functor $\FF$ is exact and admits a right adjoint which we denote by $\GG$. Explicitly, we have $\GG(N):V \mapsto \Hom_{S_d}(V^{\otimes d},N)$.
Note that $\FF\GG\cong \id$. 

Let $\lambda^t$ denote the dual of the partition $\lambda$. 
On tilting objects we have $\FF(T_\lambda) \cong Y^{\lambda^t}\otimes \sgn$, and $\FF$ is fully faithful on $\Tilt_d$.
We also have that $\FF(\Lambda^{\lambda_1}\otimes\cdots\otimes\Lambda^{\lambda_j}) \cong M^{\lambda^t}\otimes\sgn$, and $\FF$ maps simple objects to simple objects (or zero). 

It follows that $\FF$ induces an equivalence
\begin{align}\label{eq:tilt_to_young}
\FF:\Tilt_d \xto{\sim} \operatorname{sgnYoung}_d.
\end{align}

\begin{lemma}\label{lem:OTI-SWdual}
    The OTI functor commutes with Schur--Weyl duality, i.e. the following diagram commutes up to natural isomorphism:
 \[\begin{tikzcd}
\PP ol_d \arrow[r,"E^{p^r}"] \arrow[d, "\FF"] & \PP ol_{d-p^r}\boxtimes \rep S_{p^r} \arrow[r,"\mathbf{1}\boxtimes \varphi"] \arrow[d,"\FF\boxtimes \mathbf{1}"] & \PP ol_{d-p^r} \boxtimes \As \arrow[d,"\FF\boxtimes \mathbf{1}"] \\
\operatorname{Rep}S_d \arrow[r,"E^{p^r}"] & \operatorname{Rep}S_{d-p^r} \boxtimes \rep S_{p^r} \arrow[r,"\mathbf{1}\boxtimes\varphi"]  &\operatorname{Rep}S_{d-p^r} \boxtimes \As
\end{tikzcd}\] 
In particular, we have a natural isomorphism $(\FF\boxtimes 1)\Phi_{\varphi}\simeq \Phi_{\varphi}\FF$.
\end{lemma}

\begin{proof}
Recall that the categories $\PP ol$ and $\mathcal{S}ym=\bigoplus\limits_{n\in\N}\rep S_n$ carry categorical Heisenberg actions. In \cite{HY} it is shown that $\FF$ is a morphism of categorical actions, from which it easily follows that the left square commutes. 
On the other hand, because $\FF$ is exact, by Lemma~\ref{lemma additive functors extend}(2) we have $(\FF\boxtimes 1)(1\boxtimes\varphi)\cong(1\boxtimes\varphi)(\FF\boxtimes 1)$, and hence the right square commutes as well. 
\end{proof}

Write $\Tilt_{d,p^r}^{\st}$ (resp.~$\Tilt_{d,p^r}^{\trun}$) for the full additive subcategory of $\Tilt_d$ generated by indecomposable tilting modules $T(\lambda)$, where $\lambda^t$ is $p^r$-stable (resp.~$p^r$-truncated).  

\begin{prop}\label{lemma oti tilts equiv stable poly} If $\varphi(\sgn_H)$ is simple,
$\Phi_{\varphi}$ induces an equivalence $$\Tilt^{\st}_{d,p^r}\longrightarrow\Tilt^{\trun}_{d-p^r,p^r}.$$
\end{prop}

\begin{proof}
By \eqref{eq:tilt_to_young}, $\FF$ induces equivalences 
\[\Tilt^{\st}_{d,p^r} \xto{\sim} \operatorname{sgnYoung}^{\st}_{d,p^r}, \quad \text{and} \quad
\Tilt^{\trun}_{d-p^r,p^r} \xto{\sim} \operatorname{sgnYoung}^{\trun}_{d-p^r,p^r}.\]
Hence by Lemma \ref{lem:OTI-SWdual}, the result follows from~Corollary~\ref{rem sgnyoung} and~Lemma~\ref{lemma tilts to tilts} .
\end{proof}

We write $\PP ol_{d,p^r}^{\st}$ (resp.~$\PP ol_{d,p^r}^{\trun}$) for the Serre subcategory of $\PP ol_d$ generated by the simple objects $L_\lambda$, where $\lambda^t$ is $p^r$-stable (resp.~$p^r$-truncated).  By~\cite[Theorem 3.17]{BSW} and~Lemma~\ref{lemma:p^r_stability_open}, these define highest weight subcategories of $\PP ol_{d}$.

\begin{lemma}\label{lemma oti st to tr}
    We have $E^{p^r}(\PP ol_{d,p^r}^{\st})\sub \PP ol_{d-p^r,p^r}^{\trun}\boxtimes \rep S_{p^r}$.  Hence $\Phi_{\varphi}(\PP ol_{d,p^r}^{\st})\sub \PP ol_{d-p^r,p^r}^{\trun}\boxtimes\As$.
\end{lemma}
\begin{proof}
    Recall that $E$ is exact and tilting objects in $\PP ol_{d,p^r}^{\st}$ generate $\PP ol_{d,p^r}^{\st}$ as an abelian category. 
    Therefore it suffices to show that $E^{p^r}(\Lambda^{\lambda_1}\otimes \cdots \otimes \Lambda^{\lambda_j})$ lies in $\PP ol_{d,p^r}^{\trun}\boxtimes \rep S_{p^r}$ whenever $\lambda$ is $p^r$-stable.  By~\eqref{eqn lambdas}, we see that $E^{p^r}(\Lambda^{\lambda_1}\otimes \cdots \otimes \Lambda^{\lambda_j})$ is a direct sum of tiltings $T_{\mu^t}$, where $\mu\leq \lambda-p^re_1$.  Since $\lambda-p^re_1$ is $p^r$-truncated by definition, we obtain that $\mu$ must also be $p^r$-truncated.  
\end{proof}
\section{Periodicity}\label{section:periodicity_result}
In this section, we prove Theorems \ref{thm A intro} and \ref{thm intro C}. We first prove analogous equivalences for polynomial functors, exploiting the highest weight structure.  The equivalence for representations of symmetric groups is deduced using Schur--Weyl duality.  Theorem \ref{thm A intro} is proven by extending the equivalence on tiltings (Proposition \ref{lemma oti tilts equiv stable poly}).  We use the exactness property (Lemma \ref{lemma exactness property}) for CF functors given by semisimplification with respect to (generic) shifted cyclic subgroups.   For more general CF functors, we deduce the result using a technical lemma on highest weight categories.

\subsection{Highest weight categories}
We begin by recalling some background on highest weight categories, for more detail see ~\cite[Section 3]{K} (together with~\cite[Remark 8.1.7]{Kbook}), or more generally~\cite[Section 8.1]{Kbook}.

Let $(\Cs,\Lambda)$ be a highest weight category. Note that we always assume the poset $\Lambda$ is finite.
For $\lambda\in \Lambda$, we will denote the corresponding irreducible, standard and costandard object by $L_\lambda$, $\De_\lambda$ and $\Na_\lambda$, respectively.
If $\Lambda_0\subset\Lambda$, we write $\Cs_{\Lambda_0}$ for the Serre subcategory of $\Cs$ generated by the irreducibles $L_\mu$ with $\mu\in \Lambda_0$. 

Similarly, we
write $\Filt(\De_{\Lambda_0})$ for the full subcategory of objects
$X$ in $\Cs$ that admit a finite filtration
$0=X_0\subseteq X_1\subseteq\ldots\subseteq X_t=X$ such that each
factor $X_i/X_{i-1}$ is isomorphic to some $\De_\mu$ with $\mu\in \Lambda_0$.
In particular, $\Filt(\De):=\Filt(\De_\Lambda)$.

Now fix a maximal element $\lambda\in\Lambda$ and set $\Lambda_0=\Lambda\setminus \{\lambda\}$.
Recall that $\Cs_{\Lambda_0}$ is again a highest weight category, with standard objects $\{\De_\mu
\mid \mu\in\Lambda_0\}$ and costandard objects $\{\Na_\mu
\mid \mu\in\Lambda_0\}$. The inclusion functor $i_*:\Cs_{\Lambda_0}\to \Cs$ has a left (resp., right) adjoint sending an object of $\Cs$ to the largest quotient (resp., subobject) that belongs to~$\Cs_{\Lambda_0}$. 
This results in a recollement of abelian catetories,
$$\begin{tikzcd}
      \Cs_{\Lambda_0}
      \arrow[hook, r,"i_*" ] 
      &\Cs
      \arrow[l,shift left = 4.3, "i^{!}"'] \arrow[l,shift right= 4.3, "i^{*}"'] \arrow[twoheadrightarrow, r, "j^*"] 
      &\Cs_{\lambda}
      \arrow[l,shift left = 4.3, "j_{*}"'] \arrow[l,shift right= 4.3, "j_{!}"']
    \end{tikzcd}$$
with $\Cs_\lambda:=\Cs_{\{\lambda\}}$ equivalent to $\vec$. Note that the left (resp., right) adjoint of $j^*$ sends $\Bbbk$ to the projective (resp., injective) object $\De_\lambda$ (resp., $\Na_\lambda$).

Restriction to $\Filt(\De)$ gives a diagram of exact functors
\[
\begin{tikzcd}
  \Filt(\De_{\Lambda_0}) \arrow[hook,
  yshift=-0.75ex, "i_*"']{rr}&&\Filt(\De)
  \arrow[twoheadrightarrow,yshift=0.75ex, "i^*"']{ll}
  \arrow[twoheadrightarrow, yshift=-0.75ex, "j^*"']{rr} &&\Filt(\De_\lambda)\simeq \vec\, .
  \arrow[hook,yshift=0.75ex, "j_!"']{ll}
\end{tikzcd}
\]
such that for every $M$ in $\Filt(\De)$, the unit and counit fit into a short exact sequence  
\[
0\to j_!j^*M\to M\to i_*i^*M\to 0.
\]
Of course, dual statements hold for $\Filt(\Na)$.

\subsection{Additive functors on highest weight categories}\label{sec:add on hwc}
In this section, we will consider an additive functor $\phi:\Cs\to\Cs'$ from a highest weight category $(\Cs, \Lambda)$ to an abelian category~$\Cs'$.
Note that highest weight categories have finite global dimension.  From $\phi$ we will derive three triangulated functors $\D^b(\Cs)\to \D^b(\Cs')$ on the bounded derived categories:
\begin{itemize}
    \item The \emph{tilting derived functor} $\tilde\phi$ is the composition
        \[
        \D^b(\Cs)\simeq \K^b(\Tilt(\Cs))\xrightarrow{\phi}
        \K^b(\Cs')\to \D^b(\Cs'),
        \]
        where $\K^b(-)$ denotes the bounded homotopy category.  To compute $\tilde\phi(M)$ for $M\in \Cs$, apply $\phi$ degreewise to a finite complex of tilting modules quasi-isomorphic to~$M$.

    \item The \emph{left derived functor} $\L\phi$ is defined as the left Kan extension of 
        $\Cs\xrightarrow{\phi} \Cs'\to \D^b(\Cs')$ along $\Cs\to \D^b(\Cs)$. 
        In particular, $\L\phi$ is the composition 
        \[
        \D^b(\Cs)\simeq \K^b(\proj(\Cs))\xrightarrow {\phi}\K^b(\Cs')\to \D^b(\Cs').
        \]
        To compute $\L\phi (M)$ for $M\in \Cs$, replace $M$ by a projective resolution and apply $\phi$ degreewise.

    \item The \emph{right derived functor} $\R\phi$ is the right Kan extension of 
        $\Cs\xrightarrow{\phi} \Cs'\to \D^b(\Cs')$ along $\Cs\to \D^b(\Cs)$. 
        In particular, $\R\phi$ is the composition 
        \[
        \D^b(\Cs)\simeq \K^b(\inj(\Cs))\xrightarrow {\phi}\K^b(\Cs')\to \D^b(\Cs').
        \]
        To compute $\R\phi (M)$ for $M\in \Cs$, replace $M$ by an injective resolution and apply $\phi$ degreewise.
\end{itemize}

Note that we get natural transformations $\L\phi\Rightarrow \phi \Rightarrow \R\phi$ of functors $\Cs\to \D^b(\Cs')$ as usual, where we view $\phi$ as a functor $\Cs\to\D^b(\Cs')$ via composition with the embedding $\Cs'\hookrightarrow\D^b(\Cs')$.
Since objects in $\Filt(\De)$ have right resolutions by tilting objects, we furthermore have a natural transformation $\phi\Rightarrow \tilde \phi$ of functors $\Filt(\De)\to\D^b(\Cs')$.
Dually, there is a natural transformation $\tilde \phi\Rightarrow\phi$ of functors $\Filt(\Na)\to\D^b(\Cs')$.

\begin{lemma}\label{lemma hw cat nonsense}
Let $(\Cs, \Lambda)$ be a highest weight category and $\Cs'$ an abelian category.
Let $\lambda\in \Lambda$ be maximal and set $\Lambda_0=\Lambda\setminus\{\lambda\}$.
Consider an additive functor $\phi:\Cs \to \Cs'$ and suppose
\begin{enumerate}
\item $\phi$ is exact when restricted to $\Cs_{\Lambda_0}$,
\item the natural maps $\phi(\De_\lambda)\rightarrow\tilde \phi(\De_\lambda)$ and $\tilde\phi(\Na_\lambda)\rightarrow \phi(\Na_\lambda)$ are isomorphisms.
\end{enumerate}
Then the functors $\L\phi\simeq\tilde\phi\simeq\R\phi:\D^b(\Cs)\to \D^b(\Cs')$ are naturally isomorphic and 
the restriction $\tilde \phi:\Cs\to \D^b(\Cs')$ is a direct summand of $\phi$.
\end{lemma}

\begin{proof}
Consider the following diagram of exact functors
\[
\begin{tikzcd}
  \Filt(\De_{\Lambda_0}) \arrow[hook,
  yshift=-0.75ex, "i_*"']{rr}&&\Filt(\De)
  \arrow[twoheadrightarrow,yshift=0.75ex, "i^*"']{ll}
  \arrow[twoheadrightarrow, yshift=-0.75ex, "j^*"']{rr} &&\vec\, .
  \arrow[hook,yshift=0.75ex, "j_!"']{ll}
\end{tikzcd}
\]
We investigate the sequence of natural transformations $\L\phi\Rightarrow \phi\Rightarrow \tilde\phi$ on $\Filt(\De)$.
First note that $(\L\phi)j_!\Rightarrow\phi j_!$ and $\phi j_!\Rightarrow \tilde\phi j_!$ are both natural isomorphisms. Indeed, it suffices to check this on $j_!(\Bbbk)=\De_{\lambda}$. 
Since $\De_{\lambda}$ is projective in $\Cs$, clearly $\L\phi(\De_{\lambda})\rightarrow\phi(\De_{\lambda})$ is an isomorphism, whereas $\phi(\De_{\lambda})\rightarrow \tilde\phi(\De_{\lambda})$ is an isomorphism by assumption (2).

Next, we turn to the natural transformations $(\L\phi) i_*\Rightarrow\phi i_*$ and $\phi i_*\Rightarrow \tilde\phi i_*$.
The latter is an isomorphism, since $\phi i_*$ and $\tilde\phi i_*$ are both exact functors which agree on tilting objects in $\Cs_{\Lambda_0}$.
To compute $(\L\phi) i_*$, let $M\in\Filt(\De_{\Lambda_0})$ and write down a projective resolution $P_\bullet \twoheadrightarrow i_* M$ in $\Cs$.
Since projectives are in $\Filt(\De)$, we get a short exact sequence of complexes
\[\begin{tikzcd}
\cdots\arrow{r}
&j_!j^*P_1\arrow[hook]{d}\arrow{r}
&j_!j^*P_0\arrow{r}\arrow[hook]{d}
&0\arrow{r}\arrow{d}
&0\\
\cdots
\arrow{r}
&P_1\arrow[two heads]{d}\arrow{r}
&P_0\arrow{r}\arrow[two heads]{d}
&i_* M\arrow[equal]{d}\arrow{r}
&0\\
\cdots
\arrow{r}
&i_*i^*P_1\arrow{r}
&i_*i^*P_0\arrow{r}
&i_*M\arrow{r}
&0.
\end{tikzcd}\]
In fact, the rows are exact as well, since all syzygys of $P_\bullet \twoheadrightarrow i_* M$ remain in $\Filt(\De)$ by~\cite[Lemma 2.1.1]{los}, and $i_*i^*$ and $j_!j^*$ are exact functors there. Since the top row is a long exact sequence of projectives, it is null homotopic, hence $P_\bullet \to i_*i^* P_{\bullet}$ is a homotopy equivalence.
We can now compute 
$$\L\phi (i_*M)=\phi(P_\bullet)\simeq\phi(i_*i^*P_\bullet) = \phi(i_*M),$$
where the last equality follows because $\phi$ is exact on the subcategory $\Cs_{\Lambda_0}$.
We conclude that  $(\L\phi) i_*\Rightarrow\phi i_*$ is an isomorphism as well.

The composition $\L\phi\Rightarrow \phi\Rightarrow\tilde \phi$ applied to the exact sequence
$$0\to j_!j^*M\to M\to i_*i^*M\to 0$$
with $M$ in $\Filt(\De)$
gives us a morphism of triangles in $\D^b(\Cs')$,
\[\begin{tikzcd}
\L\phi(j_!j^*M)\arrow["\simeq"]{d}\arrow{r}
&\L\phi(M)\arrow{r}\arrow{d}
&\L\phi(i_*i^* M)\arrow["\simeq"]{d}\\
\tilde\phi(j_!j^*M)\arrow{r}
&\tilde\phi(M)\arrow{r}
&\tilde\phi(i_*i^*M)\ ,
\end{tikzcd}\]
hence $\L\phi\Rightarrow \tilde \phi$ is a natural isomorphism on $\Filt(\De)$.
In particular, the exact functors $\L\phi$ and $\tilde\phi$ are naturally isomorphic on $\proj(\Cs)$, so they are isomorphic on $\D^b(\Cs)$.

Dually, one can prove that the composition $\tilde \phi\Rightarrow\phi\Rightarrow\R\phi$ is an equivalence on $\Filt(\Na)$, hence $\R\phi$ and $\tilde\phi$ are isomorphic on $\D^b(\Cs)$.

Finally, we consider the composition $ \L\phi \Rightarrow\phi\Rightarrow \R\phi$ of functors $\Cs\to \D^b(\Cs')$.
On tilting objects, it factors as the composition of the isomorphisms $\L\phi \Rightarrow \tilde\phi$ and $\tilde\phi \Rightarrow \R\phi$. It follows that $\L\phi \Rightarrow\phi\Rightarrow \R\phi$ is an isomorphism of exact functors on tilting objects, and hence an isomorphism on~$\Cs$. We conclude that $\tilde\phi\simeq \L\phi \simeq \R\phi$ is a direct summand of~$\phi$.
\end{proof}

\begin{remark}\label{rem on tilting}
   Recall that an exact functor on a highest weight category $\Cs$ is entirely determined by its restriction to the subcategory of tilting objects in $\Cs$. In particular, if an exact functor $\phi_1:\Cs\to\Cs'$ and additive functor $\phi_2:\Cs\to\Cs'$ agree on tilting objects, it follows that $\phi_1$ agrees with the tilting derived functor $\t\phi_2:\D^b(\Cs)\to\D^b(\Cs')$ of $\phi_2$.
   We will use this freely throughout the remainder of~Section~\ref{section:periodicity_result}, and will often identify $\t\phi_2$ with its (co)restriction~$\phi_1$. 
\end{remark}

\subsection{The OTI functor associated to a shifted cyclic subgroup} 

Let $A\sub S_{p^r}$ denote a transitive elementary abelian $p$-group of order $p^r$. In other words, $A\cong C_p^r$ and acts transitively on $\{1,\dots,p^r\}$. Such a subgroup is unique up to conjugacy, and may be realised for instance by letting $C_p^r$ act by left multiplication on itself.

We will fix $\mathbf{a}=(a_1,a_2,\dots,a_n)\in\Bbbk^n$ with $\mathbf{a}\neq0$, and set $z_{\mathbf{a}}:=a_1x_1+\dots+a_nx_n\in \Bbbk A$, where $1-x_1,\dots,1-x_r$  are generators for $A$.  Then the subalgebra $R_{\mathbf{a}}\sub \Bbbk A$ generated by $z_{\mathbf{a}}$ is a shifted cyclic subgroup~(see~Section \ref{section OTI shifted cyclic}). 
We now impose the following genericity condition on $\mathbf{a}$: we assume that $R_{\mathbf{a}}$ does not lie in a group algebra $\Bbbk A'$ for any proper subgroup $A'\subset A$.  This will hold, for example,~if $a_1,\dots,a_n$ are linearly independent over~$\mathbb{F}_p$.  

\begin{definition}
    We will consider the composition $$\varphi_{\mathbf{a}}:\rep A\to R_{\mathbf{a}}\operatorname{--}\mod\to \Ver_p,$$ where the first functor is restriction along $R_{\mathbf{a}}\sub \Bbbk A$ and the second functor is defined as in~Equations~\eqref{formulaOTI} and~\eqref{eqn varphi}.
\end{definition}

We caution that our genericity assumption implies that $\varphi_{\mathbf{a}}$ is \emph{not} a monoidal functor.

\begin{lemma}
    If $\mathbf{a}$ is generic, then $\varphi_{\mathbf{a}}$ is a CF functor for $A\sub S_{p^r}$.
\end{lemma}
\begin{proof}
    It suffices to check that $\varphi_{\mathbf{a}}$ vanishes on induced modules $\Ind_{A'}^{A}(M)$ for proper subgroups $A'\subset A$.  This is equivalent to showing that $\Ind_{A'}^{A}(M)$ is projective over $R_{\mathbf{a}}$.  In the language of support theory, this means we want to show that $\mathbf{a}$ does not lie in the support of $\Ind_{A'}^{A}(M)$. However this follows from~\cite[Prop.~8.2.4]{Evens}.
\end{proof}

We write $$\Phi_{\mathbf{a}}:=\Phi_{\varphi_{\mathbf{a}}}:\PP ol_d\to\PP ol_{d-p^r}\boxtimes\Ver_p$$ for the OTI functor on polynomial functors corresponding to $\varphi_{\mathbf{a}}$ .  In~Section \ref{section SW duality} we introduced the highest weight subcategories $\PP ol^{\st}_{d,p^r}$ and $\PP ol^{\trun}_{d,p^r}$ of $\PP ol_d$, and we recall that $\Phi_{\mathbf{a}}(\PP ol^{\st}_{d,p^r})\sub \PP ol^{\trun}_{d-p^r,p^r}\boxtimes\Ver_p$ by~Lemma~\ref{lemma oti st to tr}.

\begin{thm}\label{thm periodicity for scs}
    The functor $\Phi_{\mathbf{a}}:\PP ol^{\st}_{d,p^r}\to \PP ol^{\trun}_{d-p^r,p^r}\boxtimes\Ver_p$ corestricts to an equivalence  $\Phi_{\mathbf{a}}:\PP ol^{\st}_{d,p^r}\xto{\sim}\PP ol^{\trun}_{d-p^r,p^r}$ of highest weight categories.
\end{thm}
\begin{proof}
   By Proposition \ref{lemma oti tilts equiv stable poly}, $\Phi_{\mathbf{a}}$ defines an equivalence between the corresponding categories of tilting modules (because $\sgn_A\cong\mathbf{1}$).  Recall that such an equivalence of tilting modules extends uniquely to an equivalence $\t\Phi_{\mathbf{a}}:\PP ol^{\st}_{d,p^r}\to \PP ol^{\trun}_{d-p^r,p^r}$ of highest weight categories (see~\cite[Corollary 1.6]{Har}, for instance).   By~Remark~\ref{rem on tilting}, $\t\Phi_{\mathbf{a}}$ is the (co)restriction of the tilting derived functor of $\Phi_{\mathbf{a}}$.

   We claim that if $\lambda^t$ is $p^r$-stable, the natural maps $\Phi_{\mathbf{a}}(\De_\lambda)\rightarrow\tilde \Phi_{\mathbf{a}}(\De_\lambda)$ and $\tilde\Phi_{\mathbf{a}}(\Na_\lambda)\rightarrow \Phi_{\mathbf{a}}(\Na_\lambda)$ are isomorphisms.
   Indeed, let 
   $$0\to \De_\lambda\to T_0\to T_1\to\cdots\to T_r\to 0,$$
   be a resolution of $\De_\lambda$ by tilting objects. By induction on $r$, repeatedly applying~Lemma~\ref{lemma exactness property} shows that $\Phi_{\mathbf{a}}(\De_\lambda)\to \Phi_{\mathbf{a}}(T_\bullet)$ is exact and that $\Phi_{\mathbf{a}}(\De_\lambda)\to\tilde \Phi_{\mathbf{a}}(\De_\lambda)$ is an isomorphism. 
   A similar argument shows $\tilde\Phi_{\mathbf{a}}(\Na_\lambda)\xto{\sim} \Phi_{\mathbf{a}}(\Na_\lambda)$. 
   
   By iteratively applying~Lemma~\ref{lemma hw cat nonsense}, we find a natural (split) transformation $\t\Phi_{\mathbf{a}}\Rightarrow \Phi_{\mathbf{a}}$ on $\PP ol^{\st}_{d,p^r}$. Consider the full subcategory of objects in $\PP ol^{\st}_{d,p^r}$ where this is an isomorphism. It contains all tilting modules and by Lemma \ref{lemma exactness property} it is closed under kernels of epimorphisms and cokernels of monomorphisms.  Hence it must be all of $\PP ol^{\st}_{d,p^r}$, which completes the proof.
\end{proof}

\subsection{Periodicity for polynomial functors} 
Now choose a subgroup $A\sub H\sub S_{p^r}$, so that in particular $z_{\mathbf{a}}\in\Bbbk H$.   For an $H$-module $V$, write 
\[
V^{\sgn_H}:=\{m\in V:h\cdot m=\sgn_H(h)m\text{ for all }h\in H\},
\]
and $V_{\sgn_H}$ for the $H$-module:
\[
V_{\sgn_H}=\frac{V^{\sgn_H}}{V^{\sgn_H}\cap \sum\limits_{h\in H}\im(h-\sgn_H(h))}.
\]
Note that $V_{\sgn_H}$ is isomorphic to a direct sum of copies of $\sgn_H$, and we have a decomposition $V\cong V_{\sgn_H}\oplus \bigoplus\limits_{\alpha}N_{\alpha}$, where each $N_{\alpha}$ is an indecomposable $H$-module not isomorphic to~$\sgn_H$.  The following lemma is easy.

\begin{lemma}
    The functor $(-)_{\sgn_H}:\rep H\to\rep H$ is a CF functor.
\end{lemma}

We will fix an exact and conservative forgetful functor $U:\PP ol_d\boxtimes\rep H\to \rep H$, which is constructed either by using~Lemma~\ref{lemma poly fun poly rep}(1) or by choosing a $\Bbbk$-coalgebra $C$ and an equivalence $\PP ol_d\boxtimes\rep H\xrightarrow{\simeq} C\operatorname{--}\operatorname{comod}_{\rep H}$ as in Section~\ref{deligne tensor}. 
We then have a commuting diagram 
\[\xymatrix{
   \PP ol_d\boxtimes\rep H
   \ar[rr]^{1\boxtimes (-)_{\sgn_H}} 
   \ar[d]^{U} 
   && \PP ol_d\boxtimes\rep H
   \ar[d]^{U}\\
   \rep H
   \ar[rr]^{(-)_{\sgn_H}}  
    && \rep H.}
   \]
We will look at the OTI functor
$$\Phi_{\sgn_H}:\quad \PP ol^{\st}_{d,p^r}\to\PP ol_{d-p^r,p^r}^{\trun}\boxtimes\rep H$$
given by $M\mapsto (1\boxtimes (-)_{\sgn_H})(E^{p^r}M)$.
By~Equations~\eqref{eq:tilt_lemma} and~\eqref{eqn lambdas}, the exact functor 
$$\Phi_{\mathbf{a}}\boxtimes\sgn_H:\ \ \PP ol^{\st}_{d,p^r}\xto{\Phi_{\mathbf{a}}} 
\PP ol^{\trun}_{d-p^r,p^r}
\xto{-\boxtimes\sgn_H}
\PP ol^{\trun}_{d-p^r,p^r}\boxtimes \langle\sgn_H\rangle$$ coincides with $\Phi_{\sgn_H}$ on the subcategory $\Tilt^{\st}_{d,p^r}$.

By~Remark~\ref{rem on tilting} and Lemma \ref{lemma tilts to tilts}, we may identify $\Phi_{\mathbf{a}}\boxtimes\sgn_H$ with the tilting derived functor of $\Phi_{\sgn_H}$. In particular, we have a natural transformation $\Phi_{\sgn_H}\Rightarrow \Phi_{\mathbf{a}}\boxtimes\sgn_H$ of functors $\Filt(\De_{\{\lambda^t\mid\lambda \text{ is } p^r-\text{ stable}\}})\to\PP ol^{\trun}_{d-p^r,p^r}\boxtimes \rep H$ and a natural transformation $\Phi_{\mathbf{a}}\boxtimes\sgn_H\Rightarrow \Phi_{\sgn_H}$ of functors $\Filt(\Na_{\{\lambda^t\mid\lambda \text{ is } p^r-\text{ stable}\}})\to\PP ol^{\trun}_{d-p^r,p^r}\boxtimes \rep H$.

\begin{remark}\label{rem sgn split} 
Note that $$U\Phi_{\sgn_H}(M)\cong (UE^{p^r}(M))_{\sgn_H}\cong\varphi_{\mathbf{a}}(UE^{p^r}(M))_{\sgn_H}\otimes \sgn_H$$ 
is a direct summand of 
$$U(\Phi_{\mathbf{a}}\boxtimes\sgn_H)(M)\cong \varphi_{\mathbf{a}}U E^{p^r}(M)\otimes \sgn_H$$ 
for every $M\in\PP ol^{\trun}_{d-p^r,p^r}$.
\end{remark}

\begin{lemma}\label{lem iso on stand}
Suppose that $\lambda$ is a $p^r$-stable partition.  
Then the natural maps $\Phi_{\sgn_H}(\De_{\lambda^t})\rightarrow\Phi_{\mathbf{a}}(\De_{\lambda^t})\boxtimes\sgn_H$ and $\Phi_{\mathbf{a}}(\Na_{\lambda^t})\boxtimes\sgn_H\rightarrow \Phi_{\sgn_H}(\Na_{\lambda^t})$  are isomorphisms.
\end{lemma}

\begin{proof}
   We prove this for standard modules, with the statement for costandard modules being dual.  
   Since $U$ is conservative, it suffices to show that 
   the natural transformation $U\Phi_{\sgn_H}\Rightarrow U(\Phi_{\mathbf{a}}\boxtimes\sgn_H)$ is an isomorphism when evaluated at $\De_{\lambda^t}$.

   Since $\Phi_{\mathbf{a}}$ is an equivalence of highest weight categories (see~Theorem~\ref{thm periodicity for scs}), it is clear that $\Phi_{\mathbf{a}}(\De_{\lambda^t})\boxtimes\sgn_H=\Delta_{(\lambda-p^r e_1)^t}\boxtimes\sgn_H$.
   Next, recall that we have a short exact sequence 
   $$0\to \Delta_{\lambda^t} \to \Lambda^{\lambda_1}\boxtimes\cdots\boxtimes\Lambda^{\lambda_\ell}\to M\to 0$$
   with $M\in \Filt(\De_{<\lambda^t})$.
   By~Equation~\eqref{eqn lambdas}, $E^{p^r}(\Lambda^{\lambda_1}\otimes\cdots\otimes\Lambda^{\lambda_\ell})$ contains
   $(\Lambda^{\lambda_1-p^r}\otimes\Lambda^{\lambda_2}\otimes\cdots\otimes\Lambda^{\lambda_\ell})\boxtimes\sgn_H$ 
   as a direct summand, with complimentary summands of the form $N\boxtimes V$ with $V\neq \sgn_H$ indecomposable and  $N\in \Filt(\De_{<(\lambda-p^r e_1)^t})$.

   We obtain the following commutative diagram: 
   \[\begin{tikzcd}
        \Delta_{(\lambda-p^r e_1)^t}\boxtimes\sgn_H
        \arrow[hook]{r}
        \arrow{d}
        &
        (\Lambda^{\lambda_1-p^r}\otimes\Lambda^{\lambda_2}\otimes\cdots\otimes\Lambda^{\lambda_\ell})\boxtimes\sgn_H
        \arrow[hook, "\oplus"]{d}
        \\
        E^{p^r}(\Delta_{\lambda^t})
        \arrow[hook]{r}
        &
        E^{p^r}(\Lambda^{\lambda_1}\otimes\Lambda^{\lambda_2}\otimes\cdots\otimes\Lambda^{\lambda_\ell})
        \ .
   \end{tikzcd}\]
   After applying $U$, the upper horizontal arrow splits, and therefore the left vertical arrow splits as well. 
   In other words, $U(\Delta_{(\lambda-p^r e_1)^t}\boxtimes\sgn_H)$ is a direct summand of 
   $UE^{p^r}(\Delta_{\lambda^t})$.
   The following commuting diagram
   \[\begin{tikzcd}
    U(\Delta_{(\lambda-p^r e_1)^t}\boxtimes\sgn_H)
    \arrow[hook, "\oplus"]{r}
    \arrow[equal]{d}
    &U\Phi_{\sgn_H}(\Delta_{\lambda^t})
    \arrow{r}
    \arrow{d}
    &U\Phi_{\sgn_H}(\Lambda^{\lambda_1}\otimes\Lambda^{\lambda_2}\otimes\cdots\otimes\Lambda^{\lambda_\ell})
    \arrow[equal]{d}
    \\
    U(\Delta_{(\lambda-p^r e_1)^t}\boxtimes\sgn_H)
    \arrow[equal]{r}
    &U(\Phi_{\mathbf{a}}\boxtimes\sgn_H)(\De_{\lambda^t})
    \arrow[hook]{r}
    &U(\Phi_{\mathbf{a}}\boxtimes\sgn_H)(\Lambda^{\lambda_1}\otimes\Lambda^{\lambda_2}\otimes\cdots\otimes\Lambda^{\lambda_\ell})
   \end{tikzcd}\]
shows that the natural map $U\Phi_{\sgn_H}(\De_{\lambda^t})\rightarrow U\Phi_{\mathbf{a}}(\De_{\lambda^t})\boxtimes\sgn_H$ is a (split) epimorphism. The lemma now follows from~Remark~\ref{rem sgn split}.
\end{proof}

\begin{lemma}\label{lem oti scs is sgn}
The OTI functor $\Phi_{\sgn_H}$ coincides with its tilting derived functor $\Phi_{\mathbf{a}}\boxtimes\sgn_H$.
\end{lemma}

\begin{proof}
By~Lemma~\ref{lem iso on stand}, the natural maps $\Phi_{\sgn_H}(\De_{\lambda^t})\rightarrow\Phi_{\mathbf{a}}(\De_{\lambda^t})\boxtimes\sgn_H$ and $\Phi_{\mathbf{a}}(\Na_{\lambda^t})\boxtimes\sgn_H\rightarrow \Phi_{\sgn_H}(\Na_{\lambda^t})$ are isomorphisms for all $p^r$-stable partitions $\lambda$. By iteratively applying~Lemma~\ref{lemma hw cat nonsense}, we find a natural transformation $\Phi_{\mathbf{a}}\boxtimes\sgn_H\Rightarrow \Phi_{\sgn_H}$ which is a split monomorphism. Remark~\ref{rem sgn split} then completes the proof. 
\end{proof}

Now let $\varphi:\operatorname{Rep}(H)\to\As$ be a CF functor, where $\As$ is a locally finite abelian category.  Consider the corresponding OTI functor 
$$\Phi_{\varphi}:\PP ol_d\to\PP ol_{d-p^r}\boxtimes\As.$$ Recall that $\Phi_{\varphi}(\PP ol^{\st}_{d,p^r})\sub \PP ol^{\trun}_{d-p^r,p^r}\boxtimes\As$ by~Lemma~\ref{lemma oti st to tr}.

\begin{thm}\label{thm periodicity poly functors}
Suppose $\varphi$ satisfies the following assumptions: 
\begin{enumerate}
    \item for all $H$-modules $V$ we have $\varphi(V)=0$ whenever $\varphi_{\mathbf{a}}(V)=0$, and
    \item $\varphi(\sgn_H)$ is a simple object of $\As$.
\end{enumerate}
Then the OTI functor $\Phi_{\varphi}:\PP ol^{\st}_{d,p^r}\to \PP ol^{\trun}_{d-p^r,p^r}\boxtimes\As$ coincides with $\Phi_{\mathbf{a}}\boxtimes \varphi(\sgn_H)$. In particular, $\Phi_{\varphi}$ corestricts to an equivalence $\PP ol^{\st}_{d,p^r}\xto{\sim}\PP ol^{\trun}_{d-p^r,p^r}$ of highest weight categories, independent of the choice of $\varphi$.
\end{thm}

We note that the semisimplification functor $\varphi_H:\rep H\to(\rep H)^{ss}$ satisfies assumptions (1) and (2), as do the CF functors in~Example~\ref{example CF functors}.
The conjectural CF functor constructed in ~\cite[Sec.~5.4]{CF} does as well.

\begin{proof}
By~Lemma~\ref{lemma tilts to tilts}, the equivalence 
$\Phi_{\mathbf{a}}\boxtimes \varphi(\sgn_H)$ coincides with $\Phi_{\varphi}$ on the subcategory $\Tilt^{\st}_{d,p^r}$. 
In particular, we will identify $\Phi_{\mathbf{a}}\boxtimes \varphi(\sgn_H)$ with the tilting derived functor of~$\Phi_{\varphi}$.
Now let $M\in \PP ol^{\st}_{d,p^r}$ and choose a decomposition  $$UE^{p^r}M=(UE^{p^r}M)_{\sgn_H}\oplus V.$$
Applying $\varphi_{\mathbf{a}}$ to both sides and using Lemma~\ref{lem oti scs is sgn}, 
we find $U\Phi_{\mathbf{a}}(M)\cong U\Phi_{\mathbf{a}}(M)\oplus \varphi_{\mathbf{a}}(V),$
hence $\varphi_{\mathbf{a}}(V)=0$, which by $(1)$ implies $\varphi(V)=0$.
Hence, using~Lemma~\ref{lem oti scs is sgn} once more, 
$$U\Phi_{\varphi}(M)
\cong\varphi((UE^{p^r}M)_{\sgn_H}\oplus V)
\cong U(1\boxtimes \varphi)(E^{p^r}M)_{\sgn_H}
\cong U(\Phi_{\mathbf{a}}(M)\boxtimes\varphi(\sgn_H)).$$
It is not hard to check that for $M=\De_{\lambda^t}$, the above isomorphism corresponds to the image under $U$ of the natural map $\Phi_{\varphi}(\De_{\lambda^t})\rightarrow(\Phi_{\mathbf{a}}\boxtimes \varphi(\sgn_H))(\De_{\lambda^t})$. 
Since $U$ detects isomorphisms, $\Phi_{\varphi}(\De_{\lambda^t})\rightarrow(\Phi_{\mathbf{a}}\boxtimes \varphi(\sgn_H))(\De_{\lambda^t})$ is an isomorphism for all $p^r$-stable partitions~$\lambda$. Dually, the natural map $(\Phi_{\mathbf{a}}\boxtimes \varphi(\sgn_H))(\Na_{\lambda^t})\rightarrow \Phi_{\varphi}(\Na_{\lambda^t})$ is also an isomorphism for all such $\lambda$. By iteratively applying~Lemma~\ref{lemma hw cat nonsense}, we find a natural transformation $\Phi_{\mathbf{a}}\boxtimes \varphi(\sgn_H)\Rightarrow \Phi_{\varphi}$ on $\PP ol^{\st}_{d,p^r}$ which is a split monomorphism. Since it must be an isomorphism after applying $U$, this natural transformation is an isomorphism.
\end{proof}

\subsection{Periodicity for symmetric groups}

We would now like to translate Theorem \ref{thm periodicity poly functors} to $\rep S_d$.  For this, define $\Rep^{\st}_{p^r} S_d$ (resp.~$\Rep^{\trun}_{p^r} S_d$) to be the full subcategory of $\PP ol^{\st}_{d,p^r}$ (resp.~$\PP ol^{\trun}_{d,p^r}$) on objects in the image of $\mathcal{F}\otimes \sgn$. Let $\KK_d^{\st}$ (resp.~$\KK_{d}^{\trun}$) denote the Serre subcategory of $\PP ol_{d,p^r}^{\st}$ (resp.~$\PP ol_{d,p^r}^{\trun}$) generated by simples $L_\lambda$ where $\lambda^t$ is $p$-singular.  Let $\varphi$ be as in~Theorem~\ref{thm periodicity poly functors}. Then~Lemma~\ref{lem:OTI-SWdual} shows we have a commutative diagram:
 \begin{equation*}\label{eqn comm diag oti}
      \xymatrix{
    \KK^{\st}_{d,p^r}\ar[r] \ar[d]_{\sim}^{\Phi_{\varphi}} & \PP ol_{d,p^r}^{\st} \ar[r]^-{\FF} \ar[d]_{\sim}^{\Phi_{\varphi}} & \Rep^{\st}_{p^r} S_d\otimes\sgn \ar[d]^{\Phi_{\varphi}} 
    \\
    \KK^{\trun}_{d-p^r,p^r} \ar[r] & \PP ol_{d-p^r,p^r}^{\trun} \ar[r]^-{\FF} & \Rep^{\trun}_{p^r} S_{d-p^r}\otimes\sgn
    \ .
    }
 \end{equation*}
Note that $\KK^{\st}_{d,p^r}$ (resp.~$\KK^{\trun}_{d,p^r}$) is exactly the Serre subcategory of objects annihilated by $\FF$ in $\PP ol_{d,p^r}^{\st}$ (resp. $\PP ol_{d,p^r}^{\trun}$).  It is easy to see from Theorem \ref{thm periodicity poly functors} that $\Phi_{\varphi}:\KK^{\st}_{d,p^r}\to\KK^{\trun}_{d-p^r,p^r}$ is an equivalence.  

We will need the following lemma. Recall that a subcategory $\As\subset\Bs$ is called \textit{quotient reflective} if the inclusion $\As\hookrightarrow \Bs$ has a left adjoint such that the unit of the adjunction is an epimorphism.

\begin{lemma}
Let $\Bs$ and $\Bs' $ be abelian categories and suppose $S:\Bs\to \Bs'$ is a Serre quotient functor. Let $\Cs\subset\Bs$ be a quotient reflective abelian subcategory.
Write $\Cs'$ for the essential image of $\Cs$ under~$S$. Then for every morphism $f:SX\to SY$ in $\Cs'$ with $X,\,Y$ in $\Cs$, there exist morphisms $g:Z\to X$ and $h:Z\to Y$ both in $\Cs$, so that $Sg$ is an isomorphism and such that $f=Sh(Sg)^{-1}$.
In particular, $\Cs'\subset \Bs'$ is an abelian subcategory and $S:\Cs\to \Cs'$ is a Serre quotient functor.
\end{lemma}

\begin{proof}
 Let $f:SX\to SY$ be a morphism in $\Cs'$ with $X,\,Y$ in $\Cs$. Since $S$ is a Serre quotient, there exists an object $Z$ in $\Bs$ and morphisms $g:Z\to X$ and $h:Z\to Y$ in $\Bs$, so that $Sg$ is an isomorphism and such that $f=Sh(Sg)^{-1}$: 
$$\begin{tikzcd}
    &Z \ar[dl, "g"'] \ar[dr, "h"]&\\
      X&  &Y.
\end{tikzcd}$$
Since $X,\,Y$ are in $\Cs$ and the inclusion $i_*:\Cs\hookrightarrow \Bs$ has a left adjoint $i^*$ by assumption, we know that $g$ and $h$ factor via the unit morphism on $Z$,
$$\begin{tikzcd}
    &Z \ar[dl, "g"'] \ar[dr, "h"]\ar[d,two heads, "\eta"]&\\
      X& i_*i^* Z\ar[r, "h'"]\ar[l, "g'"'] &Y.
\end{tikzcd}$$
Since $Sg$ is an isomorphism and $S\eta$ is an epimorphism, it follows that $Sg'$ is an isomorphism with $f=Sh'(Sg')^{-1}$. 
It follows immediately that $\Cs'$ is closed under kernels and cokernels in $\Bs'$, hence is abelian, and that 
$S:\Cs\to \Cs'$ is a Serre quotient functor.
\end{proof}

By~Lemma~\ref{lemma:p^r_stability_open}, we know that $\PP ol^{\trun}_{d,p^r}$ and $\PP ol^{\st}_{d,p^r}$ are quotient reflective subcategories of $\PP ol^d$, hence the following is an easy corollary of the lemma.

\begin{cor} Let $\mathrm{a}\in\{\trun,\st\}.$
\begin{enumerate}
    \item $\Rep^{\mathrm{a}}_{p^r} S_d$ is an abelian subcategory of $\rep S_d$, containing all $M^{\lambda}$ and $S^{\lambda}$ where $\lambda$ is $p^r$-stable/truncated, and all $D^{\lambda}$ where $\lambda$ is $p$-regular, $p^r$-stable/truncated.
    \item $\FF:\PP ol^{\mathrm{a}}_{d,p^r}\to \left(\Rep^{\mathrm{a}}_{p^r} S_d\right)\otimes\sgn$ is a Serre quotient functor whose kernel $\KK_d^{\mathrm{a}}$ has simples given by $L_{\lambda}$ where $\lambda^t$ is $p$-singular, $p^r$-stable/truncated.
\end{enumerate}
\end{cor}

Recall that we assume $A\sub H\sub S_{p^r}$, where $A$ is a transitive elementary abelian $p$-group of order $p^r$.

\begin{thm}\label{thm proof from intro ss}
Suppose $\varphi:\operatorname{Rep}(H)\to\As$ satisfies the following assumptions: 
\begin{enumerate}
    \item for all $H$-modules $V$ we have $\varphi(V\otimes \sgn_H)=0$ whenever $\varphi_{\mathbf{a}}(V)=0$, and
    \item $\varphi(\mathbf{1})$ is a simple object of $\As$.
\end{enumerate}
Then the OTI functor $\Phi_{\varphi}$ (co)restricts to an equivalence $\Phi_{\varphi}:\Rep^{\st}_{p^r} S_d\to \Rep^{\trun}_{p^r} S_{d-p^r}$, independent of the choice of $\varphi$.
\end{thm}

Again, the semisimplification functor $\varphi_H:\rep H\to(\rep H)^{ss}$ and the CF functors in~Example~\ref{example CF functors} satisfy the above assumptions, as well as the conjectural CF functor constructed in ~\cite[Sec.~5.4]{CF}.

\begin{proof}
    Let $\varphi_{\sgn}:=\varphi(-\otimes \sgn)$. Then, $\varphi_{\sgn}$ satisfies the conditions in~Theorem \ref{thm periodicity poly functors}, so
    we have a commutative diagram of Serre quotients,
\[
      \xymatrix{
    \KK^{\st}_{d,p^r}\ar[r] \ar[d]_{\sim}^{\Phi_{\varphi_{\sgn}}} & \PP ol_{d,p^r}^{\st} \ar[r]^-{\FF} \ar[d]_{\sim}^{\Phi_{\varphi_{\sgn}}} & \Rep^{\st}_{p^r} S_d\otimes\sgn \ar[d]^{\Phi_{\varphi_{\sgn}}} \\
    \KK^{\trun}_{d-p^r,p^r} \ar[r] & \PP ol_{d-p^r,p^r}^{\trun} \ar[r]^-{\FF} & \Rep^{\trun}_{p^r} S_{d-p^r}\otimes\sgn.}
    \]
    It follows that $\Phi_{\varphi}:\Rep^{\st}_{p^r} S_d\otimes\sgn\to \Rep^{\trun}_{p^r} S_{d-p^r}\otimes\sgn$ is an equivalence.  Since
    \begin{equation*}
      \xymatrix{
    \Rep^{\st}_{p^r} S_d\otimes\sgn \ar[d]^{\Phi_{\varphi_{\sgn}}}_{\sim}
    \ar[r]^-{-\otimes \sgn} 
    & \Rep^{\st}_{p^r} S_d
    \ar[d]^{\Phi_{\varphi}} 
    \\
     \Rep^{\trun}_{p^r} S_{d-p^r}\otimes\sgn
    \ar[r]^-{-\otimes \sgn} 
    & \Rep^{\trun}_{p^r} S_{d-p^r}
    }
 \end{equation*}
    commutes, $\Phi_{\varphi}:\Rep^{\st}_{p^r} S_d\to \Rep^{\trun}_{p^r} S_d$ is also an equivalence. 
\end{proof}

\begin{cor}\label{cor proof from intro}
    Suppose that $V\in\Rep^{\st}_{d,p^r} S_d$.  Then we have a decomposition 
    \[
    V|_{H}=\mathbf{1}^{\oplus \ell} \oplus \bigoplus\limits_i N_{i},
    \]
    where each $N_i$ is indecomposable and satisfies $\varphi_{\mathbf{a}}(N_i)=0$ for all generic $\mathbf{a}$.  In particular, $p \mid \dim(N_i)$ for all $i$.
\end{cor}

\begin{proof}
    Indeed, this follows by applying Theorem \ref{thm proof from intro ss} to the CF functor given by semisimplification with respect to $H$. 
\end{proof}
    
\section{Diagrammatics and the proof of Theorem \ref{thm OTI categorical commutation}} \label{section:appendix_proof_of_categorical_commutation}
In order to provide a clear proof of Theorem \ref{thm OTI categorical commutation}, we recall the diagrammatic presentation of degenerate categorical Heisenberg actions, see also~\cite[Section 3.1]{BSW}.

\subsection{Trace morphisms for finite groups}
We first record a lemma from \cite[Prop.~3.6.6]{Benson}, on representations of finite groups.
Let $H\sub G$ be finite groups and $M,N$ be $G$-modules. 
The transfer map
\[
\operatorname{Tr}_{H,G}:\Hom_H(M,N)\to\Hom_G(M,N)
\]
is defined by
    \[
    \operatorname{Tr}_{H,G}(f):=\sum\limits_{g\in G/H}gfg^{-1}\in\Hom_{G}(M,N).
    \]
\begin{lemma}\label{lem:phi-kills-quotient-orbit}
    A morphism $h:M\to N$ of $G$-modules factors through a module of the form $\Ind_{H}^GV$ if and only if $h=\operatorname{Tr}_{H,G}(f)$ for some $f\in\Hom_H(M,N)$.
\end{lemma}

\begin{cor}\label{cor:phi-kills-tr}
A $\Bbbk$-linear functor $\varphi:\rep H\to\As$ satisfies property (2) of Definition \ref{def:cf-functor} if and only if $\varphi(\operatorname{Tr}_{H',H}(f))=0$ for all non-transitive subgroups $H'\sub H$ and all $f\in\Hom_{H'}(M,N)$.
\end{cor}

\subsection{Diagrammatics}\label{section diagrammatics}

We begin by recalling the degenerate Heisenberg category $\HH eis_{k}$, where $k\in\Z$. This category is monoidal, with monoidal unit $\mathbf{1}$, and is generated by objects $\uparrow,$ $\downarrow$, and morphisms
\begin{gather*}
    \begin{tikzpicture}[vcenter=-0.2em, scale=0.2]
        \path[string] (0,0) edge[->] (0,3);
        \node [dot] at (0,1.5) {}; 
    \end{tikzpicture}
    : {\uparrow} \to {\uparrow},
    \quad
    \begin{tikzpicture}[vcenter=-0.2em, scale=0.2]
        \diagcup{0,3}{1.5,1}{3,3}[->];
    \end{tikzpicture}
    :\mathbf{1} \to {\downarrow} \otimes {\uparrow},
    \quad
    \begin{tikzpicture}[vcenter=-0.2em, scale=0.2]
        \diagcap{0,0}{1.5,2}{3,0}[->];
    \end{tikzpicture}
    : {\uparrow} \otimes {\downarrow} \to \mathbf{1},
    \\
    \begin{tikzpicture}[vcenter=-0.2em, scale=0.2]
        \path[string]
        (0,0) edge[->] (3,3)
        (3,0) edge[->] (0,3)
        ;
    \end{tikzpicture}
    : {\uparrow} \otimes {\uparrow} \to {\uparrow} \otimes {\uparrow},
    \quad
    \begin{tikzpicture}[vcenter=-0.2em, scale=0.2]
        \diagcup{0,3}{1.5,1}{3,3}[<-];
    \end{tikzpicture}
    :\mathbf{1} \to {\uparrow} \otimes {\downarrow},
    \quad
    \begin{tikzpicture}[vcenter=-0.2em, scale=0.2]
        \diagcap{0,0}{1.5,2}{3,0}[<-];
    \end{tikzpicture}
    : {\downarrow} \otimes {\uparrow} \to \mathbf{1}.
\end{gather*}
These morphisms satisfy a number of relations.  To state them, we first define
\begin{gather*}
    \begin{tikzpicture}[vcenter=-0.2em, scale=0.2]
        \path[string]
        (0,0) edge[->] (3,3)
        (3,0) edge[<-] (0,3)
        ;
    \end{tikzpicture}
    :=    
    \begin{tikzpicture}[vcenter=-0.2em, scale=0.1]
        \path[string]
        (3,-5) edge[->] (-3,5)
        (-2,-2) edge (2,3)
        (2,3) edge[out=50,in=180] (5,5)
        (5,5) edge[->,out=0,in=90] (9,-5)
        (-2,-2) edge[out=230,in=0] (-6,-5) 
        (-6,-5) edge[out=180,in=-90] (-9,5)
        ;
    \end{tikzpicture}
    \quad , \quad
    \begin{tikzpicture}[vcenter=-0.2em, scale=0.2]
        \path[string]
        (0,0) edge[<-] (3,3)
        (3,0) edge[->] (0,3)
        ;
    \end{tikzpicture}
    :=
    \begin{tikzpicture}[vcenter=-0.2em, scale=0.1]
        \path[string]
        (3,5) edge[<-] (-3,-5)
        (-2,2) edge (2,-3)
        (2,-3) edge[out=130,in=180] (5,-5)
        (5,-5) edge[out=0,in=270] (9,5)
        (-2,2) edge[out=130,in=0] (-6,5)
        (-6,5) edge[->,out=180,in=-270] (-9,-5)
        ;
    \end{tikzpicture}
    \ .
\end{gather*}
Then the relations are given as
\begin{gather}
    \label{eq:heisenberg-diag-rel}
    \begin{tikzpicture}[vcenter=-0.2em, scale=0.1]
        \path[string]
    	(2.8,0) edge[->,out=90,in=-90] (-2.8,6)
    	(-2.8,0) edge[->,out=90,in=-90] (2.8,6)
    	(2.8,-6) edge[out=90,in=-90] (-2.8,0)
    	(-2.8,-6) edge[out=90,in=-90] (2.8,0)
        ;
    \end{tikzpicture}
    =
    \begin{tikzpicture}[vcenter=-0.2em, scale=0.1]
        \path[string]
        (1.8,-6) edge[->] (1.8,6)
        (-1.8,-6) edge[->] (-1.8,6)
        ;
    \end{tikzpicture}
    \quad , \quad
    \begin{tikzpicture}[vcenter=-0.2em, scale=0.1]
        \path[string]
    	(4.5,6) edge[<-] (-4.5,-6)
    	(4.5,-6) edge[->] (-4.5,6)
        (0,-6) edge[out=90,in=-90] (-4.5,0)
        (-4.5,0) edge[->,out=90,in=-90] (0,6)
        ;
    \end{tikzpicture}
    =
    \begin{tikzpicture}[vcenter=-0.2em, scale=0.1]
        \path[string]
    	(4.5,6) edge[<-] (-4.5,-6)
    	(4.5,-6) edge[->] (-4.5,6)
        (0,-6) edge[out=90,in=-90] (4.5,0)
        (4.5,0) edge[->,out=90,in=-90] (0,6)
        ;
    \end{tikzpicture}
    \quad , \quad
    \begin{tikzpicture}[vcenter=-0.2em, scale=0.2]
        \path[string]
        (0,0) edge[->] (3,3)
        (3,0) edge[->] (0,3)
        ;
        \node [dot] at (0.75,0.75) {}; 
    \end{tikzpicture}
    =
    \begin{tikzpicture}[vcenter=-0.2em, scale=0.2]
        \path[string]
        (0,0) edge[->] (3,3)
        (3,0) edge[->] (0,3)
        ;
        \node [dot] at (2.25,2.25) {};
    \end{tikzpicture}
    +
    \begin{tikzpicture}[vcenter=-0.2em, scale=0.2]
     	\path[string]
        (0,0) edge[->] (0,3)
        (1.5,0) edge[->] (1.5,3)
        ;
    \end{tikzpicture}
    \ ,
    \\
    \begin{tikzpicture}[vcenter=-0.2em, scale=0.1]
        \path[string]
        (3,0) edge[->] (3,4)
    	(3,0) edge[out=-90, in=0] (1,-4)
    	(1,-4) edge[out = 180, in = -90] (-1,0)
    	(-1,0) edge[out=90, in=0] (-3,4)
    	(-3,4) edge[out = 180, in =90] (-5,0)
        (-5,0) edge (-5,-4)
        ;
    \end{tikzpicture}
    =
    \begin{tikzpicture}[vcenter=-0.2em, scale=0.1]
      \path[string]
      (0,-4) edge[->] (0,4)
      ;
    \end{tikzpicture}
    \quad,\quad
    \begin{tikzpicture}[vcenter=-0.2em, scale=0.1]
        \path[string]
        (3,0) edge[->] (3,-4)
        (3,0) edge[out=90,in=0] (1,4)
        (1,4) edge[out=180,in=90] (-1,0)
        (-1,0) edge[out=-90,in=0] (-3,-4)
        (-3,-4) edge[out=180,in=-90] (-5,0)
        (-5,0) edge (-5,4)
        ;
    \end{tikzpicture}
    =
    \begin{tikzpicture}[vcenter=-0.2em, scale=0.1]
      \path[string] (0,-4) edge[<-] (0,4);
    \end{tikzpicture}
    \ .
\end{gather}
Finally, we require the following maps to be isomorphisms in the additive closure of $\HH eis_k$ when $k$ is non-negative or negative, respectively

\begin{equation}\label{eqn magic iso k>=0}
     \begin{pmatrix} 
        \begin{tikzpicture}[vcenter=-0.2em, scale=0.2]
            \path[string]
            (0,0) edge[->] (3,3)
            (3,0) edge[<-] (0,3)
            ;
        \end{tikzpicture}
        \\
        \begin{tikzpicture}[vcenter=-0.2em, scale=0.2]
            \tikzfixsize{(0,0)}{(3,3)}
            \diagcap{0,0}{1.5,2}{3,0}[->];
        \end{tikzpicture}
        \\
        \begin{tikzpicture}[vcenter=-0.2em, scale=0.2]
            \tikzfixsize{(0,0)}{(3,3)}
            \diagcap{0,0}{1.5,2}{3,0}[->];
            \node [dot] at (0.25,1) {};
        \end{tikzpicture}
        \\
        \vdots
        \\
        \begin{tikzpicture}[vcenter=-0.2em, scale=0.2]
            \tikzfixsize{(0,0)}{(3,3)}
            \diagcap{0,0}{1.5,2}{3,0}[->];
            \node [dot, label={[left,xshift=1pt, yshift=-1pt]{\tiny $k\!\shortminus\!1$}}] at (0.25,1) {};
        \end{tikzpicture}
    \end{pmatrix} :
    {\uparrow} \otimes {\downarrow} \xto{\sim} {\downarrow} \otimes {\uparrow} \oplus \mathbf{1}^{\oplus k},
\end{equation}

\begin{equation}\label{eqn magic iso k<0}
    \begin{pmatrix}
        \begin{tikzpicture}[vcenter=-0.2em, scale=0.2]
            \path[string]
            (0,0) edge[->] (3,3)
            (3,0) edge[<-] (0,3)
            ;
        \end{tikzpicture}
        &
        \begin{tikzpicture}[vcenter=-0.2em, scale=0.2]
            \diagcup{0,3}{1.5,1}{3,3}[->];
        \end{tikzpicture}
        &
        \begin{tikzpicture}[vcenter=-0.2em, scale=0.2]
            \diagcup{0,3}{1.5,1}{3,3}[->];
            \node [dot] at (2.75,1.8) {};
        \end{tikzpicture}
        &
        \cdots 
        &
        \begin{tikzpicture}[vcenter=-0.2em, scale=0.2]
            \diagcup{0,3}{1.5,1}{3,3}[->];
            \node [dot, label={[right,xshift=-1pt, yshift=-2pt]{\tiny $k\!\shortminus\!1$}}] at (2.75,1.8) {};
        \end{tikzpicture}
    \end{pmatrix}:
    {\uparrow} \otimes {\downarrow} \oplus \mathbf{1}^{\oplus (-k)}\ \xto{\sim} {\downarrow} \otimes {\uparrow}.
\end{equation}

A degenerate categorical Heisenberg action on $\Cs$ is by definition a $\Bbbk$-linear, monoidal functor $H:\HH eis_{k}\to\operatorname{Fun}(\Cs,\Cs)$. We write $E \coloneqq H(\uparrow)$ and $F \coloneqq H(\downarrow)$. The structure maps $\eta, \epsilon, x, T$ correspond to the appropriate cups and caps, a dot on a strand, and crossing respectively. 

Let us now assume that $\Cs$ is a locally finite abelian category with a degenerate categorical Heisenberg action.  We will continue to draw the functors $E,F$ as ${\uparrow}$ and ${\downarrow}$, and natural transformations between them by the same diagrams used in $\HH eis_k$.  

If $\As$ is another locally finite abelian category, then as noted in Section \ref{sec:Heis-act}, $\Cs\boxtimes\As$ inherits a natural categorical Heisenberg action, where ${\uparrow}$ acts by $E\boxtimes 1$ and ${\downarrow}$ acts by $F\boxtimes 1$.  By abuse of notation, we will use the same diagrammatics for this categorical action as well.

\subsection{Diagrammatics of the OTI functor}
Recalling that $E^n$ admits an action of $S_n$, we may define a functor
\begin{center}
    \begin{tikzpicture}[vcenter=0.2em, scale=0.15]
        \path[string]
        (0,0) edge[-{>[length=0.3em]},ultra thick] (0,5)
        ;
        \node [below] at (0,0) {\tiny $n$};
    \end{tikzpicture}
    :=
    \begin{tikzpicture}[vcenter=0.3em, scale=0.15]
          \path[string]
          (0,0) edge[->] (0,5)
          (3,0) edge[->] (3,5)
          ;
          \node at (1.5,2.5) {...};
          \draw[underbrace] (0,0) -- (3,0);
          \node[below] at (1.5,-0.5) {\tiny$n$};
    \end{tikzpicture} $:\Cs\to \Cs\boxtimes\rep S_n$.
\end{center}
Now let $\varphi:\rep H\to\As$ denote a CF functor for $H\sub S_{p^r}$.   We will represent $\varphi$ diagramatically as:
\begin{gather*}
    \begin{tikzpicture}[vcenter=0.15em, scale=0.15]
        \tikzfixsize{(-0.25,0)}{(0.25,5)}
        \path[string]
        (0,0) edge[densely dashed] (0,5)
        ;
        \node [below] at (0,0) {\tiny $\varphi$};
    \end{tikzpicture}
    : \rep H \to \As.
\end{gather*}
To simplify our notation, we will assume without loss of generality that $H=S_{p^r}$ (see Remark \ref{remark CF general subgroup}).
By Corollary \ref{cor:phi-kills-tr}, condition (2) of CF functors implies that
\begin{gather}\label{eq:dotted-kills-trace}
    \begin{tikzpicture}[vcenter=0.15em, scale=0.15]
        \tikzfixsize{(0,0)}{(1,0)}
        \path[string]
        (0,0) edge[densely dashed] (0,5)
        ;
        \node [below] at (0,0) {\tiny $\varphi$};
    \end{tikzpicture}
    \operatorname{Tr}_{S_{p^r-1},S_{p^r}}(f)=0
\end{gather}
for any $S_{p^r-1}$-equivariant morphism $f$.

Abusing notation, we use the same dashed line diagram for the functor $1\boxtimes\varphi:\Cs\boxtimes\rep S_{p^r}\to\Cs\boxtimes\As$.  It follows that an OTI functor is written diagrammatically as
\begin{gather*}
    \Phi_{\varphi}=
    \begin{tikzpicture}[vcenter=0.1em, scale=0.15]
        \tikzfixsize{(-0.5,0)}{(2.5,5)}
        \path[string]
        (0,0) edge[densely dashed] (0,5)
        (2,0) edge[-{>[length=0.3em]},ultra thick] (2,5)
        ;
        \node [below] at (0,0) {\tiny $\varphi$};
        \node [below] at (2.5,0.5) {\tiny $p^r$};
    \end{tikzpicture}
    : \Cs \to \Cs \boxtimes \As.
\end{gather*}
By Lemma \ref{lemma additive functors extend}(2), we have natural transformations that commute $1 \boxtimes \varphi$ with $E \boxtimes 1$ and $F \boxtimes 1$,
\begin{gather*}
    \begin{tikzpicture}[vcenter=0.2em, scale=0.2]
        \path[string]
        (0,0) edge[densely dashed] (3,3)
        (3,0) edge[->] (0,3)
        ;
        \node [below] at (0,0) {\tiny $\varphi$};
    \end{tikzpicture}
    \quad , \quad
    \begin{tikzpicture}[vcenter=0.2em, scale=0.2]
        \path[string]
        (0,0) edge[->] (3,3)
        (3,0) edge[densely dashed] (0,3)
        ;
        \node [below] at (3,0) {\tiny $\varphi$};
    \end{tikzpicture}
    \quad , \quad
    \begin{tikzpicture}[vcenter=0.2em, scale=0.2]
        \path[string]
        (0,0) edge[densely dashed] (3,3)
        (3,0) edge[<-] (0,3)
        ;
        \node [below] at (0,0) {\tiny $\varphi$};
    \end{tikzpicture}
    \quad , \quad
    \begin{tikzpicture}[vcenter=0.2em, scale=0.2]
        \path[string]
        (0,0) edge[<-] (3,3)
        (3,0) edge[densely dashed] (0,3)
        ;
        \node [below] at (3,0) {\tiny $\varphi$};
    \end{tikzpicture}
    \ ,
\end{gather*}
such that
\begin{gather*}
    \begin{tikzpicture}[vcenter=0.2em, scale=0.2]
        \path[string]
        (0,0) edge[densely dashed, out=40, in=270] (2.25,2)
        (2.25,2) edge[densely dashed, out=90, in=320] (0,4)
        (3,0) edge[out=140, in=270] (0.75,2)
        (0.75,2) edge[->,out=90, in=220] (3,4)
        ;
        \node [below] at (0,0) {\tiny $\varphi$};
    \end{tikzpicture}
    =
    \begin{tikzpicture}[vcenter=0.2em, scale=0.2]
        \path[string]
        (0,0) edge[densely dashed] (0,4)
        (1.5,0) edge[->] (1.5,4) 
        ;
        \node [below] at (0,0) {\tiny $\varphi$};
    \end{tikzpicture}
    \quad,\quad
    \begin{tikzpicture}[vcenter=0.2em, scale=0.2]
        \path[string]
        (0,0) edge[out=40, in=270] (2.25,2)
        (2.25,2) edge[->, out=90, in=320] (0,4)
        (3,0) edge[densely dashed, out=140, in=270] (0.75,2)
        (0.75,2) edge[densely dashed, out=90, in=220] (3,4)
        ;
        \node [below] at (3,0) {\tiny $\varphi$};
    \end{tikzpicture}
    =
    \begin{tikzpicture}[vcenter=0.2em, scale=0.2]
        \path[string]
        (0,0) edge[->] (0,4)
        (1.5,0) edge[densely dashed] (1.5,4)
        ;
        \node [below] at (1.5,0) {\tiny $\varphi$};
    \end{tikzpicture}
    \quad,\quad
    \begin{tikzpicture}[vcenter=0.2em, scale=0.2]
        \path[string]
        (0,0) edge[densely dashed, out=40, in=270] (2.25,2)
        (2.25,2) edge[densely dashed, out=90, in=320] (0,4)
        (3,0) edge[<-,out=140, in=270] (0.75,2)
        (0.75,2) edge[out=90, in=220] (3,4)
        ;
        \node [below] at (0,0) {\tiny $\varphi$};
    \end{tikzpicture}
    =
    \begin{tikzpicture}[vcenter=0.2em, scale=0.2]
        \path[string]
        (0,0) edge[densely dashed] (0,4)
        (1.5,0) edge[<-] (1.5,4) 
        ;
        \node [below] at (0,0) {\tiny $\varphi$};
    \end{tikzpicture}
    \quad,\quad
    \begin{tikzpicture}[vcenter=0.2em, scale=0.2]
        \path[string]
        (0,0) edge[<-,out=40, in=270] (2.25,2)
        (2.25,2) edge[out=90, in=320] (0,4)
        (3,0) edge[densely dashed, out=140, in=270] (0.75,2)
        (0.75,2) edge[densely dashed, out=90, in=220] (3,4)
        ;
        \node [below] at (3,0) {\tiny $\varphi$};
    \end{tikzpicture}
    =
    \begin{tikzpicture}[vcenter=0.2em, scale=0.2]
        \path[string]
        (0,0) edge[<-] (0,4)
        (1.5,0) edge[densely dashed] (1.5,4)
        ;
        \node [below] at (1.5,0) {\tiny $\varphi$};
    \end{tikzpicture}
    \ .
\end{gather*}

\begin{lemma}\label{lem:dashed-commutes}
    The functor $1\boxtimes\varphi:\Cs\boxtimes\rep S_{p^r}\to\Cs\boxtimes\As$ is a morphism of degenerate categorical Heisenberg actions (see Definition \ref{defn categorical commutation}).
\end{lemma} 

\begin{proof}
It is enough to check the commuting diagrams which correspond to cups, caps, dot and crossing pulling through the dashed strand $1 \boxtimes \varphi$. Indeed, the equalities
\begin{gather*}
    \begin{tikzpicture}[vcenter=.2em, scale=0.2]
        \path[string]
        (0,0) edge[densely dashed] (3,3)
        (3,0) edge[->] (0,3)
        ;
        \node [below] at (0,0) {\tiny $\varphi$};
        \node [dot] at (2.25,0.75) {};
    \end{tikzpicture}
    =
    \begin{tikzpicture}[vcenter=0.2em, scale=0.2]
        \path[string]
        (0,0) edge[densely dashed] (3,3)
        (3,0) edge[->] (0,3)
        ;
        \node [below] at (0,0) {\tiny $\varphi$};
        \node [dot] at (0.8,2.2) {};
    \end{tikzpicture}
    \quad,\quad
    \begin{tikzpicture}[vcenter=0.2em, scale=0.2]
        \path[string]
        (0,0) edge[densely dashed] (4,4)
        (2,0) edge[->, out=45, in=-45, looseness=1.8] (2,4)
        (4,0) edge[->] (0,4)
        ;
        \node [below] at (0,0) {\tiny $\varphi$};
    \end{tikzpicture}
    =
    \begin{tikzpicture}[vcenter=0.2em, scale=0.2]
        \path[string]
        (0,0) edge[densely dashed] (4,4)
        (2,0) edge[->, out=135, in=225, looseness=1.8] (2,4)
        (4,0) edge[->] (0,4)
        ;
        \node [below] at (0,0) {\tiny $\varphi$};
    \end{tikzpicture}
    \quad,\quad
    \begin{tikzpicture}[vcenter=0.2em, scale=0.2]
        \path[string]
        (0,0) edge[densely dashed] (4,4)
        (3.5,2.5) edge[->] (2,4)
        (2.5,1.5) edge (0,4)
        (2.5,1.5) edge[out=-45, in=-45, looseness=2] (3.5,2.5)
        ;
        \node [below] at (0,0) {\tiny $\varphi$};
    \end{tikzpicture}
    =
    \begin{tikzpicture}[vcenter=0.2em, scale=0.2]
        \path[string]
        (3,0) edge[densely dashed] (3,4)
        ;
        \diagcup{0,4}{1,2.75}{2,4}
        \node [below] at (3,0) {\tiny $\varphi$};
    \end{tikzpicture}
    \quad,\quad
    \begin{tikzpicture}[vcenter=0.2em, scale=0.2]
        \path[string]
        (0,0) edge[densely dashed] (4,4)
        (3.5,2.5) edge (2,4)
        (2.5,1.5) edge[->] (0,4)
        (2.5,1.5) edge[out=-45, in=-45, looseness=2] (3.5,2.5)
        ;
        \node [below] at (0,0) {\tiny $\varphi$};
    \end{tikzpicture}
    =
    \begin{tikzpicture}[vcenter=0.2em, scale=0.2]
        \path[string]
        (3,0) edge[densely dashed] (3,4)
        ;
        \diagcup{0,4}{1,2.75}{2,4}[<-]
        \node [below] at (3,0) {\tiny $\varphi$};
    \end{tikzpicture}
\end{gather*}
follow from~Lemma~\ref{lemma additive functors extend}.
We illustrate this by proving the first equality. Let us write \[s_{\varphi,E}: (1\boxtimes\varphi)\circ(E\boxtimes 1)\xto{\sim}(E\boxtimes 1)\circ(1\boxtimes\varphi)\]
for the 2-isomorphism in $\AbCat$ natural in $E$ and~$\varphi$. Then the diagram 
\[\begin{tikzcd}
(1\boxtimes\varphi)(E\boxtimes 1)
\arrow["s_{\varphi,E}"]{d}
\arrow["(1\boxtimes 1_{\varphi})(x\boxtimes 1)"]{rr}
&&(1\boxtimes\varphi)(E\boxtimes 1)
\arrow["s_{\varphi,E}"]{d}\\
(E\boxtimes 1)(1\boxtimes\varphi)
\arrow["(x\boxtimes 1)(1\boxtimes 1_{\varphi})"]{rr}
&&(E\boxtimes 1)(1\boxtimes\varphi)
\end{tikzcd}\]
commutes by naturality of $s$ in $E$, which shows the first equality.
\end{proof}

\subsection{A structural lemma}
If $k\geq0$, then by repeatedly applying the isomorphism~\eqref{eqn magic iso k>=0} we obtain an isomorphism $\iota:E^n F \xto{\sim} F E^n \oplus\bigoplus_{i = 0}^{n-1} (E^{n-1}) ^{\oplus k}$, given explicitly by
 \begin{equation*}
 \iota= \left(
    \begin{tikzpicture}[vcenter,scale=0.15]
        \tikzfixsize{(0,0)}{(6,5)}
        \path[string]
        (0,5) edge[->] (6,0)
        (0,0) edge[->] (2,5)
        (1,0) edge[->] (3,5)
        (4,0) edge[->] (6,5)
        ;
        \node at (2.75,0.5) {...};
        \draw[underbrace] (0,0) -- (4,0);
        \node[below] at (2,-.5) {\tiny$n$};
    \end{tikzpicture}
    \ , \
    \begin{tikzpicture}[vcenter,scale=0.15]
        \path[string]
        (3,0) edge[->] (3,5)
        (6,0) edge[->] (6,5)
        ;
        \diagcap{1,0}{4,3}{7,0};
        \node at (4.5,0.5) {...};
        \draw[underbrace] (3,0) -- (6,0);
        \node[below] at (4.5,-.5) {\tiny$n\!\shortminus\!1$};
    \end{tikzpicture}
    \ , \
    \begin{tikzpicture}[vcenter,scale=0.15]
        \path[string]
        (3,0) edge[->] (3,5)
        (6,0) edge[->] (6,5)
        ;
        \diagcap{1,0}{4,3}{7,0};
        \node [dot] at (1.75,2) {};
        \node at (4.5,0.5) {...};
        \draw[underbrace] (3,0) -- (6,0);
        \node[below] at (4.5,-.5) {\tiny$n\!\shortminus\!1$};
    \end{tikzpicture}
    \ , ... ,
    \begin{tikzpicture}[vcenter,scale=0.15]
        \path[string]
        (3,0) edge[->] (3,5)
        (6,0) edge[->] (6,5)
        ;
        \diagcap{1,0}{4,3}{7,0};
        \node [dot, label={[xshift=-3pt, yshift=-1pt]{\tiny $k\!\shortminus\!1$}}] at (1.75,2) {};
        \node at (4.5,0.5) {...};
        \draw[underbrace] (3,0) -- (6,0);
        \node[below] at (4.5,-.5) {\tiny$n\!\shortminus\!1$};
    \end{tikzpicture}
    \ , \
    \begin{tikzpicture}[vcenter,scale=0.15]
        \path[string]
        (0,0) edge[->] (0,5)
        (3,0) edge[->] (3,5)
        (6,0) edge[->] (6,5)
        ;
        \diagcap{1,0}{4,3}{7,0};
        \node at (4.5,0.5) {...};
        \draw[underbrace] (3,0) -- (6,0);
        \node[below] at (4.5,-.5) {\tiny$n\!\shortminus\!2$};
    \end{tikzpicture}
    \ , ... , \
    \begin{tikzpicture}[vcenter,scale=0.15]
        \path[string]
        (0,0) edge[->] (0,5)
        (-3,0) edge[->] (-3,5)
        ;
        \diagcap{1,0}{2,2}{3,0};
        \node at (-1.5,0.5) {...};
        \draw[underbrace] (-3,0) -- (0,0);
        \node[below] at (-1.5,-.5) {\tiny$n\!\shortminus\!1$};
    \end{tikzpicture}
    \ ,..., \
    \begin{tikzpicture}[vcenter,scale=0.15]
        \path[string]
        (0,0) edge[->] (0,5)
        (-3,0) edge[->] (-3,5)
        ;
        \diagcap{1,0}{2,2}{3,0};
        \node [dot, label={[xshift=0pt, yshift=-2pt]{\tiny $k\!\shortminus\!1$}}] at (2,2) {};
        \node at (-1.5,0.5) {...};
        \draw[underbrace] (-3,0) -- (0,0);
        \node[below] at (-1.5,-.5) {\tiny$n\!\shortminus\!1$};
    \end{tikzpicture}
    \right)^\top.
    \end{equation*}
It will be convenient for us to express the components of the above morphism in a particular normal form.
\begin{lemma}\label{claim:move-dot-right}
    We have that for $i, j\geq 0$
    \begin{gather*}
    \begin{tikzpicture}[vcenter,scale=0.15]
        \tikzfixsize{(0,0)}{(6,5)}
        \path[string]
        (2,0) edge[->] (2,5)
        (5,0) edge[->] (5,5)
        ;
        \diagcap{0,0}{3,3}{6,0};
        \node [dot, label={[xshift=-1pt, yshift=-2pt]{\tiny $j$}}] at (0.75,2) {};
        \node at (3.5,0.5) {\textup{...}};
        \draw[underbrace] (2,0) -- (5,0);
        \node[below] at (3.5,-.5) {\tiny$i$};
    \end{tikzpicture}
    =
    \begin{tikzpicture}[vcenter,scale=0.15]
        \tikzfixsize{(0,0)}{(6,5)}
        \path[string]
        (1,0) edge[->] (1,5)
        (4,0) edge[->] (4,5)
        ;
        \diagcap{0,0}{3,3}{6,0};
        \node [dot, label={[xshift=1pt, yshift=-2pt]{\tiny $j$}}] at (5.25,2) {};
        \node at (2.5,0.5) {\textup{...}};
        \draw[underbrace] (1,0) -- (4,0);
        \node[below] at (2.5,-.5) {\tiny$i$};
    \end{tikzpicture}
    + \sum_{i',j'}
    \begin{tikzpicture}[vcenter,scale=0.15]
        \tikzfixsize{(0,0)}{(6,5)}
        \path[string]
        (1,0) edge[->] (1,5)
        (4,0) edge[->] (4,5)
        ;
        \diagcap{0,0}{3,3}{6,0};
        \bigbox{$\alpha$}{(-4.5,1.25)}{(-0.5,3.25)};
        \path[string]
        (-4,0) edge (-4,1.25)
        (-4,3.25) edge[-> ] (-4,5)
        (-1,0) edge (-1,1.25)
        (-1,3.25) edge[->] (-1,5)
        ;
        \node at (-2.5,0.5) {\textup{...}};
        \node at (-2.5,4) {\textup{...}};
        \node [dot, label={[xshift=1pt, yshift=-2pt]{\tiny $j'$}}] at (5.25,2) {};
        \node at (2.5,0.5) {\textup{...}};
        \draw[underbrace] (1,0) -- (4,0);
        \node[below] at (2.5,-.5) {\tiny$i'$};
    \end{tikzpicture}
    \end{gather*}
    where $i',j'$ run over some values such that $i' < i$ and $j' < j$ , and $\alpha$ is a diagram, depending on $i'$ and $j'$, composed of crossings and dots.
\end{lemma}

\begin{proof}
    We let $j \geq 0$ be arbitrary and apply induction on $i$. If $i = 0$, the result holds trivially with the second sum being empty. 
If $i=1$, we may iteratively apply~\eqref{eq:heisenberg-diag-rel} to obtain
\begin{align*}
        \begin{tikzpicture}[vcenter,scale=0.15]
            \tikzfixsize{(0,0)}{(4,5)}
            \path[string]
            (2,0) edge[->] (2,5);
            \diagcap{0,0}{2,2.5}{4,0};
            \node [dot, label={[xshift=-1pt, yshift=-2pt]{\tiny $j$}}] at (0.75,2) {};
        \end{tikzpicture}
        &=
        \begin{tikzpicture}[vcenter,scale=0.15]
            \tikzfixsize{(0,0)}{(4,5)}
            \path[string]
            (2,0) edge[->] (2,5);
            \diagcap{0,0}{2,2.5}{4,0};
            \node [dot, label={[xshift=1pt, yshift=-2pt]{\tiny $j$}}] at (3.25,2) {};
        \end{tikzpicture}
        + \sum_{j_1 + j_2 = j-1}
        \begin{tikzpicture}[vcenter,scale=0.15]
            \tikzfixsize{(0,0)}{(6,5)}
            \draw[string, ->] (0,0)
            to[out=90,in=180,looseness=1] (1.5,2.5)
            to[out=0,in=270,looseness=1] (3,5)
            ;
            \diagcap{3,0}{4,2}{5,0};
            \node [dot, label={[right, xshift=-1pt, yshift=-2pt]{\tiny $j_1$}}] at (2.8,3.35) {};
            \node [dot, label={[right, xshift=-1pt, yshift=-1pt]{\tiny $j_2$}}] at (4.75,1.5) {};
        \end{tikzpicture}.
          \end{align*}
Now let $i > 1$. Applying the result for $i-1$, we have
    \begin{align*}
        \begin{tikzpicture}[vcenter,scale=0.15]
            \tikzfixsize{(0,0)}{(6,5)}
            \path[string]
            (2,0) edge[->] (2,5)
            (5,0) edge[->] (5,5)
            ;
            \diagcap{0,0}{3,3}{6,0};
            \node [dot, label={[xshift=-1pt, yshift=-2pt]{\tiny $j$}}] at (0.75,2) {};
            \node at (3.5,0.5) {\textup{...}};
            \draw[underbrace] (2,0) -- (5,0);
            \node[below] at (3.5,-.5) {\tiny$i$};
        \end{tikzpicture}
        &=
        \begin{tikzpicture}[vcenter,scale=0.15]
            \tikzfixsize{(0,0)}{(6,5)}
            \path[string]
            (1,0) edge[->] (1,5)
            (4,0) edge[->] (4,5)
            (6.5,0) edge[->] (6.5,5)
            ;
            \diagcap{0,0}{3.75,3}{7.5,0};
            \node [dot, label={[xshift=0pt, yshift=-2pt]{\tiny $j$}}] at (5.25,2.7) {};
            \node at (2.5,0.5) {\textup{...}};
            \draw[underbrace] (1,0) -- (4,0);
            \node[below] at (2.5,-.5) {\tiny$i\!\shortminus\!1$};
        \end{tikzpicture}
        + \sum_{i',j'}
        \begin{tikzpicture}[vcenter,scale=0.15]
            \tikzfixsize{(0,0)}{(6,5)}
            \path[string]
            (1,0) edge[->] (1,5)
            (4,0) edge[->] (4,5)
            (6.5,0) edge[->] (6.5,5)
            ;
            \diagcap{0,0}{3.75,3}{7.5,0};
            \bigbox{$\alpha$}{(-4.5,1.25)}{(-0.5,3.25)};
            \path[string]
            (-4,0) edge (-4,1.25)
            (-4,3.25) edge[-> ] (-4,5)
            (-1,0) edge (-1,1.25)
            (-1,3.25) edge[->] (-1,5)
            ;
            \node at (-2.5,0.5) {\textup{...}};
            \node at (-2.5,4) {\textup{...}};
            \node [dot, label={[xshift=0pt, yshift=-2pt]{\tiny $j'$}}] at (5.25,2.7) {};
            \node at (2.5,0.5) {\textup{...}};
            \draw[underbrace] (1,0) -- (4,0);
            \node[below] at (2.5,-.5) {\tiny$i'$};
        \end{tikzpicture}
        \\
        &=
        \begin{tikzpicture}[vcenter,scale=0.15]
            \tikzfixsize{(0,0)}{(6,5)}
            \path[string]
            (1,0) edge[->] (1,5)
            (4,0) edge[->] (4,5)
            ;
            \diagcap{0,0}{3,3}{6,0};
            \node [dot, label={[xshift=1pt, yshift=-2pt]{\tiny $j$}}] at (5.25,2) {};
            \node at (2.5,0.5) {\textup{...}};
            \draw[underbrace] (1,0) -- (4,0);
            \node[below] at (2.5,-.5) {\tiny$i$};
        \end{tikzpicture}
        + \sum_{j_1 + j_2 = j-1}
        \begin{tikzpicture}[vcenter,scale=0.15]
            \tikzfixsize{(0,0)}{(6,5)}
            \path[string]
            (1,0) edge[->] (1,5)
            (4,0) edge[->] (4,5)
            ;
            \draw[string, ->] (0,0)
            to[out=90,in=180,looseness=1] (3,2.5)
            to[out=0,in=270,looseness=1] (5.5,5)
            ;
            \diagcap{5,0}{6,2}{7,0};
            \node [dot, label={[right, xshift=-1pt, yshift=-2pt]{\tiny $j_1$}}] at (5,3.35) {};
            \node [dot, label={[right, xshift=-1pt, yshift=-1pt]{\tiny $j_2$}}] at (6.75,1.5) {};
            \node at (2.5,0.5) {\textup{...}};
            \draw[underbrace] (1,0) -- (4,0);
            \node[below] at (2.5,-.5) {\tiny$i\!\shortminus\!1$};
        \end{tikzpicture}
        + \sum_{i',j'} \left(
        \begin{tikzpicture}[vcenter,scale=0.15]
            \tikzfixsize{(0,0)}{(6,5)}
            \path[string]
            (1,0) edge[->] (1,5)
            (4,0) edge[->] (4,5)
            ;
            \diagcap{0,0}{3,3}{6,0};
            \bigbox{$\alpha$}{(-4.5,1.25)}{(-0.5,3.25)};
            \path[string]
            (-4,0) edge (-4,1.25)
            (-4,3.25) edge[-> ] (-4,5)
            (-1,0) edge (-1,1.25)
            (-1,3.25) edge[->] (-1,5)
            ;
            \node at (-2.5,0.5) {\textup{...}};
            \node at (-2.5,4) {\textup{...}};
            \node [dot, label={[xshift=1pt, yshift=-2pt]{\tiny $j'$}}] at (5.25,2) {};
            \node at (2.5,0.5) {\textup{...}};
            \draw[underbrace] (1,0) -- (4,0);
            \node[below] at (2.5,-.5) {\tiny$i'\!\shortplus1$};
        \end{tikzpicture}
        + \sum_{j_1' + j_2' = j'-1}
        \begin{tikzpicture}[vcenter,scale=0.15]
            \tikzfixsize{(0,0)}{(6,5)}
            \path[string]
            (1,0) edge[->] (1,5)
            (4,0) edge[->] (4,5)
            ;
            \draw[string, ->] (0,0)
            to[out=90,in=180,looseness=1] (3,2.5)
            to[out=0,in=270,looseness=1] (5.5,5)
            ;
            \diagcap{5,0}{6,2}{7,0};
            \bigbox{$\alpha$}{(-4.5,1.25)}{(-0.5,3.25)};
            \path[string]
            (-4,0) edge (-4,1.25)
            (-4,3.25) edge[-> ] (-4,5)
            (-1,0) edge (-1,1.25)
            (-1,3.25) edge[->] (-1,5)
            ;
            \node at (-2.5,0.5) {\textup{...}};
            \node at (-2.5,4) {\textup{...}};
            \node [dot, label={[right, xshift=-1pt, yshift=-2pt]{\tiny $j_1'$}}] at (5,3.35) {};
            \node [dot, label={[right, xshift=-1pt, yshift=-1pt]{\tiny $j_2'$}}] at (6.75,1.5) {};
            \node at (2.5,0.5) {\textup{...}};
            \draw[underbrace] (1,0) -- (4,0);
            \node[below] at (2.5,-.5) {\tiny$i'$};
        \end{tikzpicture}
        \right)
    \end{align*}
    where $i' < i-1$ and $j' < j$. Momentarily ignoring the summations, the second diagram has $j_2 < j$, the third diagram has $j' < j$ and $i'+1 < i$, and the last diagram has $j_2' < j' < j$. Hence this is of the desired form.
\end{proof}

The functor $E^nF$ admits a natural action of $S_n$ via the braiding on $E^n$.  
Now let $M\in\Cs$, and write $\iota_M$ for the evaluation of $\iota$ on $M$. We consider the following subspaces of $E^nF(M)$:
\[
        W \coloneqq
        \iota^{-1}_M \left(\bigoplus_{i = 0}^{n-1} (E^{n-1}(M)) ^{\oplus k}\right) =
        \left\{m \in E^n F (M) :
        \begin{tikzpicture}[vcenter,scale=0.15]
            \tikzfixsize{(0,0)}{(6,5)}
            \path[string]
            (0,5) edge[->] (6,0)
            (0,0) edge[->] (2,5)
            (1,0) edge[->] (3,5)
            (4,0) edge[->] (6,5)
            ;
            \node at (2.75,0.5) {...};
            \draw[underbrace] (0,0) -- (4,0);
            \node[below] at (2,-.5) {\tiny$n$};
        \end{tikzpicture}
        (m) = 0
        \right\} \text{ and}
    \]
    \[
    V :=\iota_M^{-1}(FE^n(M))=
    \left\{
        m \in E^nF(M) :
        \begin{tikzpicture}[vcenter,scale=0.15]
            \tikzfixsize{(0,0)}{(7,5)}
            \path[string]
            (2,0) edge[->] (2,5)
            (5,0) edge[->] (5,5)
            (0,0) edge[->] (0,5)
            (-3,0) edge[->] (-3,5)
            ;
            \diagcap{1,0}{4,3}{7,0};
            \node [dot, label={[xshift=1pt, yshift=-2pt]{\tiny $j$}}] at (6.25,2) {};
            \node at (-1.5,0.5) {...};
            \draw[underbrace] (-3,0) -- (0,0);
            \node[below] at (-1.5,-.5) {\tiny$i$};
            \node at (3.5,0.5) {...};
            \draw[underbrace] (2,0) -- (5,0);
            \node[below] at (3.5,-.5) {\tiny$n\!\shortminus\!1\!\shortminus\!i$};
        \end{tikzpicture}
        (m) = 0 \text{ for } 0 \leq i \leq n-1, 0 \leq j \leq k-1
    \right\}.
    \]
Note that the morphism
\begin{tikzpicture}[vcenter,scale=0.15]
    \tikzfixsize{(0,0)}{(6,5)}
    \path[string]
    (0,5) edge[->] (6,0)
    (0,0) edge[->] (2,5)
    (1,0) edge[->] (3,5)
    (4,0) edge[->] (6,5)
    ;
    \node at (2.75,0.5) {...};
    \draw[underbrace] (0,0) -- (4,0);
    \node[below] at (2,-.5) {\tiny$n$};
\end{tikzpicture}
is $S_n$-equivariant, with the $S_n$-action on $FE^n(M)$ coming from the braiding on $E^n$.
In other words, $\iota_M$ restricted to $V$ induces an isomorphism 
$$\begin{tikzpicture}[vcenter,scale=0.15]
    \tikzfixsize{(0,0)}{(6,5)}
    \path[string]
    (0,5) edge[->] (6,0)
    (0,0) edge[->] (2,5)
    (1,0) edge[->] (3,5)
    (4,0) edge[->] (6,5)
    ;
    \node at (2.75,0.5) {...};
    \draw[underbrace] (0,0) -- (4,0);
    \node[below] at (2,-.5) {\tiny$n$};
\end{tikzpicture}:V\xto{\sim} FE^n(M)$$
of $S_n$-modules.
A straightforward check shows that $W$ is also $S_n$-stable. Thus 
$$E^nF(M)=V\oplus W$$ is a decomposition of $S_n$-modules.

The difficult component of the proof of Theorem \ref{thm OTI categorical commutation} is the following lemma.

\begin{lemma}\label{lemma cat action structural}
    \begin{enumerate}
        \item If $k\geq0$, $\iota_M$ induces an isomorphism of $S_n$-modules 
        $$E^nF(M)\cong FE^n(M)\oplus \Ind_{S_{n-1}}^{S_n}(E^{n-1}(M))^{\oplus k},$$ where $S_n$ acts on $FE^n$ and $E^nF$ via the braiding, and $S_{n-1}$ acts on $E^{n-1}$ via the braiding.
        \item Similarly if $k<0$, we have a natural isomorphism of $S_n$-modules 
        $$FE^n(M)\cong E^nF(M)\oplus \Ind_{S_{n-1}}^{S_n}(E^{n-1}(M))^{\oplus(-k)}.$$ 
    \end{enumerate}
\end{lemma}

\begin{proof}
We prove this for central charges $k\geq0$, with the case of $k<0$ being similar.  Consider the following subspace of $E^nF(M)$:
\[
    \tilde{W}_0 \coloneqq
    \left\{
        m \in E^nF(M) :
        \begin{tikzpicture}[vcenter,scale=0.15]
            \tikzfixsize{(0,0)}{(7,5)}
            \path[string]
            (2,0) edge[->] (2,5)
            (5,0) edge[->] (5,5)
            (0,0) edge[->] (0,5)
            (-3,0) edge[->] (-3,5)
            ;
            \diagcap{1,0}{4,3}{7,0};
            \node [dot, label={[xshift=1pt, yshift=-2pt]{\tiny $j$}}] at (6.25,2) {};
            \node at (-1.5,0.5) {...};
            \draw[underbrace] (-3,0) -- (0,0);
            \node[below] at (-1.5,-.5) {\tiny$i$};
            \node at (3.5,0.5) {...};
            \draw[underbrace] (2,0) -- (5,0);
            \node[below] at (3.5,-.5) {\tiny$n\!\shortminus\!1\!\shortminus\!i$};
        \end{tikzpicture}
        (m) = 0 \text{ for } 0 < i \leq n-1,\  0 \leq j \leq k-1
    \right\}.
    \]
    We observe that it is an $S_{n-1}$-stable subspace of $E^n F(M)$, where the $S_{n-1}$-action is permutation of all but the first upwards strand. In particular, the intersection $W_0 \coloneqq W \cap \tilde{W}_0$ is also $S_{n-1}$-stable.
      
    By Lemma \ref{claim:move-dot-right}, the components of $\iota$, other than the first, can then be written as
    \begin{eqnarray*}
    \begin{tikzpicture}[vcenter,scale=0.15]
        \tikzfixsize{(0,0)}{(7,5)}
        \path[string]
        (3,0) edge[->] (3,5)
        (6,0) edge[->] (6,5)
        (0,0) edge[->] (0,5)
        (-3,0) edge[->] (-3,5)
        ;
        \diagcap{1,0}{4,3}{7,0};
        \node [dot, label={[xshift=-1pt, yshift=-2pt]{\tiny $j$}}] at (1.75,2) {};
        \node at (-1.5,0.5) {...};
        \draw[underbrace] (-3,0) -- (0,0);
        \node[below] at (-1.5,-.5) {\tiny$i$};
        \node at (4.5,0.5) {...};
        \draw[underbrace] (3,0) -- (6,0);
        \node[below] at (4.5,-.5) {\tiny$n\!\shortminus\!1\!\shortminus\!i$};
    \end{tikzpicture}
    & = &
    \begin{tikzpicture}[vcenter,scale=0.15]
        \tikzfixsize{(0,0)}{(7,5)}
        \path[string]
        (2,0) edge[->] (2,5)
        (5,0) edge[->] (5,5)
        (0,0) edge[->] (0,5)
        (-3,0) edge[->] (-3,5)
        ;
        \diagcap{1,0}{4,3}{7,0};
        \node [dot, label={[xshift=1pt, yshift=-2pt]{\tiny $j$}}] at (6.25,2) {};
        \node at (-1.5,0.5) {...};
        \draw[underbrace] (-3,0) -- (0,0);
        \node[below] at (-1.5,-.5) {\tiny$i$};
        \node at (3.5,0.5) {...};
        \draw[underbrace] (2,0) -- (5,0);
        \node[below] at (3.5,-.5) {\tiny$n\!\shortminus\!1\!\shortminus\!i$};
    \end{tikzpicture}
    + \sum_{i',j'}
    \begin{tikzpicture}[vcenter,scale=0.15]
        \tikzfixsize{(0,0)}{(6,5)}
        \path[string]
        (1,0) edge[->] (1,5)
        (4,0) edge[->] (4,5)
        ;
        \diagcap{0,0}{3,3}{6,0};
        \bigbox{$\alpha$}{(-4.5,1.25)}{(-0.5,3.25)};
        \path[string]
        (-4,0) edge (-4,1.25)
        (-4,3.25) edge[-> ] (-4,5)
        (-1,0) edge (-1,1.25)
        (-1,3.25) edge[->] (-1,5)
        ;
        \node at (-2.5,0.5) {\textup{...}};
        \node at (-2.5,4) {\textup{...}};
        \node [dot, label={[xshift=1pt, yshift=-2pt]{\tiny $j'$}}] at (5.25,2) {};
        \node at (2.5,0.5) {\textup{...}};
        \draw[underbrace] (1,0) -- (4,0);
        \node[below] at (2.5,-.5) {\tiny$i'$};
    \end{tikzpicture}
    \end{eqnarray*}
    where $i',j'$ run over some values such that $i' < n-1-i$ and $j' < j$. Noting that the number of upwards strands in $\alpha$ is strictly greater than $i$, it is clear that $\iota_M$ maps $W_0$ isomorphically onto the $i = 0$ component of the sum $\bigoplus_{i = 0}^{n-1} (E^{n-1}(M)) ^{\oplus k}$, with matrix:
    \begin{equation}\label{eq:iota-for-W0}
    \left(0 , ... , 0 ,
    \begin{tikzpicture}[vcenter,scale=0.15]
        \tikzfixsize{(0,0)}{(6,5)}
        \path[string]
        (1,0) edge[->] (1,5)
        (4,0) edge[->] (4,5)
        ;
        \diagcap{0,0}{3,3}{6,0};
        %
        %
        \node at (2.5,0.5) {\textup{...}};
        \draw[underbrace] (1,0) -- (4,0);
        \node[below] at (2.5,-.5) {\tiny$n\!\shortminus\!1$};
    \end{tikzpicture}
    \ , \
    \begin{tikzpicture}[vcenter,scale=0.15]
        \tikzfixsize{(0,0)}{(6,5)}
        \path[string]
        (1,0) edge[->] (1,5)
        (4,0) edge[->] (4,5)
        ;
        \diagcap{0,0}{3,3}{6,0};
        \node [dot] at (5.25,2) {};
        \node at (2.5,0.5) {\textup{...}};
        \draw[underbrace] (1,0) -- (4,0);
        \node[below] at (2.5,-.5) {\tiny$n\!\shortminus\!1$};
    \end{tikzpicture}
    \ , ... , \
    \begin{tikzpicture}[vcenter,scale=0.15]
        \tikzfixsize{(0,0)}{(6,5)}
        \path[string]
        (1,0) edge[->] (1,5)
        (4,0) edge[->] (4,5)
        ;
        \diagcap{0,0}{3,3}{6,0};
        \node [dot, label={[xshift=3pt, yshift=-1.5pt]{\tiny $k\!\shortminus\!1$}}] at (5.25,2) {};
        \node at (2.5,0.5) {\textup{...}};
        \draw[underbrace] (1,0) -- (4,0);
        \node[below] at (2.5,-.5) {\tiny$n\!\shortminus\!1$};
    \end{tikzpicture}
    \ , 0 , ..., 0
    \right)^\top.
    \end{equation}
    The $S_{n-1}$-action on $E^nF$, permuting the last $n-1$ upward strands, commutes with this matrix. Hence $\iota_M: W_0\to (E^{n-1}(M)) ^{\oplus k}$ is an isomorphism of $S_{n-1}$-modules, with the action on $E^{n-1}(M)$ coming from the braiding. 
    We have now reduced the proof to showing that 
    \[\Ind_{S_{n-1}}^{S_n}W_0\to W \xto{\iota_M} \bigoplus_{i=0}^{n-1} (E^{n-1}(M))^{\oplus k}\]
    is an isomorphism, where the first morphism comes from the universal property of induced modules.
    We consider the following coset representatives of $S_n / S_{n-1}$ which act on the $E$-coloured strands of $E^n F$ by the crossing diagrams
    \begin{equation}\label{eq:Sn-1/Sn-coset-reps}
    \begin{tikzpicture}[vcenter=-1pt,scale=0.15]
        \path[string]
        (3,0) edge[->] (3,3)
        (6,0) edge[->] (6,3)
        ;
        \node at (4.5,1.5) {...};
    \end{tikzpicture}
    \ , \
    \begin{tikzpicture}[vcenter=-1pt,scale=0.15]
        \path[string]
        (0,0) edge[->] (2,3)
        (2,0) edge[->] (0,3)
        (3,0) edge[->] (3,3)
        (6,0) edge[->] (6,3)
        ;
        \node at (4.5,1.5) {...};
    \end{tikzpicture}
    \ , \
    \begin{tikzpicture}[vcenter=-1pt,scale=0.15]
        \path[string]
        (0,0) edge[->] (3,3)
        (2,0) edge[->] (0,3)
        (3,0) edge[->] (1,3)
        (4,0) edge[->] (4,3)
        (7,0) edge[->] (7,3)
        ;
        \node at (5.5,1.5) {...};
    \end{tikzpicture}
    \ , \
    \begin{tikzpicture}[vcenter=-1pt,scale=0.15]
        \path[string]
        (0,0) edge[->] (4,3)
        (2,0) edge[->] (0,3)
        (3,0) edge[->] (1,3)
        (4,0) edge[->] (2,3)
        (5,0) edge[->] (5,3)
        (8,0) edge[->] (8,3)
        ;
        \node at (6.5,1.5) {...};
    \end{tikzpicture}
    \ , ..., \
    \begin{tikzpicture}[vcenter=-1pt,scale=0.15]
        \path[string]
        (0,0) edge[->] (5,3)
        (2,0) edge[->] (0,3)
        (5,0) edge[->] (3,3)
        ;
        \node at (3.15,0.5) {...};
    \end{tikzpicture}\ .
    \end{equation}
    We denote these elements by $t_0, t_1, t_2, t_3, ..., t_{n-1}$ respectively. 
    We will write
    \[
    c_{i,j} \coloneqq
    \begin{tikzpicture}[vcenter,scale=0.15]
        \tikzfixsize{(0,0)}{(7,5)}
        \path[string]
        (2,0) edge[->] (2,5)
        (5,0) edge[->] (5,5)
        (0,0) edge[->] (0,5)
        (-3,0) edge[->] (-3,5)
        ;
        \diagcap{1,0}{4,3}{7,0};
        \node [dot, label={[xshift=1pt, yshift=-2pt]{\tiny $j$}}] at (6.25,2) {};
        \node at (-1.5,0.5) {...};
        \draw[underbrace] (-3,0) -- (0,0);
        \node[below] at (-1.5,-.5) {\tiny$i$};
        \node at (3.5,0.5) {...};
        \draw[underbrace] (2,0) -- (5,0);
        \node[below] at (3.5,-.5) {\tiny$n\!\shortminus\!1\!\shortminus\!i$};
    \end{tikzpicture}
    \]
    for the leading terms in the latter components of $\iota$. Since the dot is on the far right, it is clear that $c_{\ell,j} \circ t_\ell = c_{0,j}$ for any $j$. 
    By writing the components of $\iota$ as in Lemma \ref{claim:move-dot-right}, we see that the $i = \ell$ component of $\iota_M \circ t_\ell$ on $W_0$ is given by the matrix $(c_{0,0},c_{0,1},\ldots,c_{0,k-1})$.  Comparing to~(\ref{eq:iota-for-W0}), we know this is an isomorphism.
    Since the $i$th component of $\iota_M \circ t_\ell(W_0)$ is zero whenever $i > \ell$,
    we have indeed shown that the required map is an isomorphism:
    \[
    \Ind_{S_{n-1}}^{S_n} W_0 = \bigoplus_{\ell=0}^{n-1} t_\ell (W_0) \xto{\sim} \bigoplus_{i=0}^{n-1} (E^{n-1}(M))^{\oplus k}.
    \]
\end{proof}
 
This last result shows that $E^n$ behaves as desired, up to morphisms annihilated by $\varphi$.
\begin{prop}\label{prop:commutes-modulo-traces}
    There exist diagrams $f, f'$ and $f''$ such that
    \begin{enumerate}
        \item \phantom{.}\vspace*{-1.5em} 
            \begin{gather*}
                \begin{tikzpicture}[vcenter=0.1em, scale=0.2]
                    \tikzfixsize{(0,0)}{(3.25,3)}
                    \path[string]
                    (0,0) edge[-{>[length=0.3em]},ultra thick] (3,3)
                    (3,0) edge[->] (0,3)
                    ;
                    \node [dot] at (0.75,2.25) {};
                    \node [below] at (0,0.25) {\tiny $n$};
                \end{tikzpicture}
                =
                \begin{tikzpicture}[vcenter=0.1em, scale=0.2]
                    \path[string]
                    (0,0) edge[-{>[length=0.3em]},ultra thick] (3,3)
                    (3,0) edge[->] (0,3)
                    ;
                    \node [dot] at (2.25,0.75) {};
                    \node [below] at (0,0.25) {\tiny $n$};
                \end{tikzpicture}
                + \op{Tr}_{S_{n-1},S_n}(f),
            \end{gather*}
        \item \phantom{.}\vspace*{-1.5em}
            \begin{gather*}
                \begin{tikzpicture}[vcenter=0.1em, scale=0.2]
                    \path[string]
                    (0,0) edge[<-,out=40, in=270] (2.25,2)
                    (2.25,2) edge[out=90, in=320] (0,4)
                    (3,0) edge[ultra thick,out=140, in=270] (0.75,2)
                    (0.75,2) edge[-{>[length=0.3em]},ultra thick,out=90, in=220] (3,4)
                    ;
                    \node [below] at (3,0.25) {\tiny $n$};
                \end{tikzpicture}
                =
                \begin{tikzpicture}[vcenter=0.1em, scale=0.2]
                    \path[string]
                    (0,0) edge[<-] (0,4)
                    (1.5,0) edge[-{>[length=0.3em]}, ultra thick] (1.5,4) 
                    ;
                    \node [below] at (1.5,0.25) {\tiny $n$};
                \end{tikzpicture}
                + \op{Tr}_{S_{n-1},S_n}(f'),
            \end{gather*}
        \item \phantom{.}\vspace*{-1.5em}
            \begin{gather*}
                \begin{tikzpicture}[vcenter=0.1em, scale=0.2]
                    \path[string]
                    (0,0) edge[ultra thick, out=40, in=270] (2.25,2)
                    (2.25,2) edge[-{>[length=0.3em]}, ultra thick, out=90, in=320] (0,4)
                    (3,0) edge[<-, out=140, in=270] (0.75,2)
                    (0.75,2) edge[out=90, in=220] (3,4)
                    ;
                    \node [below] at (0,0.25) {\tiny $n$};
                \end{tikzpicture}
                =
                \begin{tikzpicture}[vcenter=0.1em, scale=0.2]
                    \path[string]
                    (0,0) edge[-{>[length=0.3em]}, ultra thick] (0,4)
                    (1.5,0) edge[<-] (1.5,4) 
                    ;
                    \node [below] at (0,0.25) {\tiny $n$};
                \end{tikzpicture}
                  + \op{Tr}_{S_{n-1},S_n}(f'').
             \end{gather*}
    \end{enumerate}
\end{prop}
\begin{proof}
Let us write $t_0,\dots,t_{n-1}\in S_n$ for coset representatives of $S_n/S_{n-1}$, as in~(\ref{eq:Sn-1/Sn-coset-reps}). We consider $S_n$ acting on the thick $n$-strand on the top and the bottom. Then for (1) we have:  \begin{align*}
        \begin{tikzpicture}[vcenter=0.1em, scale=0.2]
            \tikzfixsize{(0,0)}{(3.25,3)}
            \path[string]
            (0,0) edge[-{>[length=0.3em]},ultra thick] (3,3)
            (3,0) edge[->] (0,3)
            ;
            \node [dot] at (0.75,2.25) {};
            \node [below] at (0,0.25) {\tiny $n$};
        \end{tikzpicture}
        &=
        \begin{tikzpicture}[vcenter=0.2em,scale=0.15]
            \tikzfixsize{(0,0)}{(6,5)}
            \path[string]
            (-1,5) edge[<-] (6,0)
            (0,0) edge[->] (2,5)
            (1,0) edge[->] (3,5)
            (4,0) edge[->] (6,5)
            ;
            \node [dot] at (0.25,4.1) {};
            \node at (2.75,0.5) {...};
            \draw[underbrace] (0,0) -- (4,0);
            \node[below] at (2,-.5) {\tiny$n$};
        \end{tikzpicture}
        \\
        &=
        \begin{tikzpicture}[vcenter=0.2em,scale=0.15]
            \tikzfixsize{(0,0)}{(6,5)}
            \path[string]
            (0,5) edge[<-] (7,0)
            (0,0) edge[->] (2,5)
            (1,0) edge[->] (3,5)
            (4,0) edge[->] (6,5)
            ;
            \node [dot] at (5.6,1) {};
            \node at (2.75,0.5) {...};
            \draw[underbrace] (0,0) -- (4,0);
            \node[below] at (2,-.5) {\tiny$n$};
        \end{tikzpicture}
        +
        \sum_{\ell=0}^{n-1}
        \begin{tikzpicture}[vcenter=0.2em,scale=0.15]
            \path[string]
            (0,0) edge[->] (1,5)
            (3,0) edge[->] (4,5)
            (6,0) edge[->] (7,5)
            (9,0) edge[->] (10,5)
            ;
            \path[string]
            (5,0) edge[->] (-1,5)
            (11,0) edge[->] (5,5)
            ;
            \node at (1.6,0.5) {...};
            \draw[underbrace] (0,0) -- (3,0);
            \node[below] at (1.5,-.5) {\tiny$\ell$};
            \node at (7.6,0.5) {...};
            \draw[underbrace] (6,0) -- (9,0);
            \node[below] at (7.5,-.5) {\tiny$n\!\shortminus\!1\!\shortminus\!\ell$};
        \end{tikzpicture}
        \\
        &=
        \begin{tikzpicture}[vcenter=0.1em, scale=0.2]
            \path[string]
            (0,0) edge[-{>[length=0.3em]},ultra thick] (3,3)
            (3,0) edge[->] (0,3)
            ;
            \node [dot] at (2.25,0.75) {};
            \node [below] at (0,0.25) {\tiny $n$};
        \end{tikzpicture}
        +
        \sum_{t_\ell \in S_n/S_{n-1}}
        t_\ell \circ
        \begin{tikzpicture}[vcenter=0.2em,scale=0.15]
            \path[string]
            (6,0) edge[->] (7,5)
            (9,0) edge[->] (10,5)
            ;
            \path[string]
            (4,0) edge[->] (4,5)
            (11,0) edge[->] (5,5)
            ;
            \node at (7.6,0.5) {...};
            \draw[underbrace] (6,0) -- (9,0);
            \node[below] at (7.5,-.5) {\tiny$n\!\shortminus\!1$};
        \end{tikzpicture}
        \circ t_\ell^{-1}
        \\
        &=
        \begin{tikzpicture}[vcenter=0.1em, scale=0.2]
            \path[string]
            (0,0) edge[-{>[length=0.3em]},ultra thick] (3,3)
            (3,0) edge[->] (0,3)
            ;
            \node [dot] at (2.25,0.75) {};
            \node [below] at (0,0.25) {\tiny $n$};
        \end{tikzpicture}
        + \op{Tr}_{S_{n-1},S_n}\bigg(
        \begin{tikzpicture}[vcenter=0.2em,scale=0.15]
            \path[string]
            (6,0) edge[->] (7,5)
            (9,0) edge[->] (10,5)
            ;
            \path[string]
            (4,0) edge[->] (4,5)
            (11,0) edge[->] (5,5)
            ;
            \node at (7.6,0.5) {...};
            \draw[underbrace] (6,0) -- (9,0);
            \node[below] at (7.5,-.5) {\tiny$n\!\shortminus\!1$};
        \end{tikzpicture}
        \bigg),
    \end{align*}
    where the second equality follows from repeated application of Equation (\ref{eq:heisenberg-diag-rel}), and the third equality follows from noticing that conjugation by $t_\ell$ moves $\ell$ crossing strands from the right to the left.

    We prove (2) and (3) for the case of central charge $k\geq0$, with the case of $k<0$ being similar.  We begin by recalling that a two-sided inverse to (\ref{eqn magic iso k>=0}) has the following form:
    \[
\left(\begin{tikzpicture}[vcenter=-0.2em, scale=0.2]
            \path[string]
            (0,0) edge[<-] (3,3)
            (3,0) edge[->] (0,3)
            ;
        \end{tikzpicture}
        \ ,
        g_0, \hdots, g_{k-1}\right),
    \]
    for some morphisms $g_0,\dots,g_{k-1}$ (see Sec.~3 of \cite{Brun}).  In particular, we may write a two-sided inverse to $\iota:E^n F \xto{\sim} F E^n \oplus\bigoplus_{i = 0}^{n-1} (E^{n-1}) ^{\oplus k}$, and it will take the following form:
    \[
    \left(
    \begin{tikzpicture}[vcenter,scale=0.15]
        \tikzfixsize{(0,0)}{(6,5)}
        \path[string]
        (6,5) edge[->] (0,0)
        (2,0) edge[->] (0,5)
        (3,0) edge[->] (1,5)
        (6,0) edge[->] (4,5)
        ;
        \node at (4.25,0.5) {...};
        \draw[underbrace] (2,0) -- (6,0);
        \node[below] at (4,-.5) {\tiny$n$};
    \end{tikzpicture}
    \ ,
    \dots \right).
    \]
    By Lemma \ref{lemma cat action structural}, the identity map
    \[
    E^nF\xto{\iota} FE^n\oplus \bigoplus_{i = 0}^{n-1} (E^{n-1}) ^{\oplus k}\xto{\iota^{-1}}E^nF
    \]
    is given by
    \[
              \begin{tikzpicture}[vcenter=0.1em, scale=0.2]
                    \path[string]
                    (0,0) edge[ultra thick, out=40, in=270] (2.25,2)
                    (2.25,2) edge[-{>[length=0.3em]}, ultra thick, out=90, in=320] (0,4)
                    (3,0) edge[<-, out=140, in=270] (0.75,2)
                    (0.75,2) edge[out=90, in=220] (3,4)
                    ;
                    \node [below] at (0,0.25) {\tiny $n$};
                \end{tikzpicture} +\text{ a morphism factoring through }\Ind_{S_{n-1}}^{S_n}(E^{n-1})^{\oplus k}.
    \]
    which gives (3) via Lemma \ref{lem:phi-kills-quotient-orbit}. A similar argument gives a proof of (2).
\end{proof}

\begin{proof}[Proof of Theorem \ref{thm OTI categorical commutation}]\label{proof thm cat commutation}
Let
\begin{gather*}
    \zeta_E =
    \begin{tikzpicture}[vcenter=0.2em,scale=0.2]
        \path[string]
        (0,4) edge[<-] (3,0)
        (0,0) edge[densely dashed] (2,4)
        (1,0) edge[-{>[length=0.3em]},ultra thick] (3,4)
        ;
        \node [below] at (0,0) {\tiny $\varphi$};
        \node [below] at (1.5,0.3) {\tiny $p^r$};
    \end{tikzpicture}
    \quad \text{and} \quad
    \zeta_F =
    \begin{tikzpicture}[vcenter=0.2em,scale=0.2]
        \path[string]
        (0,4) edge[->] (3,0)
        (0,0) edge[densely dashed] (2,4)
        (1,0) edge[-{>[length=0.3em]},ultra thick] (3,4)
        ;
        \node [below] at (0,0) {\tiny $\varphi$};
        \node [below] at (1.5,0.3) {\tiny $p^r$};
    \end{tikzpicture}
    \ .
\end{gather*}
Then by~Lemma~\ref{lem:dashed-commutes}, $\zeta_E$ and $\zeta_F$ are invertible with
\begin{gather*}
    \zeta_E^{-1} =
    \begin{tikzpicture}[vcenter=0.2em,scale=0.2]
        \path[string]
        (0,0) edge[->] (3,4)
        (2,0) edge[densely dashed] (0,4)
        (3,0) edge[-{>[length=0.3em]},ultra thick] (1,4)
        ;
        \node [below] at (2,0) {\tiny $\varphi$};
        \node [below] at (3.5,0.3) {\tiny $p^r$};
    \end{tikzpicture},
    \quad \text{and} \quad
    \zeta_F^{-1} =
    \begin{tikzpicture}[vcenter=0.2em,scale=0.2]
        \path[string]
        (0,0) edge[<-] (3,4)
        (2,0) edge[densely dashed] (0,4)
        (3,0) edge[-{>[length=0.3em]},ultra thick] (1,4)
        ;
        \node [below] at (2,0) {\tiny $\varphi$};
        \node [below] at (3.5,0.3) {\tiny $p^r$};
    \end{tikzpicture}.
\end{gather*}
The conditions (1) to (4) of Definition \ref{defn categorical commutation} can then be written diagrammatically as
\begin{align*}
    \begin{tikzpicture}[vcenter=0.2em,scale=0.2]
        \path[string]
        (0,4) edge[<-] (4,0)
        (0,0) edge[densely dashed] (2,4)
        (1,0) edge[-{>[length=0.3em]},ultra thick] (3,4)
        ;
        \node [below] at (0,0) {\tiny $\varphi$};
        \node [below] at (1.5,0.3) {\tiny $p^r$};
        \node [dot] at (3,1) {};
    \end{tikzpicture}
    =
    \begin{tikzpicture}[vcenter=0.2em,scale=0.2]
        \path[string]
        (-1,4) edge[<-] (3,0)
        (0,0) edge[densely dashed] (2,4)
        (1,0) edge[-{>[length=0.3em]},ultra thick] (3,4)
        ;
        \node [below] at (0,0) {\tiny $\varphi$};
        \node [below] at (1.5,0.3) {\tiny $p^r$};
        \node [dot] at (0,3) {};
    \end{tikzpicture}
    \quad&,\quad
    \begin{tikzpicture}[vcenter=0.25em, scale=0.2]
        \path[string]
        (-1.625,0) edge[densely dashed] (3.375,5)
        (-0.375,0) edge[-{>[length=0.3em]},ultra thick] (4.625,5)
        (2,0) edge[->, out=45, in=-45, looseness=1.8] (1,5)
        (4,0) edge[->] (-1,5)
        ;
        \node [below] at (-1.625,0) {\tiny $\varphi$};
        \node [below] at (0,0.3) {\tiny $p^r$};
    \end{tikzpicture}
    =
    \begin{tikzpicture}[vcenter=0.25em, scale=0.2]
        \path[string]
        (-1.625,0) edge[densely dashed] (3.375,5)
        (-0.375,0) edge[-{>[length=0.3em]},ultra thick] (4.625,5)
        (2,0) edge[->, out=135, in=225, looseness=1.8] (1,5)
        (4,0) edge[->] (-1,5)
        ;
        \node [below] at (-1.625,0) {\tiny $\varphi$};
        \node [below] at (0,0.3) {\tiny $p^r$};
    \end{tikzpicture}
    \quad,\\
    \begin{tikzpicture}[vcenter=0.25em, scale=0.2]
        \path[string]
        (-1.625,0) edge[densely dashed] (3.375,5)
        (-0.375,0) edge[-{>[length=0.3em]},ultra thick] (4.625,5)
        (3.5,2.5) edge[->] (1,5)
        (2.5,1.5) edge (-1,5)
        (2.5,1.5) edge[out=-45, in=-45, looseness=2] (3.5,2.5)
        ;
        \node [below] at (-1.625,0) {\tiny $\varphi$};
        \node [below] at (0,0.3) {\tiny $p^r$};
    \end{tikzpicture}
    =
    \begin{tikzpicture}[vcenter=0.2em, scale=0.2]
        \path[string]
        (2,0) edge[densely dashed] (2,4)
        (3,0) edge[-{>[length=0.3em]},ultra thick] (3,4)
        ;
        \diagcup{-1,4}{0,2.75}{1,4}
        \node [below] at (2,0) {\tiny $\varphi$};
        \node [below] at (3.5,0.3) {\tiny $p^r$};
    \end{tikzpicture}
    \quad&,\quad
    \begin{tikzpicture}[vcenter=0.25em, scale=0.2]
        \path[string]
        (-1.625,0) edge[densely dashed] (3.375,5)
        (-0.375,0) edge[-{>[length=0.3em]},ultra thick] (4.625,5)
        (3.5,2.5) edge (1,5)
        (2.5,1.5) edge[->] (-1,5)
        (2.5,1.5) edge[out=-45, in=-45, looseness=2] (3.5,2.5)
        ;
        \node [below] at (-1.625,0) {\tiny $\varphi$};
        \node [below] at (0,0.3) {\tiny $p^r$};
    \end{tikzpicture}
    =
    \begin{tikzpicture}[vcenter=0.2em, scale=0.2]
        \path[string]
        (2,0) edge[densely dashed] (2,4)
        (3,0) edge[-{>[length=0.3em]},ultra thick] (3,4)
        ;
        \diagcup{-1,4}{0,2.75}{1,4}[<-]
        \node [below] at (2,0) {\tiny $\varphi$};
        \node [below] at (3.5,0.3) {\tiny $p^r$};
    \end{tikzpicture}  .
\end{align*}

\noindent The second and third equations are direct consequences of (\ref{eq:heisenberg-diag-rel}) and Lemma \ref{lem:dashed-commutes}. For the first equation we have
\begin{align*}
    \begin{tikzpicture}[vcenter=0.2em,scale=0.2]
        \path[string]
        (-1,4) edge[<-] (3,0)
        (0,0) edge[densely dashed] (2,4)
        (1,0) edge[-{>[length=0.3em]},ultra thick] (3,4)
        ;
        \node [below] at (0,0) {\tiny $\varphi$};
        \node [below] at (1.5,0.3) {\tiny $p^r$};
        \node [dot] at (0,3) {};
    \end{tikzpicture}
    =
    \begin{tikzpicture}[vcenter=0.2em,scale=0.2]
        \path[string]
        (0,4) edge[<-] (4,0)
        (0,0) edge[densely dashed] (2,4)
        (2,0) edge[-{>[length=0.3em]},ultra thick] (4,4)
        ;
        \node [below] at (0,0) {\tiny $\varphi$};
        \node [below] at (2.5,0.3) {\tiny $p^r$};
        \node [dot] at (2,2) {};
    \end{tikzpicture}
    =
    \begin{tikzpicture}[vcenter=0.2em,scale=0.2]
        \path[string]
        (0,4) edge[<-] (4,0)
        (0,0) edge[densely dashed] (2,4)
        (1,0) edge[-{>[length=0.3em]},ultra thick] (3,4)
        ;
        \node [below] at (0,0) {\tiny $\varphi$};
        \node [below] at (1.5,0.3) {\tiny $p^r$};
        \node [dot] at (3,1) {};
    \end{tikzpicture}
    +
    \Big(
    \begin{tikzpicture}[vcenter=0.1em,scale=0.2]
        \tikzfixsize{(-0.5,0)}{(3,3)}
        \path[string]
        (0,0) edge[densely dashed] (2,3)
        (2,0) edge[->] (0,3)
        (3,0) edge[-{>[length=0.3em]},ultra thick] (3,3)
        ;
        \node [below] at (0,0) {\tiny $\varphi$};
        \node [below] at (3.25,0.3) {\tiny $p^r$};
    \end{tikzpicture}
    \Big)
    \circ_v
    \Big(
    \begin{tikzpicture}[vcenter=-0.25em,scale=0.2]
        \tikzfixsize{(-0.5,0)}{(0.5,3)}
        \path[string] (0,0) edge[densely dashed] (0,3);
        \node [below] at (0,0) {\tiny $\varphi$};
    \end{tikzpicture}
    \op{Tr}_{S_{p^r-1},S_{p^r}}(f)
    \Big)
    =
    \begin{tikzpicture}[vcenter=0.2em,scale=0.2]
        \path[string]
        (0,4) edge[<-] (4,0)
        (0,0) edge[densely dashed] (2,4)
        (1,0) edge[-{>[length=0.3em]},ultra thick] (3,4)
        ;
        \node [below] at (0,0) {\tiny $\varphi$};
        \node [below] at (1.5,0.3) {\tiny $p^r$};
        \node [dot] at (3,1) {};
    \end{tikzpicture},
\end{align*}
and for the fourth equation we have
\begin{align*}
    \begin{tikzpicture}[vcenter=0.25em, scale=0.2]
        \path[string]
        (-1.625,0) edge[densely dashed] (3.375,5)
        (-0.375,0) edge[-{>[length=0.3em]},ultra thick] (4.625,5)
        (3.5,2.5) edge (1,5)
        (2.5,1.5) edge[->] (-1,5)
        (2.5,1.5) edge[out=-45, in=-45, looseness=2] (3.5,2.5)
        ;
        \node [below] at (-1.625,0) {\tiny $\varphi$};
        \node [below] at (0,0.3) {\tiny $p^r$};
    \end{tikzpicture}
    =
    \begin{tikzpicture}[vcenter=0.1em, scale=0.2]
        \path[string]
        (0,0) edge[densely dashed] (4,4)
        (3.5,2.5) edge (2,4)
        (2.5,1.5) edge[->] (0,4)
        (2.5,1.5) edge[out=-45, in=-45, looseness=2] (3.5,2.5)
        (5.25,0) edge[-{>[length=0.3em]},ultra thick] (5.25,4)
        ;
        \node [below] at (0,0) {\tiny $\varphi$};
        \node [below] at (5.5,0.3) {\tiny $p^r$};
    \end{tikzpicture}
    +
    \Big(
    \begin{tikzpicture}[vcenter=0.1em, scale=0.2]
        \tikzfixsize{(-0.5,0)}{(4,3)}
        \path[string]
        (0,0) edge[->] (0,3)
        (1,0) edge[densely dashed] (3,3)
        (3,0) edge[<-] (1,3)
        (4,0) edge[-{>[length=0.3em]},ultra thick] (4,3)
        ;
        \node [below] at (1,0) {\tiny $\varphi$};
        \node [below] at (4.25,0.3) {\tiny $p^r$};
    \end{tikzpicture}
    \Big)
    \circ_v
    \Big(
    \begin{tikzpicture}[vcenter=-0.25em,scale=0.2]
        \tikzfixsize{(-0.5,0)}{(1.5,3)}
        \path[string]
        (0,0) edge[->] (0,3)
        (1,0) edge[densely dashed] (1,3)
        ;
        \node [below] at (1,0) {\tiny $\varphi$};
    \end{tikzpicture}
    \op{Tr}_{S_{p^r-1},S_{p^r}}(f')
    \Big)
    \circ_v
    \Big(
    \begin{tikzpicture}[vcenter=0.1em,scale=0.2]
        \tikzfixsize{(-0.5,0)}{(3.5,3)}
        \path[string]
        (1,0) edge[densely dashed] (1,3)
        (3,0) edge[-{>[length=0.3em]},ultra thick] (3,3)
        ;
        \diagcup{0,3}{1,1.5}{2,3}[<-]
        \node [below] at (1,0) {\tiny $\varphi$};
        \node [below] at (3.25,0.3) {\tiny $p^r$};
    \end{tikzpicture}
    \Big)
    =
    \begin{tikzpicture}[vcenter=0.2em, scale=0.2]
        \path[string]
        (2,0) edge[densely dashed] (2,4)
        (3,0) edge[-{>[length=0.3em]},ultra thick] (3,4)
        ;
        \diagcup{-1,4}{0,2.75}{1,4}[<-]
        \node [below] at (2,0) {\tiny $\varphi$};
        \node [below] at (3.5,0.3) {\tiny $p^r$};
    \end{tikzpicture},
\end{align*}
where $f,f'$ are as in Proposition \ref{prop:commutes-modulo-traces} and the final equalities apply (\ref{eq:dotted-kills-trace}).  Here $\circ_v$ denotes vertical composition, i.e. $f \circ_v g$ is given by stacking diagram $f$ above~$g$.
\end{proof}

\bibliographystyle{amsalpha}
	
\end{document}